\documentclass[10pt]{amsart}
\usepackage[utf8]{inputenc}

\usepackage{amsmath}
\usepackage{amsfonts}
\usepackage{amssymb}
\usepackage{amsthm}
\usepackage{mathrsfs}
\usepackage{dsfont}
\usepackage{chngcntr}
\usepackage{mathtools}
\usepackage[all,cmtip]{xy}
\usepackage{tikz}
\usepackage[colorlinks]{hyperref}
\hypersetup{linkcolor=blue,citecolor=green}
\usepackage{cite}
\usepackage{fancyhdr}
\usepackage{stmaryrd}
\usepackage{yfonts}
\usepackage{adjustbox}

\tikzset{double line with arrow/.style args={#1,#2}{decorate,decoration={markings,%
mark=at position 0 with {\coordinate (ta-base-1) at (0,1pt);\coordinate (ta-base-2) at (0,-1pt);},
mark=at position 1 with {\draw[#1] (ta-base-1) -- (0,1pt);
\draw[#2] (ta-base-2) -- (0,-1pt);
}}}}

\font\wncyr=wncyr9.8
\newcommand{\sha}{\text{\wncyr{W}}}

\usepackage[top=2.6cm, bottom=2.4cm, left=2.4cm, right=2.4cm]{geometry}

\usepackage{tikz}
\usepackage{tikz-cd}

\numberwithin{equation}{section}

\newtheorem{theorem}{Theorem}
\newtheorem{lemma}[theorem]{Lemma}
\newtheorem{prop}[theorem]{Proposition}
\newtheorem{cor}[theorem]{Corollary}
\newtheorem{conj}[theorem]{Conjecture}

\theoremstyle{definition}
\newtheorem{definition}[theorem]{Definition}
\newtheorem{remark}[theorem]{Remark}

\theoremstyle{remark}
\newtheorem*{rem}{Remark}

\newcommand{\Z}{\mathbb{Z}}

\DeclareRobustCommand{\qed}{%
  \ifmmode
    \eqno \def\@badmath{$$}
    \let\eqno\relax \let\leqno\relax \let\veqno\relax
    \hbox{\openbox}%
  \else
    \leavevmode\unskip\penalty9999 \hbox{}\nobreak\hfill
    \quad\hbox{\openbox}%
  \fi
}

\renewcommand{\thetheorem}{\arabic{section}.\arabic{subsection}.\arabic{theorem}}

\newcommand{\cL}{\mathscr{L}}
\newcommand{\cN}{\mathscr{N}}

\newcommand{\Lcal}{\mathcal{L}}

\newcommand{\Fcal}{\mathcal{F}}
\newcommand{\cF}{\mathcal{F}}
\newcommand{\bZ}{\mathbb{Z}}
\newcommand{\bQ}{\mathbb{Q}}
\newcommand{\Q}{\mathbb{Q}}
\newcommand{\C}{\mathbb{C}}
\newcommand{\rH}{\mathrm{H}}

\newcommand{\fm}{\mathfrak{m}}
\DeclareMathOperator{\Frob}{Frob}

\DeclareMathOperator{\Hom}{Hom}

\DeclareMathOperator{\loc}{loc}
\DeclareMathOperator{\ord}{ord}

\DeclareMathOperator{\tr}{tr}

\DeclareMathOperator{\GL}{GL}

\DeclareMathOperator{\Gal}{Gal}
\DeclareMathOperator{\Aut}{Aut}
\DeclareMathOperator{\End}{End}

\newcommand{\CF}{\mathcal{F}}
\newcommand{\CL}{\mathscr{L}}
\newcommand{\CN}{\mathscr{N}}
\newcommand{\kap}{\kappa}
\newcommand{\ind}{\mathrm{ind}}
\newcommand{\rank}{\mathrm{rank}}

\makeatletter
\@addtoreset{theorem}{subsection}
\makeatother

\begin{document}

\title[On the anticyclotomic
Iwasawa theory of rational elliptic curves at Eisenstein primes]{On the anticyclotomic Iwasawa theory of rational elliptic curves at Eisenstein primes}

\author{Francesc Castella}
\address[F.~Castella]{University of California Santa Barbara, South Hall, Santa Barbara, CA 93106, USA}
\email{castella@ucsb.edu}

\author{Giada Grossi}
\address[G.~Grossi]{Institut Galilée, Université Sorbonne Paris Nord, 93430 Villetaneuse, FRANCE}
\email{grossi@math.univ-paris13.fr}
 
\author{Jaehoon Lee}  
\address[J.~Lee]{KAIST, 291 Daehak-ro, Yuseong-gu, Daejeon 34141, Republic of Korea}
\email{jaehoon.lee900907@gmail.com}
 
\author{Christopher Skinner} 
\address[C.~Skinner]{Princeton University, Fine Hall, Washington Road, Princeton, NJ 08544-1000, USA}
\email{cmcls@princeton.edu}

\date{\today}

\begin{abstract}
Let $E/\mathbb{Q}$ be an elliptic curve and $p$ an odd prime where $E$ has good reduction, and assume that $E$ admits a rational $p$-isogeny. In this paper we study the anticyclotomic Iwasawa theory of $E$ over an imaginary quadratic field in which $p$ splits, which we relate to the anticyclotomic Iwasawa theory of characters 
by a variation of the method of Greenberg--Vatsal. As a result of our study we obtain proofs (under relatively mild hypotheses) of Perrin-Riou's Heegner point main conjecture, a $p$-converse to the theorem of Gross--Zagier and Kolyvagin, and the $p$-part of the Birch--Swinnerton-Dyer formula in analytic rank $1$, for Eisenstein primes $p$. 
\end{abstract}
\maketitle

\tableofcontents

\section*{Introduction}
\addtocontents{toc}{\protect\setcounter{tocdepth}{0}}
\addtocontents{lof}{\protect\setcounter{tocdepth}{0}}
\renewcommand{\thetheorem}{\Alph{theorem}}

\subsection{Statement of the main results}

Let $E/\Q$ be an elliptic curve, and let $p$ be an odd prime of good reduction for $E$. We say that $p$ is an \emph{Eisenstein prime} (for $E$) if $E[p]$ is reducible as a $G_\Q$-module, where $G_\Q={\rm Gal}(\overline{\Q}/\Q)$ is the absolute Galois group of $\Q$ and $E[p]$ denotes the $p$-torsion of $E$. Equivalently, $p$ is an Eisenstein prime if $E$ admits a rational $p$-isogeny. By a result of Fontaine (see \cite{edixhoven-weight} for an account), Eisenstein primes are primes of \emph{ordinary} reduction for $E$, and by Mazur's results \cite{mazur} in fact $p\in\{3, 5, 7, 13, 37\}$. 


Let $p>2$ be an Eisenstein prime for $E$, and let $K$ be an imaginary quadratic field such that 
\begin{equation}\label{eq:intro-spl}
\textrm{$p=v\bar{v}$ splits in $K$,}\tag{spl} 
\end{equation} 
where $v$ denotes the prime of $K$ above $p$ induced by a fixed embedding $\overline{\bQ}\hookrightarrow\overline{\bQ}_p$. 
Denoting by $N$ the conductor of $E$, assume also that $K$ satisfies the following \emph{Heegner hypothesis}:
\begin{equation}\label{eq:intro-Heeg}
\textrm{every prime $\ell\vert N$ splits in $K$.}\tag{Heeg}
\end{equation}

Under these hypotheses, the anticyclotomic Iwasawa main conjecture for $E$ {considered in this paper}
can be formulated in two different guises. We begin by recalling these, since both formulations will play an important role in the proof of our main results. (Note that for the formulation $p$ can be any odd prime of good ordinary reduction for $E$.) Let $\Gamma={\rm Gal}(K_\infty/K)$ be the Galois group of the anticyclotomic $\Z_p$-extension of $K$, and for each $n$ denote by $K_n$ the subfield of $K_\infty$ with $[K_n:K]=p^n$. Set
\[
\Lambda:=\Z_p\llbracket\Gamma\rrbracket,\quad\Lambda_{\rm ac}:=\Lambda\otimes_{\Z_p}\Q_p,\quad\Lambda^{\rm ur}:=\Lambda\hat\otimes_{\Z_p}\Z_p^{\rm ur},
\]
where $\Z_p^{\rm ur}$ is the completion of the ring of integers of the maximal unramified extension of $\Q_p$. 
Following the work of Bertolini--Darmon--Prasanna \cite{BDP}, there is a $p$-adic $L$-function $\mathcal{L}_E\in\Lambda^{\rm ur}$ interpolating the central critical values of {the $L$-function of}
 $f/K$, where $f\in S_2(\Gamma_0(N))$ is the newform associated with $E$, twisted by certain characters of $\Gamma$
of infinite order. For any subfield $L\subset\overline{\Q}$, let ${\rm Sel}_{p^m}(E/L)$ be the Selmer group fitting into the descent exact sequence
\[
0\rightarrow E(L)\otimes_{\Z}\Z/p^m\Z\rightarrow{\rm Sel}_{p^m}(E/L)\rightarrow\sha(E/L)[p^m]\rightarrow 0.
\]
We put ${\rm Sel}_{p^\infty}(E/K_\infty)=\varinjlim_m{\rm Sel}_{p^m}(E/K_\infty)$, 
and let 
\[
\mathfrak{X}_E={\rm Hom}_{\Z_p}(\mathfrak{S}_E,\Q_p/\bZ_p)
\] 
be the Pontryagin dual of {modified} Selmer group $\mathfrak{S}_E$ obtained from ${\rm Sel}_{p^\infty}(E/K_\infty)$ 
by relaxing (resp. imposing triviality) at the places above $v$ (resp. $\bar{v}$).

The following formulation of the anticyclotomic Iwasawa main conjectures for $E$ can be seen as a special case of Greenberg's main conjectures \cite{greenberg-motives}.

\begin{conj}\label{conj:BDP}
Let $E/\mathbb{Q}$ be an elliptic curve and $p>2$ a prime of good ordinary reduction for $E$, and let $K$ be an imaginary quadratic field satisfying {\rm (\ref{eq:intro-Heeg})} and  {\rm (\ref{eq:intro-spl})}. 
Then $\mathfrak{X}_E$ is $\Lambda$-torsion, and
\[
{\rm char}_\Lambda\bigl(\mathfrak{X}_E\bigr)\Lambda^{\rm ur}=\bigl(\Lcal_E\bigr)
\]
as ideals in $\Lambda^{\rm ur}$. 
\end{conj}

A second formulation, originally due to Perrin-Riou \cite{perrinriou}, is in terms of Heegner points. Although a more general formulation is possible (\emph{cf.} \cite[\S{2.3}]{CMpconverse}), here as in \cite{perrinriou} we assume that
\begin{equation}\label{eq:intro-disc}
\textrm{the discriminant $D_K$ of $K$ is odd and $D_K\neq -3$,}\tag{disc}
\end{equation}
and do not require hypothesis (\ref{eq:intro-spl}). Fix a modular parametrization
\[
\pi:X_0(N)\rightarrow E,
\]
and for any subfield $L\subset\overline{\Q}$ let ${\rm Sel}_{p^m}(E/L)$ be the Selmer group fitting into the descent exact sequence
\[
0\rightarrow E(L)\otimes_{\Z}\Z/p^m\Z\rightarrow{\rm Sel}_{p^m}(E/L)\rightarrow\sha(E/L)[p^m]\rightarrow 0.
\]
Via $\pi$, the Kummer images of Heegner points on $X_0(N)$ over ring class fields of $K$ of $p$-power conductor give rise to a class $\kappa_1^{\rm Hg}\in\mathcal{S}$, where
\[
\mathcal{S}=\biggl(\varprojlim_n\varprojlim_m{\rm Sel}_{p^m}(E/K_n)\biggr)\otimes_{}\Q_p.
\] 
{The group $\mathcal{S}$ is naturally a $\Lambda_{\rm ac}$-module, and} the class $\kappa_1^{\rm Hg}$ is known to be non-$\Lambda_{\rm ac}$-torsion by results of Cornut and Vatsal \cite{cornut}, \cite{vatsal}. Denote by $\mathcal{H}\subset\mathcal{S}$ the $\Lambda_{\rm ac}$-submodule generated by $\kappa_1^{\rm Hg}$, and put 
\[
\mathcal{X}={\rm Hom}_{\Z_p}({\rm Sel}_{p^\infty}(E/K_\infty),\Q_p/\Z_p)\otimes\Q_p.
\]


\begin{conj}\label{conj:PR}
Let $E/\Q$ be an elliptic curve and $p>2$ a prime of good ordinary reduction for $E$, and let $K$ be an imaginary quadratic field satisfying {\rm(\ref{eq:intro-Heeg})} and {\rm(\ref{eq:intro-disc})}. Then $\mathcal{S}$ and $\mathcal{X}$ both
have $\Lambda_{\rm ac}$-rank one, and
\[
{\rm char}_{\Lambda_{\rm ac}}(\mathcal{X}_{\rm tors})={\rm char}_{\Lambda_{\rm ac}}\bigl(\mathcal{S}/\mathcal{H}\bigr)^2,
\] 
where $\mathcal{X}_{\rm tors}$ denotes the $\Lambda_{\rm ac}$-torsion submodule of $X$. 
\end{conj}

\begin{rem}
One can naturally formulate an integral version of Conjecture~\ref{conj:PR}, but the results of \cite{perrinriou} and \cite{howard-GZ} show that the terms appearing in the corresponding equality of $\Lambda$-module characteristic ideals are in general \emph{not} invariant under isogenies. (With $p$ inverted, i.e., as ideals in $\Lambda_{\rm ac}$, the terms are invariant under isogenies.) On the other hand, it is clear that the principal ideals in $\Lambda^{\rm ur}$ appearing in the equality of Conjecture~\ref{conj:BDP} depend only on the isogeny class of $E$.
\end{rem}

When $p$ is non-Eisenstein for $E$, Conjectures~\ref{conj:BDP} and \ref{conj:PR} have been studied by several authors \cite{bertolini,howard,howard-gl2-type,wan-heegner,castellaheights,BCK}, but the residually reducible case remained largely unexplored; in particular, unless $E$ has CM by $K$ (a case that is excluded by our hypothesis (\ref{eq:intro-Heeg}), but see \cite{CMpconverse} for this case), there seems to be no previous results towards these conjectures when $p$ is an Eisenstein prime for $E$. 

To state our main results on 
the anticyclotomic Iwasawa theory of $E$ at Eisenstein primes $p$, write
\[
E[p]^{ss}=\mathbb{F}_p(\phi)\oplus\mathbb{F}_p(\psi),
\]
where $\phi,\psi:G_\Q\rightarrow\mathbb{F}_p^\times$ are characters. Note that it follows from the Weil pairing that $\psi=\omega\phi^{-1}$, where $\omega$ is the Teichm\"uller character. Let $G_p\subset G_{\mathbb{Q}}$ be a decomposition group at $p$. 

Our most complete results towards Conjectures \ref{conj:BDP} and \ref{conj:PR} are proved under the additional hypothesis
that
\begin{equation}\label{eq:intro-sel}
\text{the $\Z_p$-corank of $\mathrm{Sel}_{p^\infty}(E/K)$ is $1$.}\tag{Sel}
\end{equation}

\begin{theorem}\label{thm:C}
Let $E/\mathbb{Q}$ be an elliptic curve, $p>2$ an Eisenstein prime for $E$, and $K$ an imaginary quadratic field satisfying {\rm (\ref{eq:intro-Heeg})}, {\rm (\ref{eq:intro-spl})}, {\rm (\ref{eq:intro-disc})}, and
{\rm (\ref{eq:intro-sel})}. 
Assume also that
$\phi\vert_{G_p}\neq\mathds{1},\omega$. Then Conjecture~\ref{conj:BDP} holds.
\end{theorem}

As we explain in more detail in the next section, a key step towards Theorem~\ref{thm:C} is the proof of a divisibility in  Conjecture~\ref{conj:PR} that we establish without the need to assume (\ref{eq:intro-spl}) (see Theorem~\ref{thm:howard-HP}). On the other hand, as first observed in \cite{castella-PhD} and \cite{wan-heegner}, when $p$ splits in $K$, Conjectures~\ref{conj:BDP} and \ref{conj:PR} are essentially equivalent (see Proposition \ref{prop:equiv-imc}). Thus from Theorem~\ref{thm:howard-HP} we deduce one of the divisibilities in Conjecture~\ref{conj:BDP}, which by the analysis of Iwasawa invariants carried out in $\S\S$\ref{sec:algebraic}-\ref{IMCII} then yields the equality of ideals in $\Lambda^{\rm ur}$ predicted by Conjecture~\ref{conj:BDP}. As a result, our analysis together with the aforementioned equivalence also yields the following.

\begin{cor}\label{cor:D}
Let $E/\mathbb{Q}$ be an elliptic curve, $p>2$ an Eisenstein prime for $E$, and $K$ an imaginary quadratic field satisfying {\rm (\ref{eq:intro-Heeg})}, {\rm (\ref{eq:intro-disc})}, and
{\rm (\ref{eq:intro-sel})}. If $E(K)[p]=0$, then $\mathcal{S}$ and $\mathcal{X}$ both
have $\Lambda_{\rm ac}$-rank one, and
\[
{\rm char}_{\Lambda_{\rm ac}}(\mathcal{X}_{\rm tors})\supset{\rm char}_{\Lambda_{\rm ac}}\bigl(\mathcal{S}/\mathcal{H}\bigr)^2.
\] 
Moreover, if in addition $K$ satisfies {\rm (\ref{eq:intro-spl})} and $\phi\vert_{G_p}\neq\mathds{1},\omega$, then
\[
{\rm char}_{\Lambda_{\rm ac}}(\mathcal{X}_{\rm tors})={\rm char}_{\Lambda_{\rm ac}}\bigl(\mathcal{S}/\mathcal{H}\bigr)^2,
\] 
and hence Conjecture~\ref{conj:PR} holds.
\end{cor}


Note that in both Theorem~\ref{thm:C} and Corollary~\ref{cor:D}, the elliptic curve $E$ is allowed to have complex multiplication  (necessarily by an imaginary quadratic field different from $K$). 

With a judicious choice of $K$, Theorem~\ref{thm:C} also has applications to the arithmetic over $\Q$ of rational elliptic curves. Specifically, for Eisenstein prime $p$, we obtain a $p$-converse to the celebrated theorem 
\begin{equation}\label{eq:GZK}
\ord_{s=1}L(E,s)=r \in \{0,1\}\;\Longrightarrow\;
{\rm rank}_\Z E(\Q)=r\;\textrm{and}\;
\#\sha(E/\bQ)<\infty,
\end{equation}
of Gross--Zagier and Kolyvagin. (The case of Eisenstein primes eluded the methods of 
\cite{pCONVskinner} and \cite{wei-zhang}, which require $E[p]$ to be absolutely irreducible as a $G_\bQ$-module.)

\begin{theorem}\label{thm:E}
Let $E/\mathbb{Q}$ be an elliptic curve and $p>2$ an Eisenstein prime for $E$, so that $E[p]^{ss}=\mathbb{F}_p(\phi)\oplus\mathbb{F}_p(\psi)$ as $G_\Q$-modules, and assume that $\phi\vert_{G_p}\neq\mathds{1},\omega$. Let $r\in\{0,1\}$. Then the following implication holds:
\[
{\rm corank}_{\Z_p}{\rm Sel}_{p^\infty}(E/\Q)=r
\;\Longrightarrow\;
{\rm ord}_{s=1}L(E,s)=r,
\]
and so ${\rm rank}_\Z E(\Q)=r$ and $\#\sha(E/\Q)<\infty$.
\end{theorem}

Note that if ${\rm rank}_{\Z}E(\Q)=r$ and $\#\sha(E/\Q)[p^\infty]<\infty$ then ${\rm corank}_{\Z_p}{\rm Sel}_{p^\infty}(E/\Q)=r$,
whence the $p$-converse to \eqref{eq:GZK}. We also note that for $p=3$, Theorem~\ref{thm:E} together with the recent work of Bhargava--Klagsbrun--Lemke Oliver--Shnidman \cite{BKLS} on the average $3$-Selmer rank for abelian varieties in quadratic twist families, 
provides additional evidence
towards Goldfeld's conjecture \cite{goldfeld} for elliptic curves $E/\bQ$ admitting a rational $3$-isogeny (see Corollary~\ref{cor:goldfeld} and Remark~\ref{rem:BKLS}, and see also \cite{kriz-li} for earlier results along these lines).

Another application of Theorem~\ref{thm:C} is the following.
\begin{theorem}\label{thm:F}
Under the hypotheses of Theorem~\ref{thm:E}, assume in addition that $\phi$ is either ramified at $p$ and odd, or unramified at $p$ and even. If ${\rm ord}_{s=1}L(E,s)=1$, then 
\[
\ord_{p}\biggl(\frac{L'(E, 1)}{\operatorname{Reg}(E / \Q) \cdot \Omega_{E}}\biggr)=\ord_{p}\biggl(\# \sha(E / \Q) \prod_{\ell \nmid \infty} c_{\ell}(E /\Q)\biggr),
\]
where 
\begin{itemize}
\item ${\rm Reg}(E/\Q)$ is the regulator of $E(\Q)$, 
\item $\Omega_E=\int_{E(\mathbb{R})}\vert\omega_E\vert$ is the N\'{e}ron  period associated to the N\'{e}ron differential $\omega_E$, and
\item $c_\ell(E/\Q)$ is the Tamagawa number of $E$ at the prime $\ell$.
\end{itemize} 
In other words, the $p$-part of the Birch--Swinnerton-Dyler formula for $E$ holds.
\end{theorem}

\subsection{Method of proof and outline of the paper}

Let us explain some of the ideas that go into the  proof of our main results, beginning with the proof of Theorem~\ref{thm:C}. As in \cite{greenvats}, our starting point is Greenberg's old observation \cite{greenberg-nagoya} that a ``main conjecture'' should be equivalent to an imprimitive one. More precisely, in the context of Theorem~\ref{thm:C}, for $\Sigma$ any finite set of non-archimedean primes of $K$ not containing any of the primes above $p$, this translates into the expectation that the $\Sigma$-imprimitive Selmer group $\mathfrak{X}_E^\Sigma$, obtained by relaxing the local condition defining (the Pontryagin dual of) $\mathfrak{X}_E$ at the primes $w\in\Sigma$, is $\Lambda$-torsion with
\begin{equation}\label{eq:BDP-imp}
{\rm char}_\Lambda\bigl(\mathfrak{X}_E^\Sigma\bigr)\Lambda^{\rm ur}\overset{?}=\bigl(\mathcal{L}_E^\Sigma\bigr)
\end{equation}
as ideals in $\Lambda^{\rm ur}$, where $\mathcal{L}_E^\Sigma:=\mathcal{L}_E\cdot\prod_{w\in\Sigma}\mathcal{P}_w(E)$ for certain elements in $\mathcal{P}_w(E)\in\Lambda$ interpolating, for varying characters $\chi$ of $\Gamma$, the $w$-local Euler factor of $L(E/K,\chi,s)$ evaluated $s=1$. 

A key advantage of the imprimitive main conjecture (\ref{eq:BDP-imp}) is that (unlike the original conjecture), for suitable choices of $\Sigma$, its associated Iwasawa invariants are well-behaved with respect to congruences mod $p$. Identifying $\Lambda$ with the power series ring $\Z_p\llbracket T\rrbracket$ by setting $T=\gamma-1$ for a fixed topological generator $\gamma\in\Gamma$, recall that by the Weierstrass preparation theorem, 
every nonzero $g\in\Lambda$ can be uniquely written in the form
\[
g=u\cdot p^\mu\cdot Q(T),
\]
with $u\in\Lambda^\times$, $\mu=\mu(g)\in\Z_{\geqslant 0}$, and $Q(T)\in\Z_p[T]$ a distinguished polynomial of degree $\lambda(g)$. The constants $\lambda$ and $\mu$ are the so-called \emph{Iwasawa invariants} of $g$. For a torsion $\Lambda$-module $\mathfrak{X}$ we let $\lambda(\mathfrak{X})$ and $\mu(\mathfrak{X})$ be the Iwasawa invariants of a characteristic power series for $\mathfrak{X}$, and for a nonzero $\mathcal{L}\in\Lambda^{\rm ur}$ we let $\lambda(\Lcal)$ and $\mu(\Lcal)$ be the Iwasawa invariants of any element of $\Lambda$ generating the same $\Lambda^{\rm ur}$-ideal as $\mathcal{L}$. 

As a first step towards Theorem~\ref{thm:C}, we deduce from the $G_\Q$-module isomorphism $E[p]^{ss}=\mathbb{F}_p(\phi)\oplus\mathbb{F}_p(\psi)$ that, taking $\Sigma$ to consist of primes that are split in $K$ and containing all the primes of bad reduction for $E$, the module $\mathfrak{X}_E^\Sigma$ is $\Lambda$-torsion with
\begin{equation}\label{eq:alg-inv}
\mu\bigl(\mathfrak{X}_E^\Sigma\bigr)=0\quad\textrm{and}\quad\lambda\bigl(\mathfrak{X}_E^\Sigma\bigr)=\lambda\bigl(\mathfrak{X}_\phi^\Sigma\bigr)+\lambda\bigl(\mathfrak{X}_\psi^\Sigma\bigr),
\end{equation}
where $\mathfrak{X}_\phi^\Sigma$ and $\mathfrak{X}_\psi^\Sigma$ are anticyclotomic Selmer groups (closely related to the Pontryagin dual of certain class groups) for the Teichm\"uller lifts of $\phi$ and $\psi$, respectively. The proof of (\ref{eq:alg-inv}), which is taken up in $\S\ref{sec:algebraic}$, uses Rubin's work \cite{rubinmainconj} on the Iwasawa main conjecture for imaginary quadratic fields and Hida's work \cite{hidamu=0} on the vanishing of the $\mu$-invariant of $p$-adic Hecke $L$-functions. 

On the other hand, in $\S\ref{IMCII}$ we deduce from the main result of \cite{kriz} that for such $\Sigma$ one also has
\begin{equation}\label{eq:an-inv}
\mu\bigl(\Lcal_E^\Sigma\bigr)=0\quad\textrm{and}\quad\lambda\bigl(\Lcal_E^\Sigma\bigr)=\lambda\bigl(\Lcal_\phi^\Sigma\bigr)+\lambda\bigl(\Lcal_\psi^\Sigma\bigr),
\end{equation}
where $\mathcal{L}_\phi^\Sigma$ and $\Lcal_\psi^\Sigma$ are $\Sigma$-imprimitive anticyclotomic Katz $p$-adic $L$-functions attached to $\phi$ and $\psi$, respectively. 

With equalities (\ref{eq:alg-inv}) and (\ref{eq:an-inv}) in hand, 
it follows easily that to prove  the equality of characteristic ideals in Conjecture~\ref{conj:BDP} it suffices to prove one of the predicted divisibilities in $\Lambda^{\rm ur}[\frac{1}{p}]$. In $\S\ref{IMCIII}$,
{by combining Howard's approach to proving Iwasawa-theoretic divisibilities \cite{howard} with a Kolyvagin system argument along the lines of 
Nekov\'a\v{r}'s \cite{nekovar} (but adapted for twists by infinite order characters and for obtaining a bound on the length of Tate--Shafarevich group and not just an annihilator),}
we prove the main result towards one of the divisibilities in Conjecture~\ref{conj:PR}: 
{${\rm char}_{\Lambda_{\rm ac}}(\mathcal{X}_{\rm tors})$ divides ${\rm char}_{\Lambda_{\rm ac}}\bigl(\mathcal{S}/\mathcal{H}\bigr)^2$ in 
$\Lambda_{\rm ac}[\frac{1}{T}]$, and even in 
$\Lambda_{\rm ac}$ assuming \eqref{eq:intro-sel}.} As already noted, hypotheses (\ref{eq:intro-spl}) and $\phi\vert_{G_p}\neq\mathds{1},\omega$ are not needed at this point. 
This yields a corresponding divisibility in (\ref{eq:BDP-imp}), from which the proof of Theorem~\ref{thm:C} follows easily. The details of the final argument, and the deduction of Corollary~\ref{cor:D}, are given in $\S\ref{sec:CD}$.
The additional hypothesis \eqref{eq:intro-sel} is required to circumvent the growth of the
`error term' in our 
Kolyvagin system arguments in the cases of twists by anticyclotomic characters 
$p$-adically close to the trivial character. {The arguments in $\S\ref{IMCIII}$ apply equally well to both the residually reducible and residually irreducible cases.}

Finally, the proofs of Theorems~\ref{thm:E} and \ref{thm:F} are given in $\S\ref{sec:EF}$, 
and they are both obtained as an application of Theorem~\ref{thm:C} for a suitably chosen $K$. In particular, the proof of Theorem~\ref{thm:F} requires knowing the $p$-part of the Birch--Swinnerton-Dyler formula in analytic rank $0$
for the quadratic twist $E^K$; this is deduced in Theorem~\ref{thmGV} from the results of Greenberg--Vatsal \cite{greenvats}, and this is responsible for the additional hypotheses on $\phi$ placed in Theorem~\ref{thm:F}. 

\subsection{Examples}

To illustrate Theorem~\ref{thm:F}, take $p=5$ and consider the elliptic curve
\begin{displaymath}
J: y^2+y=x^3+x^2-10x+10.
\end{displaymath} 
The curve $J$ has conductor $123$ and analytic rank $1$, and satisfies $J[5]^{ss}=\Z/5\Z\oplus\mu_p$ as $G_\Q$-modules ($J$ has a rational $5$-torsion point). If $\psi$ is an even quadratic character such that $\psi(5)=-1$, corresponding to a real quadratic field $\Q(\sqrt{c})$ in which $5$ is inert, then the twist $E=J_c$ of $J$ by $\psi$ satisfies the hypotheses of Theorem~\ref{thm:E} with $p=5$. Since the root number of $J$ is $-1$ (being of analytic rank one), by \cite[Thm.~B.2]{twist} we can find infinitely many $\psi$ as above for which the associated twist $E=J_c$ also has analytic rank one, and therefore for which Theorem~\ref{thm:F} applies. 

One can proceed similarly for $p=3$ (resp. $p=7$), taking real quadratic twists of, for example, the elliptic curve $y^2+y=x^3+x^2-7x+5$ of conductor 91 (resp. $y^2+xy+y=x^{3}-x^{2}-19353x+958713$ of conductor 574). For $p=13$ (resp. $p=37$), one can do the same, possibly choosing the quadratic character to be odd and/or imposing conditions at some bad primes depending on the character describing the kernel of the isogeny (which is not trivial in these cases) in order to apply \cite[Thm.~B.2]{twist}. One could consider, for example, twists of the elliptic curve $y^2+y=x^3-21x+40$ of conductor 441 (resp. $y^2+xy+y=x^3+x^2-8x+6$ of conductor 1225). 

We also note that, for each of the four primes primes above, $p=3,5,7,13$, there are infinitely many distinct $j$-invariants to which Theorem~\ref{thm:F} applies, as $X_0(p)$ has genus $0$ in these cases.

\subsection{Relation to previous works}

Results in the same vein as (\ref{eq:alg-inv}) and (\ref{eq:an-inv}) were first obtained by Greenberg--Vatsal \cite{greenvats} in the cyclotomic setting; combined with Kato's Euler system divisibility \cite{kato-euler-systems}, these results led to their proof of the cyclotomic Iwasawa main conjecture for rational elliptic curves at Eisenstein primes $p$ (under some hypotheses on the kernel of the associated rational $p$-isogeny). This paper might be seen as an extension of the Greenberg--Vatsal method for Eisenstein primes to the anticyclotomic setting.  
{However, for the anticyclotomic Selmer groups and $L$-functions considered in this paper we are able to avoid the possible variation within an isogeny class of elliptic curves of the $\mu$-invariants and periods, which must be dealt with in \cite{greenvats}. In large part this is because the periods
in the corresponding $p$-adic families are the CM periods of Hecke characters and not the periods of the elliptic curve.  Consequently, the methods are slightly more robust and the resulting applications somewhat more general. 
The $\mu$-invariants of anticyclotomic Selmer groups and $p$-adic $L$-functions were also studied in \cite{pollack2011anticyclotomic}, but for different Selmer conditions and hypotheses on $K$ (in fact, under the hypothesis \eqref{eq:intro-Heeg} the $p$-adic $L$-function in \cite{pollack2011anticyclotomic} vanishes identically and the Selmer group is not $\Lambda$-cotorsion).}

The ensuing applications {of Theorems \ref{thm:C} and Corollary \ref{cor:D}} to the $p$-converse of the Gross--Zagier--Kolyvagin theorem (Theorem~\ref{thm:E}) and the $p$-part of the Birch--Swinnerton-Dyer formula in analytic rank $1$ (Theorem~\ref{thm:F}) covers primes $p$ that were either left untouched by the recent works in these directions \cite{pCONVskinner,pCONVvenerucci,pBSDbert,wei-zhang,skinner-zhang,jsw,castella,pBSDbps} (where $p$ is assumed to be non-Eisenstein), or extending previous works \cite{tian,CLTZ,CCL} ($p=2$),  \cite{kriz-li} ($p=3$), \cite{CMpconverse} (CM cases). {Many of these results (especially \cite{pCONVskinner} and \cite{jsw}) also rely on progress toward 
 Conjecture \ref{conj:BDP} in the residually irreducible case. Such progress has generally come via Eisenstein congruences on higher rank unitary groups and has explicitly excluded the Eisenstein cases considered in this paper.}

\subsection{Weight two newforms} The methods and results of this paper should easily extend to cuspidal newforms of weight two and trivial character that are congruent to Eisenstein series at a prime above $p$. We have focused on the case of elliptic curves in the interest of not obscuring the main features of our argument with cumbersome notation. The general case will be addressed in later work that will also consider higher weight forms as well as Hilbert modular forms.

\subsection{Acknowledgements}
This paper has its origins in one of the projects proposed by F.C. and C.S. at the 2018 Arizona Winter School on Iwasawa theory, and we would like to thank the organizers for making possible a uniquely stimulating week during which G.G. and J.L. made very substantial progress on this project. 
{We also thank Ashay Burungale and the anonymous referees for many helpful comments and suggestions on an earlier draft of this paper.}
G.G. is grateful to Princeton University for the hospitality during a visit to F.C. and C.S. in February-March 2019. During the preparation of this paper, F.C. was partially supported by the NSF grant DMS-1946136 and DMS-2101458; G.G. was supported by the Engineering and Physical Sciences Research Council [EP/L015234/1], the EPSRC Centre for Doctoral Training in Geometry and Number Theory (The London School of Geometry and Number Theory), University College London;
C.S. was partially supported by the Simons Investigator Grant \#376203 from the Simons Foundation and by the NSF grant DMS-1901985.

\addtocontents{toc}{\protect\setcounter{tocdepth}{2}}
\addtocontents{lof}{\protect\setcounter{tocdepth}{2}}

\renewcommand{\thetheorem}{\arabic{section}.\arabic{subsection}.\arabic{theorem}}

\section{Algebraic side}\label{sec:algebraic}

In this section we prove Theorem~\ref{MAINalgside} below, relating the anticyclotomic Iwasawa invariants of an elliptic curve $E/\Q$ at a prime $p$ with $E[p]^{ss}=\mathbb{F}_p(\phi)\oplus\mathbb{F}(\psi)$ to the anticyclotomic Iwasawa invariants of the characters $\phi$ and $\psi$. 

Throughout, we fix a prime $p>2$ and an embedding $\iota_p:\overline{\Q}\hookrightarrow\overline{\Q}_p$, and let $K\subset\overline{\Q}$ be an imaginary quadratic field in which $p=v\bar{v}$ splits, with $v$ the prime of $K$ above $p$ induced by $\iota_p$.  We also fix an embedding $\iota_\infty:\overline{\Q}\hookrightarrow \C$.

Let $G_K = {\rm Gal}(\overline{\Q}/K)\subset G_\Q = {\rm Gal}(\overline{\Q}/\Q)$, and for each place $w$ of $K$ let $I_w\subset G_w\subset G_K$ be
corresponding inertia and decomposition groups. Let $\Frob_w\in G_w/I_w$ be the arithmetic Frobenius. For the prime $v\mid p$ we assume 
$G_v$ is chosen so that it is identified with ${\rm Gal}(\overline{\Q}_p/\Q_p)$ via $\iota_p$.

Let $\Gamma={\rm Gal}(K_\infty/K)$ be the Galois group of the anticyclotomic $\Z_p$-extension $K_\infty$ of $K$, and 
let $\Lambda=\Z_p\llbracket\Gamma\rrbracket$ be the anticyclotomic Iwasawa algebra. We shall often identify $\Lambda$ with the power series ring $\bZ_p\llbracket T\rrbracket$ by setting $T=\gamma-1$ for a fixed topological generator $\gamma\in\Gamma$.

\subsection{Local cohomology groups of characters}\label{localchar}

Let  $\theta:G_K\rightarrow\mathbb{F}_p^\times$  
be a character with conductor divisible only by primes that are split in $K$. Via the Teichm\"uller lift $\mathbb{F}_p^\times\hookrightarrow\Z_p^\times$, we shall also view $\theta$ as taking values in $\Z_p^\times$. Set
\[
M_\theta=\Z_p(\theta)\otimes_{\Z_p}\Lambda^\vee,
\]
where $(-)^\vee = {\rm Hom}_{\rm cts}(-,\bQ_p/\bZ_p)$ for topological $\Z_p$-modules. The module $M_\theta$ is equipped with a $G_K$-action via $\theta\otimes\Psi^{-1}$, where $\Psi:G_K\rightarrow\Lambda^\times$ is the character arising from the projection $G_K\twoheadrightarrow\Gamma$.

In this section, we study the local cohomology of $M_\theta$ at various primes $w$ of $K$. 

\subsubsection{$w\nmid p$ split in $K$}

Let $w$ be a prime of $K$ lying over a prime $\ell\neq p$ split in $K$, and let $\Gamma_w\subset \Gamma$ be the corresponding decomposition group. Let $\gamma_w\in\Gamma_w$ be 
the image of $\Frob_w$, and set
\begin{equation}\label{eq:Euler}
\mathcal{P}_w(\theta)=P_w(\ell^{-1}\gamma_w)
\in\Lambda,
\end{equation}
where $P_w={\rm det}(1-{\rm Frob}_wX\vert\bQ_p(\theta)_{I_w})$ is the Euler factor at $w$ of the $L$-function of $\theta$.
 
\begin{lemma}\label{lemmawneqp} 
The module $\rH^1(K_w,M_\theta)^\vee$ is $\Lambda$-torsion with
\[
{\rm char}_\Lambda(\rH^1(K_w,M_\theta)^\vee)=(\mathcal{P}_w(\theta)).
\]
In particular, $\rH^1(K_w,M_\theta)^\vee$ has $\mu$-invariant zero.
\end{lemma}

\begin{proof} 
Since $\ell$ splits in $K$, it follows from class field theory that the index $[\Gamma:\Gamma_w]$ is finite (i.e., $w$ is finitely decomposed in $K_\infty/K$). Thus the argument proving \cite[Prop.~2.4]{greenvats} can be immediately adapted to yield this result.
\end{proof}

\subsubsection{$w\mid p$} Recall that we assume that $p=v\bar{v}$ splits in $K$. We begin by recording the following commutative algebra lemma, which shall also be used later in the paper.

\begin{lemma}\label{lem:commalg}
Let $X$ be a finitely generated $\Lambda$-module satisfying the following two properties: 
\begin{itemize}
\item $X[T]=0$, 
\item $X/TX$ is a free $\Z_p$-module of rank $r$. 
\end{itemize}
Then $X$ is a free $\Lambda$-module of rank $r$.
\end{lemma}

\begin{proof}
From Nakayama's lemma we obtain a surjection $\pi:\Lambda^r\twoheadrightarrow X$ which becomes an isomorphism $\bar{\pi}$ after reduction modulo $T$. Letting $K={\rm ker}(\pi)$, from the snake lemma we deduce the exact sequence
\[
0\rightarrow K/TK\rightarrow(\Lambda/T\Lambda)^r\xrightarrow{\bar{\pi}}X/TX\rightarrow 0.
\]
Thus $K/TK=0$, and so $K=0$ by another application of Nakayama's lemma. 
\end{proof}

Let $\omega:G_\Q\rightarrow \mathbb{F}_p^\times$ be the mod $p$ cyclotomic character.  
Let $w$ be a prime of $K$ above $p$.

\begin{prop}\label{propw=v}
Assume that $\theta\vert_{G_w}\neq\mathds{1},\omega$. Then:
\begin{enumerate}
\item[(i)] The restriction map
\[
r_w:\rH^1(K_w,M_\theta)\rightarrow\rH^1(I_w,M_\theta)^{G_w/I_w}
\]
is an isomorphism. 
\item[(ii)] $\rH^1(K_w,M_{\theta})$ is $\Lambda$-cofree of rank $1$.
\end{enumerate}
\end{prop}

\begin{proof} 
The map $r_w$ is clearly surjective, so it suffices to show injectivity. Since $G_w/I_w$ is pro-cyclic,
\[
{\rm ker}(r_w)\simeq M_\theta^{I_w}/({\rm Frob}_w-1)M_\theta^{I_w},
\]
where ${\rm Frob}_w$ is a Frobenius element at $w$. Taking Pontryagin duals to the exact sequence
\[
0\rightarrow M_\theta^{G_w}\rightarrow M_\theta^{I_w}\xrightarrow{{\rm Frob}_w-1}M_\theta^{I_w}\rightarrow M_\theta^{I_w}/({\rm Frob}_w-1)M_\theta^{I_w}\rightarrow 0
\]
and using the vanishing of $M_\theta^{G_w}$ (which follows from $\theta\vert_{G_w}\neq\mathds{1}$) we deduce a $\Lambda$-module surjection 
\begin{equation}\label{eq:surj}
(M_\theta^\vee)_{I_w}\twoheadrightarrow(M_\theta^\vee)_{I_w},
\end{equation}
hence an isomorphism (by the Noetherian property of $\Lambda$).
Since the kernel of (\ref{eq:surj}) is isomorphic to ${\rm ker}(r_w)^\vee$, (i)  follows. For (ii), in light of Lemma~\ref{lem:commalg}, letting 
\[
X:=\rH^1(K_w,M_\theta)^\vee,
\] 
it suffices to show that $X[T]=0$ and the quotient $X/TX$ is $\bZ_p$-free of rank $1$. Taking cohomology for the exact sequence $0\rightarrow\bQ_p/\bZ_p(\theta)\rightarrow M_{\theta}\xrightarrow{\times T}M_{\theta}\rightarrow 0$ we obtain 
\begin{equation}\label{eq:coh-char}
\frac{\rH^1(K_w,M_{\theta})}{T\rH^1(K_w,M_{\theta})}=0,\quad
\quad\rH^1(K_w,\Q_p/\Z_p(\theta))\simeq\rH^1(K_w,M_{\theta})[T],
\end{equation}
using that $\rH^2(K_w,M_\theta)=0$ (which follows from $\theta|_{G_w} \neq \omega$)
for the first isomorphism and 
$\rH^0(K_w,M_{\theta})=0$ for the second. The first isomorphism shows that $X[T]=0$. On the other hand, taking cohomology for the exact sequence $0\rightarrow\mathbb{F}_p(\theta)\rightarrow\bQ_p/\bZ_p(\theta)\xrightarrow{p}\bQ_p/\bZ_p(\theta)\rightarrow 0$ and using that $\theta\vert_{G_w}\neq\omega$ we obtain
\[
\frac{\rH^1(K_w,\bQ_p/\bZ_p(\theta))}{p\rH^1(K_w,\bQ_p/\bZ_p(\theta))}\simeq\rH^2(K_w,\mathbb{F}_p(\theta))=0,
\]
which together with the second isomorphism in (\ref{eq:coh-char}) shows that $X/TX\simeq\rH^1(K_w,\bQ_p/\bZ_p(\theta))^\vee$ is $\bZ_p$-free of rank $1$ (the value of the rank following from the local Euler characteristic formula), concluding the proof.
\end{proof}

\subsection{Selmer groups of characters}
\label{subsec:Sel-char}

As in the preceding section, let $\theta:G_K\rightarrow\mathbb{F}_p^\times$ be a character whose conductor is divisible only by primes split in $K$ (that is, which are unramified over $\Q$ and have degree one).

Let $\Sigma$ be a finite set of places of $K$ containing $\infty$ and the primes dividing $p$ or the conductor of $\theta$
and such that every finite place in $\Sigma$ is split in $K$, and denote by $K^\Sigma$ the maximal extension of $K$ unramified outside $\Sigma$. 

\begin{definition} 
The \emph{Selmer group} of $\theta$ is
\[
\rH^1_{\Fcal_{\rm Gr}}(K,M_\theta):=\ker\biggr\{\rH^1(K^\Sigma/K,M_\theta)\rightarrow\prod_{w \in \Sigma, w \nmid p}\rH^1(K_w,M_\theta)\times\rH^1(K_{\bar{v}},M_\theta)\biggr\},
\]
and letting $S=\Sigma\setminus\{v,\bar{v},\infty\}$, we define the \emph{$S$-imprimitive Selmer group} of $\theta$ by 
\[
\rH^1_{\Fcal_{\rm Gr}^S}(K,M_\theta):=\ker\biggr\{\rH^1(K^\Sigma/K,M_\theta)\rightarrow\rH^1(K_{\bar{v}},M_\theta)\biggr\}.
\]
Replacing $M_\theta$ by $M_\theta[p]$ in the above definitions, we obtain the \emph{residual Selmer group} $\rH^1_{\Fcal_{\rm Gr}}(K,M_\theta[p])$ and its $S$-imprimitive variant $\rH^1_{\Fcal_{\rm Gr}^S}(K,M_\theta[p])$.
\end{definition}

It is well-known that the above  groups are cofinitely generated over the corresponding Iwasawa algebra ($\Lambda$ and $\Lambda/p$), and that the Selmer group and residual Selmer groups 
are independent of the choice of the set $\Sigma$ as above. 


The following result, combining work of Rubin  
and Hida, will play a key role in our proofs.

\begin{theorem}[Rubin, Hida]\label{thmmu=0} 
Assume that $\theta\vert_{G_{\bar{v}}}\neq\mathds{1},\omega$. Then $\rH^1_{\Fcal_{\rm Gr}}(K,M_\theta)^\vee$ is a torsion $\Lambda$-module with $\mu$-invariant zero.
\end{theorem}

\begin{proof}
Let $K_\theta\subset\overline{\bQ}$ be the fixed field of ${\rm ker}(\theta)$, and set $\Delta_\theta={\rm Gal}(K_\theta/K)$. The restriction map
\[
\rH^1(K^\Sigma/K,M_\theta)\rightarrow\rH^1(K^\Sigma/K_\theta,M_\theta)^{\Delta_\theta}
\] 
is an isomorphism (since $p\nmid\vert\Delta_\theta\vert$), which combined with Shapiro's lemma gives rise to an identification
\begin{equation}\label{eq:CFT}
\rH^1(K^\Sigma/K_\theta,M_\theta)
\simeq{\rm Hom}_{\rm cts}((\mathcal{X}_\infty^\Sigma)^\theta,\bQ_p/\bZ_p),
\end{equation}
where $\mathcal{X}_\infty^\Sigma={\rm Gal}(\mathcal{M}_\infty^\Sigma/K_\infty K_\theta)$ is the Galois group of the maximal abelian pro-$p$ extension of $K_\infty K_\theta$ unramified outside $\Sigma$, and $(\mathcal{X}_\infty^\Sigma)^\theta$ is the $\theta$-isotypic component of $\mathcal{X}_\infty^\Sigma$ for the action of $\Delta_\theta$, identified as a subgroup of ${\rm Gal}(K_\infty K_\theta/K)$ via the decomposition ${\rm Gal}(K_\infty K_\theta/K)\simeq\Gamma\times\Delta_\theta$. 

Now, by \cite[Rem.~3.2]{pollack2011anticyclotomic} (since the primes $w\nmid p$ in $\Sigma$ are finitely decomposed in $K_\infty/K$) and Proposition~\ref{propw=v}(i), the Selmer group $\rH^1_{\Fcal_{\rm Gr}}(K,M_\theta)$ is the same as the one defined by the unramified local conditions, i.e., as
\[
\ker\biggr\{\rH^1(K^\Sigma/K,M_\theta)\rightarrow\prod_{w \in \Sigma, w \nmid p}\rH^1(I_w,M_\theta)^{G_v/I_v}\times\rH^1(I_{\bar{v}},M_\theta)^{G_{\bar{v}}/I_{\bar{v}}}\biggr\},
\]
and so under the identification $(\ref{eq:CFT})$ we obtain
\[
\rH^1_{\Fcal_{\rm Gr}}(K,M_\theta)\simeq{\rm Hom}_{\rm cts}(\mathcal{X}_\infty^\theta,\bQ_p/\bZ_p)
\] 
where $\mathcal{X}_\infty={\rm Gal}(\mathcal{M}_\infty/K_\infty K_\theta)$ is the Galois group of the maximal abelian pro-$p$ extension of $K_\infty K_\theta$ unramified outside $v$. Thus from the works of  Rubin \cite{rubinmainconj},
which identifies $\mathrm{char}_\Lambda(\rH^1_{\Fcal_{\rm Gr}}(K,M_\theta)^\vee)$ with the ideal generated by an anticyclotomic projection of a Katz $p$-adic $L$-function,  
and Hida \cite{hidamu=0}, proving the vanishing of the $\mu$-invariant of such anticyclotomic $p$-adic $L$-functions, we obtain the theorem.
\end{proof}

\begin{remark}
Following the notations introduced in the proof of Theorem~\ref{thmmu=0}, and letting $\mathcal{X}_\infty^{sp}={\rm Gal}(\mathcal{M}_\infty^{sp}/K_\infty K_\theta)$ be the Galois group of the maximal abelian pro-$p$ extension of $K_\infty K_\theta$ unramified outside $v$ and in which the primes above $\bar{v}$ split completely, Proposition~\ref{propw=v}(i) shows $(\mathcal{X}_\infty^\Sigma)^\theta=(\mathcal{X}_\infty^{sp})^\theta$. 
\end{remark}

The next two results will allow us to determine $\lambda(\mathfrak{X}_\theta^S)$ in terms of the residual Selmer group $\rH^1_{\Fcal_{\rm Gr}^S}(K,M_\theta[p])$. In brief, the fact that $\mathfrak{X}_\theta^S$ has no nonzero pseudo-null $\Lambda$-submodules (shown in Proposition~\ref{propchar} below) yields the equality $\lambda(\mathfrak{X}_\theta^S)=\mathrm{dim}_{\mathbb{F}_p}\bigl(\rH^{1}_{\Fcal_{\rm Gr}^S}(K,M_\theta)[p]\bigr)$, which combined with the next lemma yields the desired result.

\begin{lemma}\label{rmkchar}
Assume that $\theta\vert_{G_{\bar{v}}}\neq\mathds{1}$. Then
\[
\rH^1_{\Fcal_{\rm Gr}^S}(K,M_\theta[p])\simeq
\rH^1_{\Fcal_{\rm Gr}^S}(K,M_\theta)[p].
\] 
\end{lemma}

\begin{proof}
The hypothesis on $\theta$ implies in particular that $\rH^0(K,\mathbb{F}_p(\theta))=0$, and so $\rH^0(K,M_\theta)=0$. Thus the natural map
\[
\rH^1(K^\Sigma/K,M_\theta[p])\rightarrow\rH^1(K^\Sigma/K,M_\theta)[p]
\]
induced by multiplication by $p$ on $M_\theta$ is an isomorphism. To conclude it suffices to check that the natural map $r_{\bar{v}}:\rH^1(K_{\bar{v}},M_\theta[p])\rightarrow\rH^1(K_{\bar{v}},M_\theta)[p]$ is an injection, but since $\rH^0(K_{\bar{v}},\mathbb{F}_p(\theta))=0$ by the hypothesis, the same argument as above shows that $r_{\bar{v}}$ is an isomorphism.
\end{proof}

Let
\[
\mathfrak{X}_\theta^S:=\rH^1_{\Fcal_{\rm Gr}^S}(K,M_\theta)^\vee \ \ \text{and} \ \ 
\mathfrak{X}_\theta:=\rH^1_{\Fcal_{\rm Gr}}(K,M_\theta)^\vee,
\]
and recall the element $\mathcal{P}_w(\theta)\in\Lambda$ introduced in $(\ref{eq:Euler})$.

\begin{prop}\label{propchar} 
Assume that $\theta\vert_{G_{\bar{v}}}\neq\mathds{1},\omega$. Then
$\mathfrak{X}_{\theta}^{S}$ is a torsion $\Lambda$-module with $\mu$-invariant zero and 
its $\lambda$-invariant satisfies
\begin{displaymath}
\lambda\bigl(\mathfrak{X}_{\theta}^{S}\bigr)=\lambda\bigl(\mathfrak{X}_{\theta}\bigr)+\sum_{w\in\Sigma, w\nmid p}\lambda\bigl(\mathcal{P}_w(\theta)\bigr).
\end{displaymath} 
Moreover, $\rH^{1}_{\Fcal^{S}_{\rm Gr}}(K,M_\theta[p])$ is finite and 
\[
\mathrm{dim}_{\mathbb{F}_p}\bigl(\rH^{1}_{\Fcal_{\rm Gr}^S}(K,M_\theta[p])\bigr)=\lambda\bigl(\mathfrak{X}_\theta^S\bigr).
\]
\end{prop}

\begin{proof}
Since $\mathfrak{X}_{\theta}$ is $\Lambda$-torsion by Theorem~\ref{thmmu=0} and the Cartier dual ${\rm Hom}(\bQ_p/\bZ_p(\theta),\mu_{p^\infty})$ has no non-trivial $G_{K_\infty}$-invariants, from \cite[Prop.~A.2]{pollack2011anticyclotomic} we obtain that the restriction map in the definition of $\rH^{1}_{\Fcal_{\rm Gr}}(K,M_\theta)$ is surjective, and so the sequence
\begin{equation}\label{eq:sur1}
 0 \rightarrow \rH^{1}_{\Fcal_{\rm Gr}}(K,M_\theta) \rightarrow\rH^{1}(K^{\Sigma}/K, M_\theta) \rightarrow \prod_{w\in\Sigma, w\nmid p}\rH^1(K_w,M_\theta) \times \rH^{1}(K_{\bar{v}}, M_\theta)\rightarrow 0
\end{equation}
is exact. From the definitions, this readily yields the exact sequence 
\begin{equation}\label{eq:sur2}
0 \rightarrow\rH^{1}_{\Fcal_{\rm Gr}}(K,M_\theta) \rightarrow\rH^{1}_{\Fcal_{\rm Gr}^S}(K,M_\theta) \rightarrow \prod_{w\in S}\rH^1(K_w,M_\theta) \rightarrow 0,
\end{equation}
which combined with Theorem~\ref{thmmu=0} and Lemma~\ref{lemmawneqp}  gives the first part of the proposition.

For the second part, note that $\rH^{2}(K^{\Sigma}/K,M_\theta)=0$. (Indeed, by the Euler characteristic formula, the $\Lambda$-cotorsionness of $\rH^{1}_{\Fcal_{\rm Gr}}(K,M_\theta)$ implies that $\rH^{2}(K^{\Sigma}/K,M_\theta)$ is $\Lambda$-cotorsion; being $\Lambda$-cofree, as follows immediately from the fact that ${\rm Gal}(K^\Sigma/K)$ has cohomological dimension $2$, it must vanish.) Thus from the long exact sequence in cohomology induced by 
$0\rightarrow \bQ_p/\bZ_p(\theta)\rightarrow M_\theta\xrightarrow{\times T}M_\theta\rightarrow 0$ 
we obtain the isomorphism
\[
\frac{\rH^{1}(K^{\Sigma}/K,M_\theta)}{T\rH^{1}(K^{\Sigma}/K,M_\theta)} \simeq\rH^{2}(K^{\Sigma}/K,\bQ_p/\bZ_p(\theta)).
\]

Since $\rH^{2}(K^{\Sigma}/K,\bQ_p/\bZ_p(\theta))$ is $\bZ_p$-cofree (because ${\rm Gal}(K^\Sigma/K)$ has cohomological dimension $2$), it follows that $\rH^{1}(K^{\Sigma}/K,M_\theta)^\vee$ has no nonzero pseudo-null $\Lambda$-submodules (\emph{cf.} \cite[Prop.~5]{greenberg-iwasawa}), and since $(\ref{eq:sur1})$ and $(\ref{eq:sur2})$ readily imply that
\[
\mathfrak{X}_\theta^{S}\simeq\frac{\rH^1(K^\Sigma/K,M_\theta)^\vee}{\rH^1(K_{\bar{v}},M_\theta)^\vee}
\]
as $\Lambda$-modules, by Proposition~\ref{propw=v}(iii) and \cite[Lem.~2.6]{greenvats} we conclude that also $\mathfrak{X}_\theta^S$ has no nonzero pseudo-null $\Lambda$-submodules. Finally, since $\mathfrak{X}_\theta^S$ is $\Lambda$-torsion with $\mu$-invariant zero by Theorem~\ref{thmmu=0}, the finiteness of $\rH^{1}_{\Fcal^{S}_{\rm Gr}}(K,M_\theta)[p]$ (and therefore of $\rH^{1}_{\Fcal^{S}_{\rm Gr}}(K,M_\theta[p])$ by Lemma~\ref{rmkchar}) follows from the structure theorem. It also follows that 
$\rH^1_{\mathcal{F}_{\rm Gr}^S}(K,M_\theta)$ is divisible. In particular,
\[
\rH^1_{\mathcal{F}_{\rm Gr}^S}(K,M_\theta)\simeq(\bQ_p/\bZ_p)^{\lambda},
\]
where $\lambda=\lambda(\mathfrak{X}_\theta^S)$, which together with Lemma~\ref{rmkchar} gives the final formula for the $\lambda$-invariant.
\end{proof}

The following corollary will be used crucially in the next section.

\begin{cor}\label{corcharacters} 
Assume that $\theta\vert_{G_{\bar{v}}}\neq\mathds{1},\omega$. Then 
$\rH^{2}(K^{\Sigma}/K,M_\theta[p])=0$ and the sequence
\[
0\rightarrow\rH^1_{\Fcal_{\rm Gr}^S}(K,M_\theta[p])\rightarrow\rH^1(K^\Sigma/K,M_\theta[p])\rightarrow\rH^1(K_{\bar{v}},M_\theta[p])\rightarrow 0
\]
is exact.
\end{cor}

\begin{proof} 
In the course of the proof of Proposition~\ref{propchar} we showed that $\rH^2(K^\Sigma/K,M_\theta)=0$, and so the cohomology long exact sequence induced by multiplication by $p$ on $M_\theta$ yields an isomorphism
\begin{equation}\label{eq:iso}
\frac{\rH^1(K^\Sigma/K,M_\theta)}{p\rH^1(K^\Sigma/K,M_\theta)}\simeq\rH^2(K^\Sigma/K,M_\theta[p]).
\end{equation}
On the other hand, from the exactness of $(\ref{eq:sur1})$  we deduce the exact sequence
\begin{equation}\label{eq:sur3}
0\rightarrow\rH^1_{\Fcal_{\rm Gr}^S}(K,M_\theta)\rightarrow\rH^1(K^\Sigma/K,M_\theta)\rightarrow\rH^1(K_{\bar{v}},M_\theta)\rightarrow 0.
\end{equation}
Since we also showed in that proof that $\rH^1_{\Fcal_{\rm Gr}^S}(K,M_\theta)$ is divisible, and $\rH^1(K_{\bar{v}},M_\theta)$ is $\Lambda$-cofree by Proposition~\ref{propw=v}(ii), it follows from $(\ref{eq:sur3})$ that $\rH^1(K^\Sigma/K,M_\theta)^\vee$ has no $p$-torsion, and so 
\[
\rH^2(K^\Sigma/K,M_\theta[p])=0
\] 
by $(\ref{eq:iso})$, giving the first claim in the statement.

For the second claim, consider the commutative diagram
\begin{center}
\begin{tikzcd}[column sep=2em, row sep=2em]
0 \arrow[r] & \rH^1_{\Fcal_{\rm Gr}^S}(K,M_\theta) \arrow[d, "p"] \arrow[r]  \arrow[r] & \rH^{1}(K^{\Sigma}/K, M_\theta) \arrow[r] \arrow[d, "p"]  \arrow[r] & \rH^{1}(K_{\bar{v}},M_\theta) \arrow[d, "p"] \arrow[r]  & 0 \\
0 \arrow[r] & \rH^1_{\Fcal_{\rm Gr}^S}(K,M_\theta) \arrow[r] \arrow[r] & \rH^{1}(K^\Sigma/K,M_\theta) \arrow[r]  & \rH^{1}(K_{\bar{v}},M_\theta) \arrow[r] & 0,
\end{tikzcd}
\end{center}
where the vertical maps are the natural ones induced by multiplication by $p$ on $M_\theta$. Since $\rH^1_{\Fcal_{\rm Gr}^S}(K,M_\theta)$ is divisible, the snake lemma applied to this diagram yields the exact sequence
\[
0\rightarrow\rH^1_{\Fcal_{\rm Gr}^S}(K,M_\theta)[p]\rightarrow\rH^1(K^\Sigma/K,M_\theta)[p]\rightarrow\rH^1(K_{\bar{v}},M_\theta)[p]\rightarrow 0,
\]
which by Lemma~\ref{rmkchar} (and the natural isomorphisms shown in its proof) is identified with the exact sequence in the statement.
\end{proof}

\subsection{Local cohomology groups of \texorpdfstring{$E$}{E}}
\label{subsec:local-E}

Now we let $E/\mathbb{Q}$ be an elliptic curve of conductor $N$ with good reduction at $p$ and admitting a rational $p$-isogeny. 
The $G_\bQ$-module $E[p]$ is therefore reducible, fitting into an exact sequence 
\begin{equation}\label{eq:es-mod}
0\rightarrow\mathbb{F}_p(\phi)\rightarrow E[p]\rightarrow\mathbb{F}_p(\psi)\rightarrow 0,
\end{equation}
where $\phi,\psi:G_\bQ\rightarrow\mathbb{F}_p^\times$ are characters such that $\phi\psi=\omega$ by the Weil pairing. We assume that every prime $\ell\vert N$ splits in $K$ and continue to assume that $p=v\bar{v}$ splits in $K$, so the results of the preceding sections can be applied to the restrictions of $\phi$ and $\psi$ to $G_K$.

Let $T=T_pE$ be the $p$-adic Tate module of $E$, and denote by $M_E$ the $G_K$-module
\[
M_E:=T\otimes_{\bZ_p}\Lambda^\vee,
\]
where the tensor product is endowed with the diagonal $G_K$-action (and the action on $\Lambda^\vee$ is via $\Psi^{-1}$, as before).

\begin{lemma}\label{lemma411}
Let $w$ be a prime of $K$ above $p$, and assume that $E(K_w)[p]=0$. Then $\rH^1(K_w,M_E)$ is $\Lambda$-cofree of rank $2$.
\end{lemma}

\begin{proof} 
The proof is virtually the same as the proof of Proposition~\ref{propw=v}(ii). Letting $X:=\rH^1(K_w,M_E)^\vee$, by Lemma~\ref{lem:commalg} it suffices to show that $X[T]=0$ and $X/TX$ is $\bZ_p$-free of rank $2$. The hypotheses imply that $E(K_w)[p^\infty]=0$, and so $\rH^2(K_w,E[p^\infty])=0$ by local duality. Taking cohomology for the exact sequence 
\[
0\rightarrow E[p^\infty]\rightarrow M_E\xrightarrow{\times T}M_E\rightarrow 0
\] 
it follows that
\begin{equation}\label{eq:coh}
\frac{\rH^1(K_w,M_E)}{T\rH^1(K_w,M_E)}\simeq\rH^2(K_w,E[p^\infty])=0,\quad
\quad\rH^1(K_w,E[p^\infty])\simeq\rH^1(K_w,M_E)[T].
\end{equation}

The first isomorphism shows that $X[T]=0$. On the other hand, taking cohomology for the exact sequence $0\rightarrow E[p]\rightarrow E[p^\infty]\xrightarrow{p}E[p^\infty]\rightarrow 0$ we obtain
\[
\frac{\rH^1(K_w,E[p^\infty])}{p\rH^1(K_w,E[p^\infty])}\simeq\rH^2(K_w,E[p])=0,
\]
which together with the second isomorphism in (\ref{eq:coh}) shows that $X/TX\simeq\rH^1(K_w,E[p^\infty])^\vee$ is $\bZ_p$-free. That its rank is $2$ follows from the local Euler characteristic formula.
\end{proof}

\subsection{Selmer groups of \texorpdfstring{$E$}{E}}
\label{subsec:Sel-E}

Fix a finite set $\Sigma$ of places of $K$ containing $\infty$ and the primes above $Np$, and such that the finite places in $\Sigma$ are all split in $K$. 

Similarly as in $\S\ref{subsec:Sel-char}$, we define a Selmer group for $E$ by  
\[
\rH^1_{\Fcal_{\rm Gr}}(K,M_E):=\ker\biggr\{\rH^1(K^\Sigma/K,M_E)\rightarrow\prod_{w \in \Sigma, w \nmid p}\rH^1(K_w,M_E)\times\rH^1(K_{\bar{v}},M_E)\biggr\},
\]
and an 
{$S$-imprimitive Selmer group}, where $S=\Sigma\setminus\{v,\bar{v},\infty\}$, by 
\[
\rH^1_{\Fcal_{\rm Gr}^S}(K,M_E):=\ker\biggr\{\rH^1(K^\Sigma/K,M_E)\rightarrow\rH^1(K_{\bar{v}},M_E)\biggr\}.
\]
The residual Selmer groups $\rH^1_{\Fcal_{\rm Gr}}(K,M_E[p])$ and $\rH^1_{\Fcal_{\rm Gr}^S}(K,M_E[p])$ are defined in the same manner.

Viewing the characters $\phi$ and $\psi$ appearing in the exact sequence $(\ref{eq:es-mod})$ as taking values in $\Z_p^\times$ via the Teichm\"uller lift, we obtain an exact sequence 
\begin{equation}\label{eq:esIw}
0 \rightarrow M_\phi[p] \rightarrow M_E[p] \rightarrow M_\psi[p] \rightarrow 0
\end{equation}
of $\mathrm{Gal}(K^{\Sigma}/K)$-modules. Let $G_p\subset G_\Q$ be a decomposition group at $p$.

\begin{prop}\label{cormodp} 
Assume that $\phi\vert_{G_{p}}\neq\mathds{1},\omega$. Then 
$(\ref{eq:esIw})$ induces a natural exact sequence 
\[
0 \rightarrow \rH^1_{\Fcal_{\rm Gr}^S}(K,M_\phi[p]) \rightarrow \rH^1_{\Fcal_{\rm Gr}^S}(K,M_E[p]) \rightarrow 
\rH^1_{\Fcal_{\rm Gr}^S}(K,M_\psi[p]) \rightarrow 0.
\]
In particular, $\rH^1_{\Fcal_{\rm Gr}^S}(K,M_E[p])$ is finite, and 
\[
\mathrm{dim}_{\mathbb{F}_p}\bigl(\rH^1_{\Fcal_{\rm Gr}^S}(K,M_E[p])\bigr)=\lambda(\mathfrak{X}_\phi^S)+\lambda(\mathfrak{X}_{\psi}^S).
\]
\end{prop}

\begin{proof} 
Taking cohomology for the exact sequence $(\ref{eq:esIw})$ we obtain the commutative diagram 
\begin{center}
\begin{tikzcd}[column sep=2em, row sep=2em]
0 \arrow[r] & \rH^{1}(K^{\Sigma}/K,M_\phi[p]) \arrow[d] \arrow[r] & \rH^{1}(K^{\Sigma}/K, M_E[p]) \arrow[r] \arrow[d]  & \rH^{1}(K^{\Sigma}/K,M_\psi[p]) \arrow[d]  \arrow[r] & 0 \\
0 \arrow[r] & \rH^{1}(K_{\bar{v}},M_\phi[p]) \arrow[r] & \rH^{1}(K_{\bar{v}},M_E[p]) \arrow[r]  & \rH^{1}(K_{\bar{v}},M_\psi[p])
\arrow[r] & 0,
\end{tikzcd}
\end{center}
where the exactness of the rows follows immediately from Corollary~\ref{corcharacters} and the hypothesis on $\phi$
(which implies that $\psi|_{G_p}\neq\mathds{1},\omega$ as well), 
and the vertical maps are given by restriction. Since the left vertical arrow is surjective by Corollary~\ref{corcharacters}, the snake lemma applied to this diagram yields the exact sequence in the statement. The last claim now follows from the last claim of Proposition~\ref{propchar}.
\end{proof}

Now we can relate the imprimitive residual and $p^\infty$-Selmer groups. Set
\[ 
\mathfrak{X}_{E}^S:=\rH^1_{\Fcal_{\rm Gr}^S}(K,M_E)^{\vee}, \quad \mathfrak{X}_{E}:=\rH^1_{\Fcal_{\rm Gr}}(K,M_E)^{\vee}.
\]

\begin{prop}\label{propmodp}
Assume that $\phi\vert_{G_{p}}\neq\mathds{1},\omega$. Then 
\[
\rH^1_{\Fcal_{\rm Gr}^S}(K,M_E[p]) \simeq \rH^1_{\Fcal_{\rm Gr}^S}(K,M_E)[p]. 
\]
Moreover, the modules $\mathfrak{X}_E^S$ and $\mathfrak{X}_E$ are both $\Lambda$-torsion with $\mu=0$.
\end{prop}

\begin{proof} 
Since $\psi=\omega\phi^{-1}$, our assumption on $\phi$ implies that $E(K_{\bar{v}})[p]=0$, and therefore $\rH^0(K_{\bar{v}},M_E)=0$. Thus the same argument as in the proof of Lemma~\ref{rmkchar} yields the  isomorphism in the statement. It follows from Proposition~\ref{cormodp} that $\rH^1_{\Fcal_{\rm Gr}^S}(K,M_E)[p]$ is finite, and so $\mathfrak{X}_E^S$ is $\Lambda$-cotorsion with $\mu=0$. Since $\mathfrak{X}_E$ is a quotient of $\mathfrak{X}_E^S$, this completes the proof.
\end{proof}

Now we can deduce the following analogue of Proposition~\ref{propchar} for $M_E$.

\begin{cor}\label{prop426} 
Assume that $\phi\vert_{G_{p}}\neq\mathds{1},\omega$. Then $\mathfrak{X}_E^S$ has no non-trivial finite $\Lambda$-submodules, and 
\[
\lambda\bigl(\mathfrak{X}_E^S\bigr)=\mathrm{dim}_{\mathbb{F}_p}\bigl(\rH^1_{\Fcal_{\rm Gr}^S}(K,M_E[p])\bigr).
\] 
\end{cor}

\begin{proof} 
Since $M_E^*={\rm Hom}(M_E,\mu_{p^\infty})$ has no non-trivial $G_K$-invariants and $\mathfrak{X}_E^S$ is $\Lambda$-torsion by Proposition~\ref{propmodp}, from \cite[Prop.~A.2]{pollack2011anticyclotomic} we deduce that the sequence
\begin{equation}\label{eq:1}
0 \rightarrow \rH^1_{\Fcal_{\rm Gr}}(K,M_E) \rightarrow \rH^{1}(K^{\Sigma}/K, M_E) \rightarrow \prod_{w\in\Sigma, w\nmid p}\rH^1(K_w,M_E)\times\rH^{1}(K_{\bar{v}}, M_E) \rightarrow 0
\end{equation}
is exact. Proceeding as in the proof of Proposition~\ref{propchar}, we see that the $\Lambda$-torsionness of $\mathfrak{X}_E$ implies that $\rH^{2}(K^{\Sigma}/K, M_E)=0$ and that $\rH^{1}(K^{\Sigma}/K, M_E)^{\vee}$ has no nonzero pseudo-null $\Lambda$-submodules. The exactness of $(\ref{eq:1})$ readily implies a $\Lambda$-module isomorphism 
\[
\mathfrak{X}_E^S\simeq\frac{\rH^1(K^\Sigma/K,M_E)^\vee}{\rH^{1}(K_{\bar{v}}, M_E)^\vee}.
\]
 
Since $\rH^1(K_{\bar{v}},M_E)$ is $\Lambda$-cofree by Lemma~\ref{lemma411}, we thus conclude from \cite[Lem.~2.6]{greenvats} that $\mathfrak{X}_E^S$ has no nonzero finite $\Lambda$-submodules. Together with the isomorphism  $\rH^1_{\Fcal_{\rm Gr}^S}(K,M_E[p]) \simeq \rH^1_{\Fcal_{\rm Gr}^S}(K,M_E)[p]$ of Proposition~\ref{propmodp}, the last claim in the statement
of the corollary follows from this.
\end{proof}

Finally, we note that as in Lemma~\ref{lemmawneqp}, one can show that for primes $w\nmid p$ split in $K$, the module $\rH^1(K_w,M_E)^\vee$ is $\Lambda$-torsion with characteristic ideal generated by the element
\[
\mathcal{P}_w(E)=P_w(\ell^{-1}\gamma_w)
\in\Lambda,
\]
where $P_w={\rm det}(1-{\rm Frob}_wX\vert V_{I_w})$, for $V=T\otimes\bQ_p$, is the Euler factor at $w$ of the $L$-function of $E$.

\subsection{Comparison I: Algebraic Iwasawa invariants}

We now arrive at the main result of this section. Recall that every prime $w\in\Sigma\setminus\{\infty\}$ is split in $K$, and we set $S=\Sigma\setminus\{v,\bar{v},\infty\}$. 

\begin{theorem}\label{MAINalgside}
Assume that $\phi\vert_{G_p}\neq\mathds{1},\omega$. Then the module $\mathfrak{X}_{E}$ is $\Lambda$-torsion with $\mu(\mathfrak{X}_{E})=0$ and 
\begin{displaymath}
\lambda\bigl(\mathfrak{X}_E\bigr)=\lambda\bigl(\mathfrak{X}_\phi\bigr)+\lambda\bigl(\mathfrak{X}_\psi\bigr)
+\sum_{w\in S} \lbrace \lambda\bigl(\mathcal{P}_w(\phi)\bigr)+\lambda\bigl(\mathcal{P}_w(\psi)\bigr)-\lambda\bigl(\mathcal{P}_w(E)\bigr) \rbrace.
\end{displaymath}
\end{theorem}

\begin{proof}
That $\mathfrak{X}_E$ is $\Lambda$-torsion with $\mu$-invariant zero is part of Proposition~\ref{propmodp}. For the $\lambda$-invariant, combining Corollary~\ref{prop426} and the last claim of Proposition~\ref{cormodp} we obtain
\begin{equation}\label{eq:lambda-imp}
\lambda\bigl(\mathfrak{X}_E^S\bigr)
=\lambda\bigl(\mathfrak{X}_\phi^S\bigr)+\lambda\bigl(\mathfrak{X}_{\psi}^S\bigr).
\end{equation}
On the other hand, from (\ref{eq:1}) we deduce the exact sequence
\[
0 \rightarrow \rH^1_{\Fcal_{\rm Gr}}(K,M_E) \rightarrow \rH^1_{\Fcal_{\rm Gr}^S}(K,M_E) \rightarrow \prod_{w\in S}\rH^1(K_w,M_E) \rightarrow 0, 
\]
and therefore the relation $\lambda\bigl(\mathfrak{X}_E^S\bigr)=\lambda\bigl(\mathfrak{X}_E\bigr)+\sum_{w\in S}\lambda\bigl(\mathcal{P}_w(E)\bigr)$. This, combined with the second part of Proposition~\ref{propchar} shows that $(\ref{eq:lambda-imp})$ reduces to the equality of $\lambda$-invariants in the statement
of the theorem.
\end{proof}

\section{Analytic side}\label{IMCII}

Let $E/\Q$ be an elliptic curve of conductor $N$, $p\nmid 2N$ a prime of good reduction for $E$, and $K$ an imaginary quadratic field satisfying hypotheses (\ref{eq:intro-Heeg}), (\ref{eq:intro-spl}), and (\ref{eq:intro-disc}) from the introduction; in particular, $p=v\bar{v}$ splits in $K$.

In this section, assuming $E[p]^{ss}=\mathbb{F}_p(\phi)\oplus\mathbb{F}_p(\psi)$ as $G_\Q$-modules, we prove an analogue of  Theorem~\ref{MAINalgside} on the analytic side, relating the Iwasawa invariants of an anticyclotomic $p$-adic $L$-function of $E$ to the Iwasawa invariants of anticyclotomic Katz $p$-adic $L$-functions attached to $\phi$ and $\psi$.

\subsection{\texorpdfstring{$p$}{p}-adic \texorpdfstring{$L$}{L}-functions}\label{subsec:Lp}

Recall that $\Lambda=\Z_p\llbracket\Gamma\rrbracket$ denotes the anticyclotomic Iwasawa algebra, and set $\Lambda^{\rm ur}=\Lambda\hat{\otimes}_{\Z_p}\Z_p^{\rm ur}$,  
for $\Z_p^{\rm ur}$ the completion of the ring of integers of the maximal unramified extension of $\Q_p$. 

We shall say that an algebraic Hecke character $\psi:K^\times\backslash\mathbb{A}_K^\times\rightarrow\mathbb{C}^\times$ has infinity type $(m,n)$ if 
the component $\psi_\infty$ of $\psi$ at $\infty$ satisfies
$\psi_\infty(z)=z^{m}\bar{z}^{n}$ for all $z\in (K\otimes\mathbb{R})^\times\simeq\mathbb{C}^\times$, 
where the last identification is made via $\iota_\infty$.

\subsubsection{The Bertolini--Darmon--Prasanna \texorpdfstring{$p$}{p}-adic \texorpdfstring{$L$}{l}-functions} 

Fix an integral ideal $\mathfrak{N}\subset\mathcal{O}_K$ with 
\begin{equation}\label{eq:N}
\mathcal{O}_K/\mathfrak{N}\simeq\Z/N\Z.
\end{equation}
Let $f\in S_2(\Gamma_0(N))$ be the newform associated with $E$. Following \cite{BDP}, one has the following result. 

\begin{theorem}\label{thm:BDP}
There exists an element $\Lcal_E\in\Lambda^{\rm ur}$ characterized by the following interpolation property: For every character $\xi$ of $\Gamma$ crystalline at both $v$ and $\bar{v}$ and corresponding to a Hecke character of $K$ of infinity type $(n,-n)$ with $n\in\Z_{>0}$ and $n\equiv 0\pmod{p-1}$, we have
\[
\mathcal{L}_E(\xi)=\frac{\Omega_p^{4n}}{\Omega_\infty^{4n}}\cdot\frac{\Gamma(n)\Gamma(n+1)\xi(\mathfrak{N}^{-1})}{4(2\pi)^{2n+1}\sqrt{D_K}^{2n-1}}\cdot\bigl(1-a_p\xi(\bar{v})p^{-1}+\xi(\bar{v})^2p^{-1}\bigr)^2\cdot L(f/K,\xi,1),
\]  
where $\Omega_p$ and $\Omega_\infty$ are CM periods attached to $K$ as in \cite[\S{2.5}]{cas-hsieh1}.
\end{theorem}

\begin{proof}
This was originally constructed in \cite{BDP} as a continuous function of $\xi$, and later explicitly constructed as a measure in \cite{cas-hsieh1} (following the approach in \cite{brakocevic}). Since this refined construction will be important for our purposes in this section, we recall some of the details. 

Let ${\rm Ig}(N)$ be the Igusa scheme over $\Z_{(p)}$ parametrizing elliptic curves with $\Gamma_1(Np^\infty)$-level structure as in \cite[\S{2.1}]{cas-hsieh1};  its complex points admit a uniformization 
\begin{equation}\label{eq:unif}
[\;,\;]:\mathfrak{H}\times{\rm GL}_2(\hat{\Q})\rightarrow{\rm Ig}(N)(\mathbb{C}).
\end{equation} 
Let $c$ be a positive integer prime to $Np$. Then $\vartheta:=(D_K+\sqrt{-D_K})/2$ and the element $\xi_c:=\varsigma^{(\infty)}\gamma_c\in\mathrm{GL}_2(\hat{\mathbb{Q}}^{})$ constructed in \cite[p.~577]{cas-hsieh1} define a point 
\[
x_c:=[(\vartheta,\xi_c)]\in\mathrm{Ig}(N)(\mathbb{C})
\]
rational over $K[c](v^\infty)$, the compositum of the ring class field $K$ of conductor $c$ and the ray class field of $K$ of conductor $v^\infty$. For every $\mathcal{O}_c$-ideal $\mathfrak{a}$ prime to $\mathfrak{N}v$, let $a\in\hat{K}^{(cp),\times}$ be such that $\mathfrak{a}=a\hat{\mathcal{O}}_c\cap K$ and set
\[
\sigma_\mathfrak{a}:=\mathrm{rec}_K(a^{-1})\vert_{K[c](v^\infty)}\in\mathrm{Gal}(K[c](v^\infty)/K),
\]
where ${\rm rec}_K:K^\times\backslash\hat{K}^\times\rightarrow G_K^{\rm ab}$ is the reciprocity map (geometrically normalized). Then by Shimura's reciprocity law, the point $x_\mathfrak{a}:=x_c^{\sigma_\mathfrak{a}}$ is defined by the pair $(\vartheta,\overline{a}^{-1}\xi_c)$ under $(\ref{eq:unif})$.  

Let $V_p(N;R)$ be the space of $p$-adic modular forms of tame level $N$ defined over a $p$-adic ring $R$ (as recalled in \cite[\S{2.2}]{cas-hsieh1}), and let $S_{\mathfrak{a}}\hookrightarrow{\rm Ig}(N)_{/\Z_p^{\rm ur}}$ be the local deformation space of $x_\mathfrak{a}\otimes\overline{\mathbb{F}}_p\in{\rm Ig}(N)(\overline{\mathbb{F}}_p)$, so we have $\mathcal{O}_{S_{\mathfrak{a}}}\simeq\Z_p^{\rm ur}\llbracket t_\mathfrak{a}-1\rrbracket$ by Serre--Tate theory. Viewing $f$ as a $p$-adic modular form, the Serre--Tate expansion 
\[
f(t_{\mathfrak{a}}):=f\vert_{S_\mathfrak{a}}\in\bZ_p^{\rm ur}\llbracket t_{\mathfrak{a}}-1\rrbracket 
\]
defines a $\Z_p^{\rm ur}$-valued measure ${\rm d}\mu_{f,\mathfrak{a}}$ on $\Z_p$ characterized (by Mahler's theorem, see e.g. \cite[\S{3.3}, Thm.~1]{hida-blue}) by
\begin{equation}\label{eq:def-meas}
\int_{\Z_p}\binom{x}{m}{\rm d}\mu_{f,{\mathfrak{a}}}=\binom{\theta}{m}f(x_{\mathfrak{a}})
\end{equation}
for all $m\geqslant 0$, 
where $\theta:V_p(N;\Z_p)\rightarrow V_p(N;\Z_p)$ is the Atkin--Serre operator, acting as $qd/dq$ on $q$-expansions. Similarly, the $p$-depletion 
\[
f^\flat=\sum_{p\nmid n}a_nq^n
\]
defines a $\Z_p^{\rm ur}$-valued measure ${\rm d}\mu_{f^\flat,\mathfrak{a}}$  on $\Z_p$ (supported on $\Z_p^\times$) with $p$-adic Mellin transform $f^\flat(t_\mathfrak{a})$, and we let ${\rm d}\mu_{f^\flat_{\mathfrak{a}}}$ be the measure on $\Z_p^\times$ corresponding to $f(t_{\mathfrak{a}}^{{\rm N}(\mathfrak{a})^{-1}\sqrt{-D_K}^{-1}})$ (see \cite[Prop.~3.3]{cas-hsieh1}).

Letting $\eta$ be an auxiliary anticyclotomic Hecke character of $K$ of infinity type $(1,-1)$ and conductor $c$, define $\mathscr{L}_{v,\eta}\in\Lambda^{\mathrm{ur}}$ by
\begin{equation}\label{eq:cusp-measure}
\mathscr{L}_{v,\eta}(\phi)=
\sum_{[\mathfrak{a}]\in\mathrm{Pic}(\mathcal{O}_c)}
\eta(\mathfrak{a})\mathbf{N}(\mathfrak{a})^{-1}
\int_{\Z_p^\times}\eta_v(\phi\vert[\mathfrak{a}])\;\mathrm{d}\mu_{f^\flat_\mathfrak{a}}
\end{equation}
where $\eta_v$ denotes the $v$-component of $\eta$, and $\phi\vert[\mathfrak{a}]:\Z_p^\times\rightarrow\mathcal{O}_{\mathbb{C}_p}^\times$ is defined by $(\phi\vert[\mathfrak{a}])(x)=\phi({\rm rec}_v(x)\sigma_{\mathfrak{a}}^{-1})$ for the local reciprocity map ${\rm rec}_v:K_v^\times\rightarrow G_K^{\rm ab}\twoheadrightarrow\Gamma$. Then by \cite[Prop.~3.8]{cas-hsieh1} the element $\mathcal{L}_E\in\Lambda^{\rm ur}$
defined by
\[
\mathcal{L}_E(\xi):=\mathscr{L}_{v,\eta}(\eta^{-1}\xi)^2
\]
has the stated interpolation property. 
\end{proof}

\subsubsection{Katz $p$-adic L-functions} 

Let $\theta:G_\Q\rightarrow\mathbb{Z}_p^\times$ be a Dirichlet character of conductor $C$. As it will suffice for our purposes, we assume that $C\vert N$ (so $p\nmid C$), and let $\mathfrak{C}\vert\mathfrak{N}$ be such that $\mathcal{O}_K/\mathfrak{C}=\Z/C\Z$. 

The next result follows from the work of Katz \cite{katz}, as extended by Hida--Tilouine \cite{hidatilouineI}.

\begin{theorem}\label{thm:Katz}
There exists an element $\mathcal{L}_{\theta}\in\Lambda^{\rm ur}$ characterized by the following interpolation property: For every character $\xi$ of $\Gamma$ crystalline at both $v$ and $\bar{v}$ and corresponding to a Hecke character of $K$ of infinity type $(n,-n)$ with $n\in\Z_{>0}$ and $n\equiv 0\pmod{p-1}$, we have
\begin{align*}
\Lcal_{\theta}(\xi)=\frac{\Omega_p^{2n}}{\Omega_\infty^{2n}}\cdot 4\Gamma(n+1)\cdot\frac{(2\pi i)^{n-1}}{\sqrt{D_K}^{n-1}}&\cdot\bigl(1-\theta^{-1}(p)\xi^{-1}(v)\bigr)\cdot\bigl(1-\theta(p)\xi(\bar{v})p^{-1})\bigr)\\
&\times\prod_{\ell\vert C}(1-\theta(\ell)\xi(w)\ell^{-1})\cdot L(\theta_{K}\xi\mathbf{N}_K,0),
\end{align*}
where $\Omega_p$ and $\Omega_\infty$ are as in Theorem~\ref{thm:BDP}, and for each $\ell\vert C$ we take the prime $w\vert\ell$ with $w\vert\mathfrak{C}$.
\end{theorem}

\begin{proof}
The character $\theta$  (viewed as a character of $K$) defines a projection 
\[
\pi_\theta:\bZ_p^{\rm ur}\llbracket{\rm Gal}(K(\mathfrak{C}p^\infty)/K)\rrbracket\rightarrow\Lambda^{\rm ur},
\]
where $K(\mathfrak{C}p^\infty)$ is the ray class field of $K$ of conductor $\mathfrak{C}p^\infty$
(this projection is just $g\mapsto \theta(g)[g]$ for $g\in {\rm Gal}(K(\mathfrak{C}p^\infty)/K$ 
and $[g]$ the image of $g$ in $\Gamma$).
The element $\Lcal_{\theta}$ is then obtained by applying $\pi_\theta$ to the Katz $p$-adic $L$-function described in \cite[Thm.~27]{kriz}, setting $\chi^{-1}=\theta_K\xi\mathbf{N}_K$. 
\end{proof}

\subsection{Comparison II: Analytic Iwasawa invariants}

The following theorem follows from the main result of \cite{kriz}. Following the notations in \emph{op.\,cit},
we let $N_0$ be the square-full part of $N$ (so the quotient $N/N_0$ is square-free), and fix an integral ideal $\mathfrak{N}\subset\mathcal{O}_K$ as in (\ref{eq:N}).

Let also $f=\sum_{n=1}^\infty a_nq^n\in S_2(\Gamma_0(N))$ be the newform associated with $E$, and denote by $\lambda^\iota$ the image of $\lambda\in\Lambda$ under the involution of $\Lambda$ given by $\gamma\mapsto\gamma^{-1}$ for $\gamma\in\Gamma$.

\begin{theorem}\label{thm:kriz}
Assume that $E[p]^{ss}\simeq\mathbb{F}_p(\phi)\oplus\mathbb{F}_p(\psi)$ as $G_\bQ$-modules,
with the characters $\phi, \psi$ labeled so that $p\nmid\mathrm{cond}(\phi)$, and suppose $\phi\neq\mathds{1}$.
Then there is a factorization $N/N_0=N_+N_-$ with
\[
\begin{cases}
a_\ell\equiv\phi(\ell)\pmod{p}&\textrm{if $\ell\vert N_+$,}\\
a_\ell\equiv\psi(\ell)\pmod{p}&\textrm{if $\ell\vert N_-$,}\\
a_\ell\equiv 0\pmod{p}&\textrm{if $\ell\vert N_0,$}
\end{cases}
\]
such that the following congruence holds
\[
\mathcal{L}_E\equiv(\mathcal{E}_{\phi,\psi}^\iota)^2\cdot(\mathcal{L}_\phi)^2\pmod{p\Lambda^{\rm ur}},
\]
where 
\[
\mathcal{E}_{\phi,\psi}=\prod_{\ell\vert N_0N_-}\mathcal{P}_w(\phi)\cdot\prod_{\ell\vert N_0N_+}\mathcal{P}_w(\psi),
\] 
and for each $\ell\vert N$ we take the prime $w\vert\ell$ with $w\vert\mathfrak{N}$.
\end{theorem}

\begin{proof} 
By Theorem~34 and Remark~32 in \cite{kriz}, our hypothesis on $E[p]$ implies that there is a congruence
\begin{equation}\label{eq:cong-mf}
f\equiv G^{}\pmod{p},
\end{equation}
where $G^{}$ is a certain weight two Eisenstein series (denoted $E_{2}^{\phi,\phi^{-1},(N)}$ in \emph{loc.\,cit.}). 
Viewed as a $p$-adic modular form, $G$ defines $\Z_p^{\rm ur}$-valued measures $\mu_{G,\mathfrak{a}}$
on $\Z_p$ by the rule (\ref{eq:def-meas}). With the notations introduced in the proof of Theorem~\ref{thm:BDP}, set
\begin{equation}\label{eq:Eis-measure}
\mathscr{L}_{v,\eta}(G,\phi)=
\sum_{[\mathfrak{a}]\in\mathrm{Pic}(\mathcal{O}_c)}
\eta(\mathfrak{a})\mathbf{N}(\mathfrak{a})^{-1}
\int_{\Z_p^\times}\eta_v(\phi\vert[\mathfrak{a}])\;\mathrm{d}\mu_{G^\flat,{\mathfrak{a}}}
\end{equation}
where the cusp form $f$ in $(\ref{eq:def-meas})$ has been replaced by $G$, and let $\mathcal{L}_G^{}\in\Lambda^{\rm ur}$ be the element defined by 
\[
\Lcal_G(\xi):=\mathscr{L}_{v,\eta}(\eta^{-1}\xi).
\]
Then for $\xi$ an arbitrary character of $\Gamma$ crystalline at both $v$ and $\bar{v}$ and corresponding to a Hecke character of $K$ of infinity type $(n,-n)$ for some $n\in\Z_{>0}$ with $n\equiv 0\pmod{p-1}$, the calculation in \cite[Prop.~37]{kriz} (taking $\chi^{-1}=\xi\mathbf{N}_K$, $\psi_1=\phi$, and $\psi_2=\phi^{-1}=\psi\omega^{-1}$ in the notations of \emph{loc.cit.}, so in particular $j=n-1$) shows that
\begin{equation}\label{eq:kriz-prop37}
\Lcal_G(\xi)=\frac{\Omega_p^{2n}}{\Omega_\infty^{2n}}\cdot\frac{\Gamma(n+1)\phi^{-1}(-\sqrt{D_K})\xi(\bar{\mathfrak{t}})t}{\mathfrak{g}(\phi)}\cdot\frac{(2\pi i)^{n-1}}{\sqrt{D_K}^{n-1}}\cdot
\Xi_{\xi^{-1}\mathbf{N}_K^{-1}}(\phi,\psi\omega^{-1},N_+,N_-,N_0)\cdot L(\phi_K\xi\mathbf{N}_K,0),
\end{equation}
where $\phi_K$ denotes the base change of $\phi$ to $K$, and
\begin{align*}
\Xi_{\xi^{-1}\mathbf{N}_K^{-1}}(\phi,\psi\omega^{-1},N_+,N_-,N_0)=\prod_{\ell\vert N_+}&(1-\phi^{-1}\xi(\bar{w}))\cdot\prod_{\ell\vert N_-}(1-\phi\xi(\bar{w})\ell^{-1})\\
&\times\prod_{\ell\vert N_0}(1-\phi^{-1}\xi(\bar{w}))(1-\phi\xi(\bar{w})\ell^{-1}).
\end{align*}
Comparing with the interpolation property of $\mathcal{L}_\phi$ in Theorem~\ref{thm:kriz}, and noting that
\[
\mathcal{E}_{\phi,\psi}(\xi^{-1})=\Xi_{\xi^{-1}\mathbf{N}_K^{-1}}(\phi,\psi\omega^{-1},N_+,N_-,N_0)
\]
for all $\xi$ as above, the equality $(\ref{eq:kriz-prop37})$ implies that
\begin{equation}\label{eq:cor-37}
\Lcal_G = \mathcal{E}_{\phi,\psi}^\iota\cdot\Lcal_\phi.
\end{equation}

On the other hand, the congruence $(\ref{eq:cong-mf})$ implies the congruences
\[
\binom{\theta}{m}f(x_{\mathfrak{a}})\equiv
\binom{\theta}{m}G(x_{\mathfrak{a}})\pmod{p\Z_p^{\rm ur}}
\]
for all $m\geqslant 0$, which in turn yield the congruence
\begin{equation}\label{eq:cor-meas}
\Lcal_E\equiv(\Lcal_G)^2\pmod{p\Lambda^{\rm ur}}.
\end{equation}
The combination of (\ref{eq:cor-37}) and (\ref{eq:cor-meas}) now yields the theorem.
\end{proof}

\begin{theorem}\label{cor:Kriz}
Assume that $E[p]^{ss}=\mathbb{F}_p(\phi)\oplus\mathbb{F}_p(\psi)$ as $G_\bQ$-modules, with the characters $\phi, \psi$ labeled so that $p\nmid\mathrm{cond}(\phi)$, and suppose $\phi\neq\mathds{1}$. Then $\mu(\Lcal_E)=0$ and 
\[
\lambda(\Lcal_E)=\lambda(\Lcal_\phi)+\lambda(\Lcal_\psi)+\sum_{w\in S}\bigl\{\lambda(\mathcal{P}_w(\phi))+\lambda(\mathcal{P}_w(\psi))-\lambda(\mathcal{P}_w(E))\bigr\}.
\]
\end{theorem}

\begin{proof}
Since $K$ satisfies (\ref{eq:intro-Heeg}) and (\ref{eq:intro-spl}), the conductors of both $\phi$ and $\psi$ are only divisible by primes split in $K$, and hence the vanishing of $\mu(\mathcal{L}_E)$ follows immediately from the congruence of Theorem~\ref{thm:kriz} and Hida's result \cite{hidamu=0} (note that the factors $\mathcal{P}_w(\phi)$ and $\mathcal{P}_w(\psi)$ also have vanishing $\mu$-invariant, since again the primes $w$ are split in $K$). 

As for the equality between $\lambda$-invariants, note that the involution of $\Lambda$ given by $\gamma\mapsto\gamma^{-1}$ for $\gamma\in\Gamma$
preserves $\lambda$-invariants, and so
\[
\lambda(\mathcal{P}_w(\theta)^2)=\lambda(\mathcal{P}_w(\theta))+\lambda(\mathcal{P}_{\bar{w}}(\theta)),
\]
using that complex conjugation acts as inversion on $\Gamma$. For the term $\mathcal{E}_{\phi,\psi}$ in Theorem~\ref{thm:kriz} we thus have
\[
\lambda((\mathcal{E}_{\phi,\psi}^\iota)^2)=
\lambda((\mathcal{E}_{\phi,\psi})^2)=\sum_{w\vert N_0N_-}\lambda(\mathcal{P}_w(\phi))+\sum_{w\vert N_0N_+}\lambda(\mathcal{P}_w(\psi)),
\]
where $w$ runs over all divisors, not just the ones dividing $\mathfrak{N}$. Using the congruence relations in Theorem~\ref{thm:kriz} (in particular, that $a_\ell\equiv 0\pmod{p}$ for $\ell\vert N_0$) this can be rewritten as
\begin{equation}\label{eq:Euler-comp}
\lambda((\mathcal{E}_{\phi,\psi}^\iota)^2)=\sum_{w\in S}\bigl\{\lambda(\mathcal{P}_w(\phi))+\lambda(\mathcal{P}_w(\psi))-\lambda(\mathcal{P}_w(E))\bigr\}.
\end{equation}
On the other hand, since $\psi=\phi^{-1}\omega$, the functional equation for the Katz $p$-adic $L$-function (see e.g. \cite[Thm.~27]{kriz}) yields
\begin{equation}\label{eq:functional}
\lambda(\mathcal{L}_\psi)=\lambda(\mathcal{L}_\phi).
\end{equation}
The result now follows from Theorem~\ref{thm:kriz} combined with $(\ref{eq:Euler-comp})$ and $(\ref{eq:functional})$.
\end{proof}

Together with the main result of $\S\ref{sec:algebraic}$, we arrive at the following.

\begin{theorem}\label{mulambda} 
Assume that $E[p]^{ss}=\mathbb{F}_p(\phi)\oplus\mathbb{F}_p(\psi)$ with $\phi\vert_{G_p}\neq\mathds{1},\omega$. Then $\mu(\mathcal{L}_E)=\mu(\mathfrak{X}_E)=0$ and 
\[
\lambda(\Lcal_E)=\lambda(\mathfrak{X}_E).
\]
\end{theorem}

\begin{proof}
The vanishing of $\mu(\mathfrak{X}_E)$ (resp. $\mu(\mathcal{L}_E)$) has been shown in Proposition~\ref{propmodp} (resp. Theorem~\ref{cor:Kriz}). On the other hand, 
Iwasawa's main conjecture for $K$ (a theorem of Rubin \cite{rubinmainconj}) yields in particular the equalities
$\lambda(\Lcal_\phi)=\lambda(\mathfrak{X}_\phi)$ and  
$\lambda(\Lcal_\psi)=\lambda(\mathfrak{X}_\psi)$. The combination of Theorem~\ref{MAINalgside} and Theorem~\ref{cor:Kriz} therefore yields the result.
\end{proof}

\section{A Kolyvagin system argument}\label{IMCIII}

The goal of this section is to prove Theorem~\ref{thm:howard} below, extending \cite[Thm.~2.2.10]{howard} to the residually reducible setting. This result, which assumes the existence of a non-trivial Kolyvagin system, will be applied in $\S\ref{sec:CD}$ to a Kolyvagin system derived from Heegner points to prove one of the divisibilities towards Conjecture~\ref{conj:PR}.

\subsection{Selmer structures and Kolyvagin systems}

Let $K$ be an imaginary quadratic field, let $(R,\mathfrak{m})$ be a complete Noetherian local ring with finite residue field of characteristic $p$, and let 
$M$ be a topological $R[G_K]$-module such that the $G_K$-action is unramified outside a finite set of primes. We define a \emph{Selmer structure} $\mathcal{F}$ on $M$ to be a finite set $\Sigma=\Sigma(\mathcal{F})$ of places of $K$ containing $\infty$, the primes above $p$, and the primes where $M$ is ramified, together with a choice of $R$-submodules (called local conditions) 
$\rH^1_{\Fcal}(K_w,M)\subset\rH^1(K_w,M)$ for every $w\in\Sigma$. The associated \emph{Selmer group} is then defined by
\[
\rH^1_{\Fcal}(K,M):={\rm ker}\biggl\{\rH^1(K^\Sigma/K,M)\rightarrow\prod_{w\in\Sigma}\frac{\rH^1(K_w,M)}{\rH^1_{\Fcal}(K_w,M)}\biggr\},
\]
where $K^\Sigma$ is the maximal extension of $K$ unramified outside $\Sigma$. 

Below we shall use the following local conditions. First, the \emph{unramified} local condition is
\[
\rH^1_{\rm ur}(K_w,M):={\rm ker}\bigl\{\rH^1(K_w,M)\rightarrow\rH^1(K_w^{\rm ur},M)\bigr\}.
\]
If $w\mid p$ is a finite prime where $M$ is unramified, we set $\rH^1_{\rm f}(K_w,M):=\rH^1_{\rm ur}(K_w,M)$, which is sometimes called the \emph{finite} local condition. 
The \emph{singular quotient} $\rH^1_{\rm s}(K,M)$ is defined by the exactness of the sequence
\[
0\rightarrow\rH^1_{\rm f}(K_w,M)\rightarrow\rH^1(K_w,M)\rightarrow\rH^1_{\rm s}(K_w,M)\rightarrow 0.
\]

Denote by $\cL_0=\cL_0(M)$ the set of rational primes $\ell\neq p$ such that
\begin{itemize}
\item $\ell$ is inert in $K$,
\item $M$ is unramified at $\ell$.
\end{itemize}
Letting $K[\ell]$ be the ring class field of $K$ of conductor $\ell$, define the \emph{transverse} local condition at $\lambda\vert\ell\in\cL_0$ by
\[
\rH^1_{\rm tr}(K_\lambda,T):={\rm ker}\bigl\{\rH^1(K_\lambda,T)\rightarrow\rH^1(K[\ell]_{\lambda'},T)\bigr\},
\]
where $K[\ell]_{\lambda'}$ is the completion of $K[\ell]$ at any prime $\lambda'$ above $\lambda$.

As in \cite{howard}, we call a \emph{Selmer triple} $(M,\mathcal{F},\cL)$ the data of a Selmer structure $\mathcal{F}$ on $M$ and a subset $\cL\subset\cL_0$ with $\cL\cap\Sigma(\mathcal{F})=\emptyset$. Given a Selmer triple $(M,\Fcal,\cL)$ and pairwise
coprime integers $a,b,c$ divisible only by primes in $\cL_0$, the modified Selmer group $\rH^1_{\Fcal^a_b(c)}(K,M)$ is the one defined by 
$\Sigma(\Fcal^a_b(c))=\Sigma(\Fcal)\cup\{w\vert abc\}$
and the local conditions 
\[
\rH^1_{\Fcal^a_b(c)}(K_\lambda,T)
=\begin{cases}
\rH^1(K_\lambda,T) &\text{if }\lambda\vert a,\\
0 &\text{if }\lambda\vert b,\\
\rH^1_{\tr}(K_\lambda,T) &\text{if }\lambda\vert c, \\
\rH^1_{\Fcal}(K_w,T) & \text{if }\lambda\nmid abc.
\end{cases}
\]

Let $T$ be a compact $R$-module equipped with a continuous linear $G_K$-action 
that is unramified outside a finitely set of primes.
For each $\lambda\vert\ell\in\cL_0 = \cL_0(T)$, let $I_\ell$ be the smallest ideal containing $\ell+1$ for which the Frobenius element ${\rm Frob}_\lambda\in G_{K_\lambda}$ acts trivially on $T/I_\ell T$. By class field theory, the prime $\lambda$ splits completely in the Hilbert class field of $K$, and the $p$-Sylow subgroups of $G_\ell:={\rm Gal}(K[\ell]/K[1])$ and $k_\lambda^\times/\mathbb{F}_\ell^\times$ are identified via the Artin symbol, where $k_\lambda$ is the residue field of $\lambda$. Hence by \cite[Lem.~1.2.1]{mazrub} there is a \emph{finite-singular comparison isomorphism}
\begin{equation}\label{eq:f-s}
\phi_\lambda^{\rm fs}:\rH^1_{\rm f}(K_\lambda,T/I_\ell T)\simeq T/I_\ell T\simeq\rH^1_{\rm s}(K_\lambda,T/I_\ell T)\otimes G_\ell.
\end{equation}

Given a subset $\cL\subset\cL_0$, we let $\cN=\cN(\cL)$ be the set of square-free products of primes $\ell\in\cL$, and for each $n\in\cN$ define
\[
I_n=\sum_{\ell\vert n}I_n\subset R,\quad G_n=\bigotimes_{\ell\vert n}G_\ell,
\]
with the convention that $1\in\cN$, $I_1=0$, and $G_1=\Z$.

\begin{definition}
A \emph{Kolyvagin system} for a Selmer triple $(T,\Fcal,\cL)$ is a collection of classes
\[
\kappa=\{\kappa_n\in\rH^1_{\Fcal(n)}(K,T/I_nT)\otimes G_n\}_{n\in\cN}
\]
such that 
$(\phi_\lambda^{\rm fs}\otimes 1)({\rm loc}_\lambda(\kappa_n))={\rm loc}_\lambda(\kappa_{n\ell})$ for all $n\ell\in\cN$.
\end{definition}

We denote by $\mathbf{KS}(T,\Fcal,\cL)$ the $R$-module of Kolyvagin systems for $(T,\Fcal,\cL)$.

\subsection{Bounding Selmer groups}\label{sec:Zp-twisted} 

Here we state our main result on bounding Selmer groups of anticyclotomic twists of Tate modules of elliptic curves, whose proof is given in the next section. The reader mostly interested in the Iwasawa-theoretic consequences of this result might wish to proceed to $\S\ref{subsec:Iw-proof}$ after reading the statement of Theorem~\ref{thm:Zp-twisted}.

Let $E/\Q$ be an elliptic curve of conductor $N$, let $p\nmid 2N$ be a prime of good ordinary reduction for $E$, and let $K$ be an imaginary quadratic field of discriminant $D_K$  prime to $Np$. We assume 
\begin{equation}\label{eq:h1}
E(K)[p]=0.\tag{h1}
\end{equation}

As  before, let $\Gamma={\rm Gal}(K_\infty/K)$ be the Galois group of the anticyclotomic $\Z_p$-extension of $K$.
Let $\alpha:\Gamma\rightarrow R^\times$ be a character with values in the ring of integers $R$ of a finite extension $\Phi/\Q_p$. Let 
\[
r=\rank_{\Z_p}R.
\] 
Let $\rho_E:G_{\bQ}\rightarrow{\rm Aut}_{\bZ_p}(T_pE)$ give the action of $G_\bQ$ on the $p$-adic Tate module of $E$ and consider the $G_K$-modules
\begin{equation}
T_\alpha:=T_pE\otimes_{\bZ_p}R(\alpha),\quad
V_\alpha:=T_\alpha\otimes_R\Phi,\quad A_\alpha:=T_\alpha\otimes_R\Phi/R\simeq V_\alpha/T_\alpha,\nonumber
\end{equation}
where $R(\alpha)$ is the free $R$-module of rank one on which $G_K$ acts the projection $G_K\twoheadrightarrow\Gamma$ composed with $\alpha$, and the $G_K$-action on $T_\alpha$ is via $\rho_\alpha = \rho_E\otimes\alpha$. 

Let $\fm\subset R$ be the maximal ideal, 
with uniformizer $\pi\in\fm$, 
and let $\bar{T}:=T_\alpha\otimes_{} R/\fm$ be the residual representation associated to $T_\alpha$. Note that
\begin{equation}\label{eq:modp}
\bar{T}\simeq E[p]\otimes_{}R/\fm
\end{equation}
as $G_K$-modules, since $\alpha\equiv 1\pmod{\fm}$. In particular, (\ref{eq:h1}) implies that $\bar{T}^{G_K}=0$. 

For $w\vert p$ a prime of $K$ above $p$,  set 
\[
{\rm Fil}_w^+(T_pE):={\rm ker}\bigl\{T_pE\rightarrow T_p\tilde{E}\bigr\},
\]
where $\tilde{E}$ is the reduction of $E$ at $w$, and put
\[
{\rm Fil}_w^+(T_\alpha):={\rm Fil}_w^+(T_pE)\otimes_{\bZ_p}R(\alpha),\quad{\rm Fil}_w^+(V_\alpha):={\rm Fil}_w^+(T_\alpha)\otimes_{R}\Phi.
\]
Following \cite{coates-greenberg},  define the \emph{ordinary} Selmer structure $\Fcal_{\rm ord}$ on $V_\alpha$ by
taking $\Sigma(\Fcal_{\rm \ord}) = \{w\mid pN\}$ and  
\[
\rH^1_{\Fcal_{\rm ord}}(K_w,V_\alpha):=
\begin{cases}
{\rm im}\bigl\{\rH^1(K_w,{\rm Fil}_w^+(V_\alpha))\rightarrow\rH^1(K_w,V_\alpha)\bigr\} & \textrm{if $w\vert p$,}\\
\rH^1_{\rm ur}(K_w,V_\alpha)& \textrm{else.}
\end{cases}
\]
Let $\Fcal_{\rm ord}$ also denote the Selmer structure on $T_\alpha$ and $A_\alpha$ obtained by propagating $\rH^1_{\Fcal_{\rm ord}}(K_w,V_\alpha)$ under the maps induced by the exact sequence 
$0\rightarrow T_\alpha\rightarrow V_\alpha\rightarrow A_\alpha\rightarrow 0$.

Let $\gamma\in\Gamma$ be a topological generator, and let 
\begin{equation}\label{eq:erralpha}
C_\alpha:=
\begin{cases}
v_p(\alpha(\gamma)-\alpha^{-1}(\gamma)) & \alpha\neq \alpha^{-1}, \\
0 & \alpha = \alpha^{-1},
\end{cases}
\end{equation}
where $v_p$ is the $p$-adic valuation normalized so that $v_p(p)=1$. Finally, let 
\[
\cL_E:=\{\ell\in\cL_0(T_pE)\;:\;a_\ell\equiv\ell+1\equiv 0\;({\rm mod}\;p)\},
\]
where $a_\ell=\ell+1-\vert\tilde{E}(\mathbb{F}_\ell)\vert$, and $\cN=\cN(\cL_E)$. 

\begin{theorem}\label{thm:Zp-twisted}
Suppose $\alpha\neq 1$ and there is a Kolyvagin system $\kappa_\alpha=\{\kappa_{\alpha,n}\}_{n\in\cN}\in\mathbf{KS}(T_\alpha,\Fcal_{\rm ord},\cL_E)$ with $\kappa_{\alpha,1}\neq 0$. Then $\rH^1_{\Fcal_{\rm ord}}(K,T_\alpha)$ has rank one, and there is a finite $R$-module $M_\alpha$ such that
\[
{\rm H}^1_{\Fcal_{\rm ord}}(K,A_\alpha)\simeq(\Phi/R)\oplus M_\alpha\oplus M_\alpha
\]
with 
\[
{\rm length}_R(M_\alpha)\leqslant {\rm length}_R\bigl({\rm H}^1_{\Fcal_{\rm ord}}(K,T_\alpha)/R\cdot\kappa_{\alpha,1}\bigr)+E_\alpha
\]
for some constant $E_\alpha\in\Z_{\geqslant 0}$ depending only on  $C_{\alpha}$, $T_pE$, and ${\rm rank}_{\bZ_p}(R)$.
\end{theorem}

When $\rho_E\vert_{G_K}:G_K\rightarrow{\rm End}_{\Z_p}(T_pE)$ is surjective, Theorem~\ref{thm:Zp-twisted} (with $E_\alpha=0$) can be deduced from \cite[Thm.~1.6.1]{howard}, but the proof of Theorem~\ref{thm:Zp-twisted} assuming only (\ref{eq:h1}) requires new ideas, some of which were inspired by Nekov\'{a}\v{r}'s work \cite{nekovar}.

\subsection{Proof of Theorem~\ref{thm:Zp-twisted}}

To ease notation, let $(T,\mathcal{F},\cL)$ denote the Selmer triple $(T_\alpha,\Fcal_{\rm ord},\cL_E)$, and let $\rho=\rho_\alpha$. For any $k\geqslant 0$, let 
\[
R^{(k)}=R/\fm^kR,\quad T^{(k)}=T/\fm^kT,\quad\cL^{(k)}=\{\ell\in\cL\colon I_\ell\subset p^k\Z_p\},
\]
and let $\mathscr{N}^{(k)}$ be the set of square-free products of primes $\ell\in\cL^{(k)}$. 

We begin by recalling two preliminary results from \cite{mazrub} and \cite{howard}.

\begin{lemma}\label{lemmamod} 
For every $n\in\cN^{(k)}$ and $0\leqslant i\leqslant k$ there are natural isomorphisms
\begin{equation}\label{eq2}
\rH^1_{\Fcal(n)}(K,T^{(k)}/\fm^iT^{(k)})\xrightarrow{\sim}\rH^1_{\Fcal(n)}(K,T^{(k)}[\fm^i])\xrightarrow{\sim}{\rm H}_{\Fcal(n)}^1(K,T^{(k)})[\fm^i]\nonumber
\end{equation}
induced by the maps $T^{(k)}/\fm^iT^{(k)}\xrightarrow{\pi^{k-i}}T^{(k)}[\fm^i]\rightarrow T^{(k)}$.
\end{lemma}
\begin{proof}
The proof of \cite[Lem.~3.5.4]{mazrub} carries over, since it only requires the vanishing of $\bar{T}^{G_K}$.
\end{proof}

\begin{prop}\label{propstructure}
For every $n\in\cN^{(k)}$ there is an $R^{(k)}$-module $M^{(k)}(n)$  and an integer $\epsilon$ such that
\begin{equation}\label{eq1}
\rH^1_{\CF(n)}(K,T^{(k)}) \simeq (R^{(k)})^\epsilon\oplus M^{(k)}(n)\oplus M^{(k)}(n).\nonumber
\end{equation}
Moreover, $\epsilon$ can be taken to be $\epsilon\in\{0,1\}$ and is independent of $k$ and $n$.
\end{prop}

\begin{proof}
This is shown in \cite[Prop.~1.5.5]{howard}, whose proof makes use of hypothesis (\ref{eq:h1}) and hypotheses (H.3) and (H.4) in \emph{op.\,cit.}, the latter two being satisfied in our setting by \cite[Lem.~3.7.1]{mazrub} and \cite[Lem.~2.2.1]{mazrub}, respectively. We note that the independence of $\epsilon$ follows from the fact that, by Lemma~\ref{lemmamod}, we have
\[
\epsilon\equiv{\rm dim}_{R/\mathfrak{m}}\rH^1_{\CF(n)}(K,\bar{T})\pmod{2},
\]
and the right dimension is independent of $k$ and $n$ by the ``parity lemma'' of \cite[Lem.~1.5.3]{howard}, whose proof is also given under just the aforementioned hypotheses.
\end{proof}

\subsubsection{The \v{C}ebotarev argument}
For any finitely-generated torsion $R$-module $M$ and $x\in M$, write 
\[
\ord(x):=\min\{m\geqslant 0: \pi^m\cdot x =0\}.
\]

When $\rho_E$ has large image, a standard application of the \v{C}ebotarev density theorem can be used to show that, given $R$-linearly independent classes $c_1,\dots,c_s\in\rH^1(K,T^{(k)})$, there exist infinitely many primes $\ell\in\cL$ such 
that ${\rm ord}({\rm loc}_\ell(c_i))={\rm ord}(c_i)$, $i=1,\dots,s$ (see \cite[Cor.~3.2]{mccallum}). 
Assuming only hypothesis (\ref{eq:h1}), one can obtain a similar result with ``error terms''.  Our version of this is 
Proposition~\ref{prop:prime2} below, which provides the key technical input for our proof of Theorem~\ref{thm:Zp-twisted}. 
Before proving this proposition we define the error terms that appear in its statement.

For any field $F\subset\overline{\Q}$ let $F(E[p^\infty])$ be the fixed field of the kernel of $\rho_E|_{G_F}$. 
Since $(D_K,Np)=1$, and therefore $E$ does not have CM by $K$, and $p$ is odd by hypothesis, $\Q(E[p^\infty])\cap K_\infty = \Q$, 
as any subfield of $K_\infty$ that is Galois over $\Q$ is either $\Q$ or contains $K$.
Hence the natural projection 
$\Gal(K_\infty(E[p^\infty])/K_\infty)\rightarrow \Gal(\Q(E[p^\infty])/\Q)$ is an isomorphism and 
so $\rho_E(G_{K_\infty}) = \rho_E(G_\Q)$. 

The first error term comes from the following.

\begin{lemma}\label{lem:image1} 
The intersection $U = \Z_p^\times\cap \mathrm{im}(\rho_E|_{G_{K_\infty}})$ is an open subgroup of $\Z_p^\times$
such that $U\subset \mathrm{Im}(\rho)\subset \Aut_R(T)$ for all characters $\alpha$. 
\end{lemma}

\begin{proof}  
By 
\cite[Prop.~(6.1.1)(i)]{nekovar}, $U=\Z_p^\times\cap \mathrm{im}(\rho_E)\subset \Aut_{\Z_p}(T_pE)\simeq \GL_2(\Z_p)$ 
is an open subgroup of $\Z_p^\times$. Since 
$\mathrm{im}(\rho_E|_{G_{K_\infty}}) =  \mathrm{im}(\rho_E)$, 
$U = \Z_p^\times\cap \mathrm{im}(\rho_E|_{G_{K_\infty}})$. 
As $\alpha$ is trivial on $G_{K_\infty}$ the claim for all characters $\alpha$ follows.
\end{proof}

For $U = \Z_p^\times\cap \mathrm{im}(\rho_E)$ as in Lemma \ref{lem:image1}, let
\[
C_1 := \min\{v_p(u-1)\colon u\in U\}.
\]
Since $U$ is an open subgroup of $\Z_p^\times$, $0\leqslant C_1 < \infty$.

To define the second error term, note that $\End_{\Z_p}(T_pE)/\rho_E(\Z_p[G_{\Q}])$ is a torsion $\Z_p$-module, as $\rho_E$ is irreducible. Hence there exists $m\in\Z_{\geq 0}$ such that $p^m (\End_{\Z_p}(T_pE)/\rho_E(\Z_p[G_{\Q}]))=0$. Then
\[
C_2:=\min\bigl\{ m\geqslant 0 \colon p^m\End_{\Z_p}(T_pE)\subset\rho_E(\Z_p[G_{\Q}])\bigr\}
\]
is such that $0\leqslant C_2 <\infty$. 

\begin{lemma}\label{lem:image2} For any $\alpha$,  $p^{C_2}$ annihilates
$\End_R(T)/\rho(R[G_{K_\infty}])$. 
\end{lemma}

\begin{proof}  
Since $\rho(G_{K_\infty}) = \rho_E(G_{K_\infty}) = \rho_E(G_\Q)$ (using that $E$ does not have CM by $K$), it follows that
$$
\End_R(T)/\rho(R[G_{K_\infty}]) = \End_R(T)/\rho_E(R[G_\Q]) = (\End_{\Z_p}(T_pE)/\rho_E(\Z_p[G_\Q]))\otimes_{\Z_p}R
$$
is annihilated by $p^{C_2}$.
\end{proof}

\begin{remark}\label{rmkirred} 
If $\rho_E$ is surjective, then clearly $C_1 = 0$. Similarly, if $E[p]$ is irreducible, then $C_2=0$. In particular, if $\rho_E$ is surjective, then $C_1 = 0 = C_2$.
\end{remark}

The third error term is given by the quantity $C_\alpha$ defined before.

\begin{prop}\label{prop:prime2} Suppose $\alpha\neq 1$.
Let $c_1, c_2, c_3 \in \rH^1(K,T^{(k)})$. Suppose $Rc_1 + Rc_2$ contains a submodule isomorphic to 
$\fm^{d_1}R^{(k)}\oplus \fm^{d_2}R^{(k)}$ for some $d_1,d_2\geqslant 0$.
Then there exist infinitely many primes $\ell\in \CL^{(k)}$ such that 
$$
\ord(\loc_\ell(c_3)) \geqslant \ord(c_3) - r(C_1+C_2+C_\alpha),  
$$
and
$R\loc_\ell(c_1) + R \loc_\ell(c_2) \subset \rH^1(K_\ell,T^{(k)})$ contains a
submodule isomorphic to 
$$
\fm^{d_1+d_2+2r(C_1+C_2+C_\alpha)}(R^{(k)}\oplus R^{(k)}).
$$
\end{prop}

\begin{proof} 
Let $m_i = \max\{0,\ord(c_i) - r(C_1+C_2+C_{\alpha})\}$.
Note that since $R c_1 + R c_2$ contains a submodule isomorphic to $\fm^{d_1}R^{(k)}\oplus \fm^{d_2}R^{(k)}$,
it must be that $\max\{\ord(c_1),\ord(c_2)\}\geqslant k-d_1,k-d_2$ and hence 
if $m_1=m_2=m_3=0$,  
then the lemma is trivially true. So we suppose $\max\{m_1,m_2,m_3\}>0$.

Let $K_{\alpha}\subset K_{\infty}$ be such that $\alpha|_{G_{K_\alpha}}\equiv 1 \mod \fm^k$.
Let $L=K_{\alpha}(E[p^k])$ be the fixed field of the kernel of the action of $G_{K_\alpha}$ on 
$E[p^k]$ (so in particular, $G_L$ acts trivially on $T^{(k)}$). 
Then $\rho$ induces an injection
\[
\rho: \Gal(L/K)\hookrightarrow \Aut(T^{(k)}).
\] 
Let $u\in \Z_p^\times\cap\mathrm{im}(\rho_E|_{G_{K_\infty}})$ such that $\ord_p(u-1) = C_1$. 
Then $u = \rho(g)$ for some $g\in \Gal(L/K)$. Let $T_{E}^{(k)} = T_pE\otimes_{\Z_p} R/\fm^k$. 
It follows from Sah's lemma that $g-1$ annihilates $\rH^1(\Gal(L/K), T^{(k)})$, and therefore the kernel of the restriction map
$$
\rH^1(K,T^{(k)})\rightarrow \rH^1(L,T^{(k)})= \rH^1(L,T_{E}^{(k)})^{(\alpha)}=\Hom(G_L,T_{E}^{(k)})^{(\alpha)}
$$
is annihilated by $p^{C_1}$ and hence by $\pi^{rC_1}$ (cf.\cite[Prop.~(6.1.2)]{nekovar}). 
Here and in the following we denote by $(-)^{(\alpha)}$ the submodule on which $\Gal(L/K)$ acts via the character $\alpha$.
The restriction of the $c_i$ to $G_L$ therefore yields homomorphisms $f_i \in \Hom(G_L,T_E^{(k)})^{(\alpha)}$ such that 
$$
\ord(f_i) \geqslant \ord(c_i) - rC_1, i=1,2,3,
$$
and $Rf_1+Rf_2\subset \Hom(G_L,T_E^{(k)})^{(\alpha)}$ contains a submodule isomorphic to 
$\fm^{d_1+rC_1}R^{(k)}\oplus \fm^{d_2+rC_1}R^{(k)}$.

Note that the complex conjugation $\tau$ acts naturally on $\Hom(G_L,T_E^{(k)})$, 
and that this action maps an element $f\in\Hom(G_L,T_E^{(k)})^{(\alpha)}$ to an element $\tau\cdot f\in \Hom(G_L,T_E^{(k)})^{(\alpha^{-1})}$. 
The intersection $\Hom(G_L,T_E^{(k)})^{(\alpha)}\cap \Hom(G_L,T_E^{(k)})^{(\alpha^{-1})}$
is annihilated by $\gamma-\alpha^{\pm 1}(\gamma)$ and so by $\alpha(\gamma)-\alpha^{-1}(\gamma)$, for all $\gamma\in\Gal(L/K)$.
Since 
\[
\{\alpha(\gamma)-\alpha^{-1}(\gamma)\;{\rm mod}\;\fm^k \ : \ \gamma\in\Gal(L/K)\} = \{\alpha(\gamma)-\alpha^{-1}(\gamma)\;{\rm mod}\;\fm^k \ : \ \gamma\in\Gamma\}
\] 
and since
$\alpha\neq \alpha^{-1}$ (as $\alpha\neq 1$ and $p$ is odd),
it follows from the definition of $C_\alpha$ that  $\Hom(G_L,T_E^{(k)})^{(\alpha)}\cap \Hom(G_L,T_E^{(k)})^{(\alpha^{-1})}$ is annihilated by $\pi^{rC_\alpha}$. This implies that
$f_i^\pm = (1\pm \tau)\cdot f_i$ satisfies
$$
\ord(f_i^\pm) \geqslant \ord(f_i) - rC_\alpha\geqslant{\ord(c_i)-r(C_1+C_\alpha)}, \ \ i=1,2,3,
$$
and that $Rf_1^\pm +R f_2^\pm = (1\pm\tau)\cdot(Rf_1+Rf_2)$ contains a submodule isomorphic to 
$\fm^{d_1+r(C_1+C_\alpha)}R^{(k)}\oplus \fm^{d_2+r(C_1+C_\alpha)}R^{(k)}$.
Note that since $\max\{m_1,m_2,m_3\}>0$, it follows that for some $j$ both $f_j^+$ and $f_j^-$ are non-zero.

The $R$-module spanned by the image of $f_i^\pm$ contains $R[G_{K_\infty}]\cdot f_i^\pm(G_L)$. 
By Lemma \ref{lem:image2}, the latter contains $p^{C_2}(\End_{\Z_p}(T_pE)\otimes_{\Z_p}R)\cdot f_i^\pm(G_L) \subset T_E^{(k)}$. 
Since $f_i^\pm$ has order at least $\ord(f_i) - rC_\alpha$, $f_i^\pm(G_L)$ contains an element of order at least 
$\ord(f_i) - rC_\alpha$ and hence $\pi^{k-\ord(f_i) + r(C_2+C_\alpha)} T_E^{(k)}\subset 
p^{C_2}(\End_{\Z_p}(T_pE)\otimes_{\Z_p}R)\cdot f_i^\pm(G_L)$.
In particular, the $R$-module spanned by the image of $f_i^\pm$ contains 
$\fm^{k-m_i} T_E^{(k)}$.

Let $H \subset G_L$ be the intersection of the kernels of the $f_i^\pm$. Since some $f_j^\pm$ is non-zero, $H\neq G_L$
and $Z = G_L/H$ is a non-zero torsion $\Z_p$-module. The subgroup $H$ is stable under the action of complex conjugation and hence this action descends to $Z$, which then decomposes as $Z = Z^+ \oplus Z^-$ with respect to this action.
Each $f_i^\pm$ can be viewed as an element of $\Hom(Z,T_E^{(k)})$.
Let $g_i^\pm$ be the composition of $f_i^\pm$ with the projection of $T_E^{(k)}$ to $(T_E^{(k)})^\pm$.
Fix an $R^{(k)}$-basis $u_\pm$ of $(T_E^{(k)})^\pm$. Since the $R$-span of the image of $f_i^\pm$ contains
$\fm^{k-m_i} T_E^{(k)}$, the $R$-span of the image of $g_i^\pm$ contains $\fm^{k-m_i}R^{(k)} u_\pm$.
Moreover, since $f_i^\pm \in \Hom(Z,T_E^{(k)})^\pm$, $g^\pm_i(Z^-)=0$ and so $g_i^\pm(Z) = g_i^\pm(Z^+)$.
Since $\max\{m_1,m_2,m_3\}>0$, it follows that $Z^+$ is nontrivial.

Let $W^\pm = \sum_{i=1}^3 Rf_i^\pm \subset \Hom(G_L,T_E^{(k)})^\pm$ and let
$W = W^+\oplus W^- \subset \Hom(G_L,T_E^{(k)})$. 
Each $f\in W$ can be viewed as a homomorphism from 
$Z$ to $T_E^{(k)}$, and evaluation at $z\in Z$ yields an injection
$$
Z\hookrightarrow \Hom_R(W,T_E^{(k)}).
$$
Furthermore, this injection is equivariant with respect to the action of complex conjugation, so 
the restriction to $Z^+$ is an injection
$$
Z^+ \hookrightarrow \Hom_R(W,T_E^{(k)})^+ = \Hom_R(W^+,(T_E^{(k)})^+)\oplus \Hom_R(W^-,(T_E^{(k)})^-).
$$
Let $X^+\subset \Hom_R(W,T_E^{(k)})^+$ be the $R$-span of the image of $Z^+$. It follows
from \cite[Cor.~(6.3.4)]{nekovar} and Lemma \ref{lem:image2} that 
\begin{equation}\label{eq:spanimage}
p^{C_2} \Hom_R(W,T_E^{(k)})^+\subset X^+.
\end{equation}

Given $(\phi^+,\phi^-) \in \Hom_R(W^+,(T_E^{(k)})^+)\oplus \Hom_R(W^-,(T_E^{(k)})^-)$, define
$$
q(\phi^+,\phi^-) = \det\left(\begin{matrix}
\beta(\phi^+(f_1^+)) & \beta(\phi^-(f_1^-)) \\
\beta(\phi^+(f_2^+)) & \beta(\phi^-(f_2^-))
\end{matrix}\right), \ \ 
\phi^\pm(-) = \beta(\phi^\pm(-))u_{\pm}\in R^{(k)}u_\pm = (T_E^{(k)})^\pm.
$$
The restriction of $q$ to $X^+$ defines an $R^{(k)}$-valued quadratic form 
on $X^+$ that we denote by $q(x)$. 
Since $W^+$ contains $Rf_1^+ +R f_2^+$, which in turn contains
a submodule isomorphic to 
$\fm^{d_1+r(C_1+C_\alpha)}R^{(k)}\oplus \fm^{d_2+r(C_1+C_\alpha)}R^{(k)}$, there exists
$\psi^+\in \Hom_R(W^+,(T_E^{(k)})^+)$ and $j\in \{1,2\}$ such that 
$\beta(\psi^+(f_j^+)) \in \pi^{\max{d_1,d_2}+r(C_1+C_\alpha)}(R^{(k)})^\times$.
Similarly, there exists $\psi^-\in \Hom_R(W^-,(T_E^{(k)})^-)$
such that $\beta(\psi^-(f_{3-j}^-)) \in \pi^{\min{d_1,d_2}+r(C_1+C_\alpha})(R^{(k)})^\times$
and $\beta(\psi^-(f_j^-)) = 0$.
For such a pair $(\psi^+,\psi^-)$, 
$$
q(\psi^+,\psi^-) \in \pi^{d_1+d_2+2r(C_1+C_\alpha)}(R^{(k)})^\times.
$$
From \eqref{eq:spanimage} it follows that $p^{C_2}(\psi^+,\psi^-) = x_\psi$ for some $x_\psi\in X^+$, and
$$
q(x_\psi) = p^{2C_2}q(\psi^+,\psi^-) 
\in\pi^{d_1+d_2+2r(C_1+C_2+C_\alpha)}(R^{(k)})^\times.
$$
It then follows from \cite[Lem.~(6.6.1)(ii)]{nekovar} that 
\begin{equation}\label{eq:qimage}
q(Z^+) \not\subset \fm^{d_1+d_2+2r(C_1+C_2+C_\alpha)+1}R^{(k)}.
\end{equation}

If $m_3>0$, let $Z_3\subset Z^+$ be the submodule such that 
$g_3^+(Z_3) = \fm^{k-{m_3}+1}R^{(k)} u_+$. Otherwise, let $Z_3 =0$. Then $Z_3$ is a proper $\Z_p$-submodule of $Z^+$.
It then follows from \cite[Lem.~(6.6.1)(iii)]{nekovar} and \eqref{eq:qimage} that
\begin{equation}\label{eq:z}
\text{there exists $z\in Z^+$ such that $z\not\in Z_3$ and $q(z) \not\in \fm^{d_1+d_2+2r(C_1+C_2+C_\alpha)+1}R^{(k)}$.}
\end{equation}

Let $M$ be the fixed field of the subgroup $H\subset G_L$, so $\Gal(M/L) = Z$. 
Let $g = \tau z \in \Gal(M/\Q)$, and let 
$\ell\nmid pND_K$ be any prime such that each $c_i$ is unramified at $\ell$ and $\Frob_\ell = g$ in $\Gal(M/\Q)$ 
(there are infinitely many such $\ell$: this is the application of the \v{C}ebotarev Density Theorem). 
Since $G_L$ fixes $E[p^k]$ and $K$, 
$\Frob_\ell$ acts as $\tau$ on $K$ and $E[p^k]$.
This means that $\ell$ is inert in $K$ and that $a_\ell(E) \equiv \ell +1 \equiv 0 \mod p^k$. 
That is, $\ell\in \CL^{(k)}$.

Since $\ell$ is inert in $K$, the Frobenius element for $K_\ell$ is $\Frob_\ell^2$. Consider the restriction of
$c_i$ to $K_\ell$. Since $c_i$ is unramified at $\ell$, $\loc_\ell(c_i)\in \rH^1_{\rm ur}(K_\ell,T^{(k)}$.
Evaluation at $\Frob_\ell^2$ is an isomorphism
$$
\rH^1_{\rm ur}(K_\ell, T^{(k)}) \xrightarrow\sim T^{(k)}/(\Frob_\ell^2-1)T^{(k)} = T^{(k)} = T_E^{(k)},
$$
where the last equality is because $\Frob_\ell^2$ acts as $\tau^2 = 1$ on $T^{(k)}$ by the choice of $\ell$.
This means that $\loc_\ell(c_i)$ is completely determined by $c_i(\Frob_\ell^2)$.
Furthermore, since $\Frob_\ell^2 = g^2  = z^2 \in \Gal(M/L)$, $c_i(\Frob_\ell^2) = f_i(z^2)$. Hence 
\begin{equation}\label{eq:ciz}
c_i(\Frob_\ell^2) = f_i(z^2) = 2f_i(z) = f_i^+(z) + f_i^-(z) = (g_i^+(z),g_i^-(z)) \in T_E^{(k)} = (T_E^{(k)})^+\oplus (T_E^{(k)})^-,
\end{equation}
since the projection of $f_i^\pm$ to $(T_E^{(k)})^\mp$ vanishes on $Z^+$.  

From \eqref{eq:ciz} we see that $\ord(\loc_\ell(c_3)) = \ord(c_3(\Frob_\ell^2) = \ord(f_3(z^2)) \geqslant \ord(g_3^+(z))$.
Since $z\not\in Z_3$ by \eqref{eq:z},
$$
\ord(\loc_\ell(c_3)) \geqslant m_3,
$$
which shows that $\ell$ satisfies the first condition of the theorem.

From \eqref{eq:ciz} we also see that
$$
R\loc_\ell(c_1) + R\loc_\ell(c_2) \xrightarrow\sim R (g_1^+(z),g_1^-(z)) + R(g_2^+(z),g_2^-(z)) \subset T_E^{(k)} = (T_E^{(k)})^+\oplus (T_E^{(k)})^-.
$$
Write $g_i^\pm(z) = \beta_i^\pm(z)u_\pm$. Then 
$$
q(z) = \det\left(\begin{matrix}
\beta_1^+(z) & \beta_1^-(z) \\
\beta_2^+(z) & \beta_2^-(z) 
\end{matrix}\right).
$$
Since $q(z) \not\in \fm^{d_1+d_2+2r(C_1+C_2+C_\alpha)+1}R^{(k)}$ by \eqref{eq:z}, it follows from the above 
expression for $q(z)$ that the module
$R (g_1^+(z),g_1^-(z)) + R(g_2^+(z),g_2^-(z))$ contains a submodule isomorphic to 
$\fm^{d_1+d_2+2r(C_1+C_2+C_\alpha)}(R^{(k)}\oplus R^{(k)})$, which shows that $\ell$ also satisfies the
second condition of the theorem.
\end{proof}

\begin{cor}\label{cor:prime2} Suppose $\alpha\neq 1$.
Let $c_1, c_2 \in \rH^1(K,T^{(k)})$. Suppose $Rc_1 + Rc_2$ contains a submodule isomorphic to 
$\fm^{d_1} R^{(k)}\oplus \fm^{d_2} R^{(k)}$ for some $d_1,d_2 \geqslant 0$.
Then there exist infinitely many primes $\ell\in \CL^{(k)}$ such that 
$\ord(\loc_\ell(c_1)) \geqslant \ord(c_1) - r(C_1+C_2+C_\alpha)$ and 
$R\loc_\ell(c_1) + R \loc_\ell(c_2) \subset \rH^1(K_\ell,T^{(k)})$ contains a
submodule isomorphic to 
$$
\fm^{k-\ord(c_1)+r(C_1+C_2+C_\alpha)}R^{(k)} \oplus \fm^{d_1+d_2+2r(C_1+C_2+C_\alpha)}R^{(k)}.
$$
\end{cor}

\begin{proof} We apply Proposition \ref{prop:prime2}
with $c_3 = c_1$. Then $R\loc_\ell(c_1) = R\loc_\ell(c_3)$ contains a submodule isomophic 
to $\fm^{k-\ord(c_3)+r(C_1+C_2+C_\alpha)}R^{(k)} = \fm^{k-\ord(c_1)+r(C_1+C_2+C_\alpha)}R^{(k)}$, and
$R\loc_\ell(c_1) + R\loc_\ell(c_2)$ contains a submodule isomorphic to 
$\fm^{d_1+d_2+2r(C_1+C_2+C_\alpha)}(R^{(k)}\oplus R^{(k)})$,
whence the conclusion of the corollary.
\end{proof}

With Proposition \ref{prop:prime2} -- and especially Corollary \ref{cor:prime2} -- in hand, we next prove the following theorem, which implies the first statement of Theorem~\ref{thm:Zp-twisted} and will be used in the next section to prove the bound on
the length of $M_\alpha$.

\begin{theorem}\label{thm:rank1}
Suppose $\alpha\neq 1$. 
If $\kap_{\alpha,1}\in \rH^1(K,T)$ is non-zero, then $\epsilon=1$ and for $k\gg 0$, every element in $M^{(k)}(1)$ has order strictly less than $k$. In particular, $\rH^1_{\CF}(K,T)\simeq R$.
\end{theorem} 

\begin{proof} 
Suppose $\kap_1:=\kap_{\alpha,1} \neq 0$.
The assumption $\bar{T}^{G_K}=0$ implies that $\rH^1_{\CF}(K,T)$ is torsion-free, so $\epsilon \geqslant 1$. 

If $k\gg 0$, then the image of $\kap_1$ in $\rH^1_{\CF}(K,T^{(k)})$, still denoted by $\kap_1$ by abuse of notation, 
is non-zero and $\ind(\kap_1, \rH^1_{\CF}(K,T))=\ind(\kap_1, \rH^1_{\CF}(K,T^{(k)}))$, where by the index $\ind(c,M)$
for $M$ a finitely generated $R$-module and $c\in M$ we mean the smallest integer $m\geqslant 0$ such that $c$ has non-zero
image in $M/\fm^{m+1}M$ (equivalently, $c\in \fm^mM$).  
Let $s= \ind(\kap_1, \rH^1_{\CF}(K,T))$. Let $e= r(C_1+C_2+C_\alpha)$. Suppose $k$ also satisfies
\begin{equation}\label{eq:kse}
k > s + 3e.
\end{equation}

By the definition of $s$, there exist $c_1\in \rH^1_{\CF}(K,T^{(k)})$ such that  
the image of $c_1$ in $\rH^1_{\CF}(K,T^{(k)})/\fm \rH^1_{\CF}(K,T^{(k)})$ is non-zero and $\kap_1 = \pi^{s}c_1$.
The assumption $\bar{T}^{G_K}=0$ implies that $\rH^1_{\CF}(K,T)$ is torsion-free, so $Rc_1 \simeq R^{(k)}$.
Suppose $c_2\in \rH^1_{\CF}(K,T^{(k)})$ is such that $c_2\not \in R c_1$. We will show that 
$\pi^{s+3e} c_2 \in Rc_1$. By \eqref{eq:kse} this implies
that $\rH^1_{\CF}(K,T^{(k)})/Rc_1$ is annihilated by $\pi^{k-1}$ and hence that $\epsilon \leqslant 1$.
It then follows that $\epsilon  = 1$ and every element in $M^{(k)}(1)$ has order strictly less than $k$.
This in turn implies $\rH^1_{\CF}(K,T)\simeq R$, since $\rH^1_{\CF}(K,T)$ is torsion-free. 

Let $d$ be the order of the image of $c_2$ in $\rH^1_{\CF}(K,T^{(k)})/Rc_1$. Then
$Rc_1 + Rc_2$ contains a submodule isomorphic to $R^{(k)} \oplus \fm^{k-d}R^{(k)}$.
By Corollary \ref{cor:prime2}, there exists $\ell\in \CL^{(k)}$ such that 
$\ord(\loc_\ell(c_1)) \geqslant k-e$ and 
\begin{equation}\label{eq:rk1}
\text{$R\loc_\ell(c_1) + R\loc_\ell(c_2)$ contains a submodule isomorphic to $\fm^{e}R^{(k)}\oplus \fm^{k-d+2e}R^{(k)}$.}
\end{equation}

We now make use of the assumption that $\kap_1$ belongs to a Kolyvagin system. 
The finite-singular relation of the definition of a Kolyvagin system implies that the image of $\kap_\ell:=\kappa_{\alpha,\ell}$ in $\rH^1_{\CF}(K,T^{(k)})$,
which we also denote by $\kap_\ell$, satisfies
\begin{equation}\label{eq:rk2}
\ord(\loc_{\ell,{\rm s}}(\kap_\ell)) = \ord(\loc_\ell(\kap_1)) = \ord(\loc_\ell(\pi^{s}c_1)) \geqslant k-s-e,
\end{equation}
where by $\loc_{\ell,{\rm s}}$ we mean the composition of $\loc_\ell$ with the projection to 
$\rH^1_{\rm s}(K_\ell, T^{(k)})$.

By global duality, the images of
$$
\rH^1_{\CF}(K,T_{\alpha}^{(k)}) \xrightarrow{\loc_\ell} \rH^1_{\rm ur}(K_\ell, T_{\alpha}^{(k)})
\ \ \text{and} \ \
\rH^1_{\CF^\ell}(K,T_{\alpha^{-1}}^{(k)}) \xrightarrow{\loc_{\ell,\rm{s}}} \rH^1_{\rm s}(K_\ell, T_{\alpha^{-1}}^{(k)})
$$
are mutual annihilators under local duality. 
Since $\tau\cdot \kap_\ell \in \rH^1_{\CF^\ell}(K,T_{\alpha^{-1}}^{(k)})$, we easily conclude from \eqref{eq:kse}, \eqref{eq:rk1}, 
and \eqref{eq:rk2} that
$$
k-s-e \leqslant k-d+2e.
$$
That is, $d \leqslant s+3e$, as claimed.
\end{proof}

\subsubsection{Some simple algebra} Our adaptation of Kolyvagin's arguments 
relies on the following simple results about finitely-generated torsion $R$-modules. 
For a finitely-generated torsion $R$-module $M$ we write 
\begin{displaymath}
\exp(M):=\min\{n\geqslant 0: \pi^n M=0\} = \max\{\ord(m) : m\in M\}.
\end{displaymath}

\begin{lemma}\label{lem:key0}
Let $N\subset M$ be finitely-generated torsion $R$-modules. 
Suppose $N \simeq \oplus_{i=1}^r R/\fm^{d_i(N)}$, $d_1(N)\geqslant \cdots \geqslant d_r(N)$,
and $M\simeq\oplus_{i=1}^s R/\fm^{d_i(M)}$, $d_1(M)\geqslant \cdots \geqslant d_s(M)$.
Then $r\leqslant s$ and
\[
d_i(N) \leqslant d_i(M), \ \ i=1,\dots,r.
\]
\end{lemma}

\begin{proof}
We have $r = \dim_{R/\fm} N[\pi] \leqslant \dim_{R/\fm} M[\pi] = s$,
which proves the first claim.

We prove the second claim by induction on $r$.
Let $d = d_r(N)$. Since $N[\pi^d] = N \cap M[\pi^d]$, the inclusion $N\subset M$
induces an inclusion 
$$
N' = N/N[\pi^d] \subset M/M[\pi^d] = M'.
$$
Clearly, $N' \simeq \oplus_{i=1}^{r'} R/\fm^{d_i(N)-d}$, where 
$r'$ is the smallest integer such that $d_i(N) = d$ for $r'+1\leqslant i \leqslant r$.
Similarly, $M'\simeq \oplus_{i=1}^{s'} R/\fm^{d_i(M)-d}$. 
Since $r'<r$, the induction hypothesis implies that
$d_i(M)\geqslant d_i(N)$ for $i=1,\dots,r'$.
To complete the induction step we just need to show that at least
$r$ of the $d_i(M)$'s are $\geqslant d$. But this is clear from the injection
$N[\pi^d]/N[\pi^{d-1}] \hookrightarrow  M[\pi^d]/M[\pi^{d-1}]$,
from which it follows that
$$
r = \dim_{R/\fm} N[\pi^d]/N[\pi^{d-1}] \leqslant \dim_{R/\fm} M[\pi^d]/M[\pi^{d-1}].
$$
\end{proof}

Next we consider two short exact sequences of finitely-generated torsion $R$-modules
\begin{equation}\label{ses1}
0 \rightarrow X \rightarrow R/\fm^k \oplus M \xrightarrow{\alpha} R/\fm^{k-a} \oplus R/\fm^{b} \rightarrow 0
\end{equation}
and
\begin{equation}\label{ses2}
0 \rightarrow X \rightarrow R/\fm^k \oplus M'\xrightarrow{\beta} R/\fm^{a'} \oplus R/\fm^{k-b'} \rightarrow 0
\end{equation}
satisfying:
\begin{equation}\label{ses-hyp}
k>\exp(M)+2a \ \ \ \text{and} \ \ \ a'\leqslant a.
\end{equation}
We further assume that both $M$ and $M'$ are the direct sum of two isomorphic $R$-modules. Let $2s:= \dim_{R/\fm} M[\pi], 2s':= \dim_{R/\fm} M'[\pi]$ and $d_1(M),\dots,d_{2s}(M)$ be the lengths of the $R$-summands in a decomposition of $M$ as a direct sum of cyclic $R$-modules,
ordered so that
$$
d_1(M) = d_2(M) \geqslant d_3(M)=d_4(M) \geqslant \cdots \geqslant d_{2s-1}(M) = d_{2s}(M).
$$
Note that $d_1(M) = \exp(M)$.
Fix a decomposition 
$$
M = \oplus_{i=1}^{2s} M_i, \ \ M_i \simeq R/\fm^{d_i(M)}.
$$
Let $d_1(M'),\dots,d_{2s'}(M')$ be similarly defined for $M'$. 
\begin{lemma}\label{lem:key} The following hold:
\begin{itemize}
\item[(i)]  $s-1\leqslant s'\leqslant s+1$,
\item[(ii)] $b\leqslant \exp(M)$,
\item[(iii)] $\exp(X) \leqslant \exp(M)+a$.
\end{itemize}
\end{lemma}

\begin{proof} Let $r(-)$ denote the minimal number of $R$-generators of $(-)$. Then from \eqref{ses1} it follows 
that $r(M) -1\leqslant r(X) \leqslant r(M)+1$ (see Lemma \ref{lem:key0}). 
Similarly, it follows from \eqref{ses2} that
$r(M')-1\leqslant r(X) \leqslant r(M')+1$. From this we conclude that $r(M)-1\leqslant r(M')+1$ and $r(M')-1\leqslant r(M)+1$. 
Since $r(M) = 2s$ and $r(M') = 2s'$, this implies $2s\leqslant 2s'+2$ and $2s'\leqslant 2s+2$. That is,
$s-1\leqslant s'\leqslant s+1$, as claimed in part (i).

For part (ii) we note that since $k-a>\exp(M)$ by \eqref{ses-hyp}, the image under $\alpha$ of the summand $R/\fm^k$ in the middle
of \eqref{ses1} must be isomorphic to $R/\fm^{\max\{k-a,b\}}$ (else $\exp(\mathrm{im}(\alpha))\leqslant \max\{k-a,b\}-1$). 
It follows that $\alpha$ induces a surjection $M\twoheadrightarrow (R/\fm^ {k-a}\oplus R/\fm^b)/\alpha(R/\fm^k) \simeq R/\fm^{\min\{k-a,b\}}$.
In particular, $\min\{k-a,b\}\leq \exp(M)$. As $k-a>\exp(M)$, this implies part (ii).

For part (iii) we note that \eqref{ses1} induces an inclusion
$$
X/(X\cap R/\fm^k) \hookrightarrow (R/\fm^k\oplus M)/(R/\fm^k) \simeq M.
$$
It follows that $\exp(X)\leqslant \exp(M) + \exp(X\cap R/\fm^k)$.  As noted in the proof of part (ii), $\alpha(R/\fm^k) \simeq R/\fm^{k-a}$
so $X\cap R/\fm^k \simeq R/\fm^a$. Part (iii) follows. 
\end{proof}

\begin{prop}\label{prop:key} The following hold:\hfill
\begin{itemize}
\item[(i)] There exists $1\leqslant i_0\leqslant 2s$ such that there is an inclusion
$\oplus_{i=1, i\neq i_0}^{2s} M_i \hookrightarrow X. $
\item[(ii)] There exists an inclusion $ X \hookrightarrow M'\oplus R/\fm^{\exp(X)}$.
\item[(iii)] $d_{i}(M') \geqslant  d_{i+2}(M)$, for $i=1,\dots,2s-2$.
\end{itemize}
\end{prop}

\begin{proof} 
As explained in the proof of Lemma \ref{lem:key}(ii),
 the image under $\alpha$ of the $R/\fm^k$ summand in the middle of \eqref{ses1} has exponent $k-a$.
In particular, we may assume that the $R/\fm^{k-a}$ summand on the right in \eqref{ses1} 
is the image under $\alpha$ of the $R/\fm^k$-summand in the middle.

Let $1\leqslant i_0\leqslant 2s$ be such that $\mathrm{im}(\alpha) = R/\fm^{k-a}+\alpha(M_{i_0})$. 
It follows from \eqref{ses1} that there is a surjection
$$
X\twoheadrightarrow (R/\fm^k\oplus M)/(R/\fm^k\oplus M_{i_0})\simeq \oplus_{i=1, i\neq i_0}^{2s} M_i.
$$
Taking duals we deduce the existence of an inclusion
$\oplus_{i=1, i\neq i_0}^{2s} M_i \hookrightarrow X$,
proving (i). (Here and in the following we are using that the (Pontryagin) dual of a torsion $R$-module is
isomorphic to itself as an $R$-module.)

For (ii), we first claim that $\beta(R/\fm^k) \simeq R/\fm^{k-b'}$. Suppose
that $\beta(R/\fm^k) \simeq R/\fm^{k-b''}$ for some $b''>b'$. This would imply that there exists $m'\in M$ such that
$\beta(1\oplus m' ) \in R/\fm^{a'}\oplus 0 \subset R/\fm^{a'}\oplus R/\fm^{k-b'}$. In particular, we would have
 $\pi^{a'}(1\oplus m') \in X$. But since $\ord(\pi^{a'}(1\oplus m')) = k-a'$ this would mean that $X$ contains
 a submodule isomorphic to $R/\fm^{k-a'}$. 
 But since $k-a'> \exp(M)+a \geqslant \exp(X)$ by  \eqref{ses-hyp} and Lemma \ref{lem:key}(iii), we reach a contradiction. 
Thus we may assume that 
the $R/\fm^{k-b'}$ summand on the right in \eqref{ses2} 
is the image under $\beta$ of the $R/\fm^k$-summand in the middle. 

Let $M''\subset M'$ be the submodule such that $\beta(M'')\subseteq R/\fm^{k-b'}$. 
Then \eqref{ses2} implies that there is an exact sequence
$$
0\rightarrow X \rightarrow R/\fm^k \oplus M'' \rightarrow \beta(R/\fm^k)\rightarrow 0.
$$
From this it follows that there is an exact sequence
$$
0\rightarrow X\cap R/\fm^k \rightarrow X \rightarrow M'' \rightarrow 0.
$$
Taking duals we conclude that there exists a short exact sequence
$$
0 \rightarrow M'' \rightarrow X \xrightarrow{\gamma} X\cap R/\fm^k \rightarrow 0.
$$
Note that $X\cap R/\fm^k$ is a cyclic $R$-module. 
Let $R/\fm^{d} \subset X$ be an $R$-summand that surjects
onto $X\cap R/\fm^k$ via $\gamma$. Then there is a surjection
$M''\oplus R/\fm^{d}\twoheadrightarrow X$.
Taking duals we deduce the existence of inclusions
$$
X\hookrightarrow M'' \oplus R/\fm^d \hookrightarrow M'\oplus R/\fm^{\exp(X)}.
$$
This proves (ii).

Let 
$d_1(X) \geqslant d_2(X) \geqslant \cdots \geqslant d_{t}(X)$
be the lengths of the summands in a decomposition of $X$ as a direct sum of cyclic $R$-modules.
Note that $d_1(X)=\exp(X)$. From part (i) we see that $t\geqslant 2s-1$.
From part (i) and Lemma \ref{lem:key0} we also easily conclude that $d_i(X) \geqslant d_{i+1}(M)$.
Similarly, from part (ii) we conclude that $d_i(M') \geqslant d_{i+1}(X)$. Combining these yields (iii).
\end{proof}

\subsubsection{Finishing the proof of Theorem \ref{thm:Zp-twisted}}
We now have all the pieces needed to prove Theorem \ref{thm:Zp-twisted}. 

Since the character $\alpha$ is fixed, for the rest of the proof we denote $\kap_{n}:=\kappa_{\alpha,n}$ for all $n\in\cN$. In particular, our assumption is that $\kap_1\neq 0$. Let 
$$
\ind(\kap_1) = \max\{m \ : \ \kap_1\in \fm^m\rH^1_\CF(K,T)\}.
$$
We can write $\rH^1_\CF(K,A)=(\Phi/R)^n \oplus M$, for $n\geq 0$ and $M$ a finite $R$-module. 
Since $\rH^1_\CF(K,A) = \varinjlim_k \rH^1_\CF(K,T^{(k)})$, it follows from Lemma \ref{lemmamod} that 
\begin{displaymath}
\rH^1_\CF(K,A)[\fm^k] \simeq \rH^1_\CF(K,T^{(k)}).
\end{displaymath}
Recall that by Theorem \ref{thm:rank1} (and its proof),
$\rH^1_\CF(K,T)$ has $R$-rank one and for $k\gg 0$
\begin{displaymath}
(R/\fm^k)^n \oplus M[\fm^k]\simeq R/\fm^k \oplus M^{(k)}(1) \oplus M^{(k)}(1),
\end{displaymath}
with $\exp(M^{(k)}(1))< k$
and hence 
$$
\rH^1_\CF(K,A) \simeq \Phi/R \oplus M, \ \ M \simeq M_0\oplus M_0,
$$
for some finitely-generated torsion $R$-module $M_0$ such that $M_0\simeq M^{(k)}(1)$ for $k\gg 0$.

Let $r(M)$ be the minimal number of $R$-generators of $M$ and let
\[
e = (C_1+C_2 + C_\alpha)\rank_{\Z_p}(R).
\]
We will show that 
\begin{equation}\label{eq:main}
\ind(\kap_1)+\frac{3}{2}r(M)e \geqslant \mathrm{length}_R(M_0).
\end{equation}
Since by Lemma~\ref{lemmamod} and (\ref{eq:modp}) we have
\begin{displaymath}
r(M)+1=\dim_{R/\fm} \rH^1_\CF(K,T^{(k)})[\fm] =\dim_{R/\fm} \rH^1_\CF(K,\bar{T})=\dim_{\mathbb{F}_p} \rH^1_\CF(K,E[p]),
\end{displaymath}
it follows that (\ref{eq:main}) yields the inequality in Theorem~\ref{thm:Zp-twisted} with an error term $E_\alpha=r(M)e$ 
that depends only on $C_{\alpha}$, $T_pE$, and $\rank_{\Z_p}(R)$. 

Let $s=r(M)/2$ and fix an integer $k>0$ such that 
\begin{equation}\label{eq:kbig}
k/2 > \mathrm{length}_R(M_0) + \ind(\kap_1) + (r(M)+1)e
\end{equation}
and $M_0\simeq M^{(k)}(1)$.
Our proof of \eqref{eq:main} relies on making a good choice of integers in $\CN^{(k)}$, which in turn relies on a good choice
of primes in $\CL^{(k)}$. 

Let $n\in \CN^{(k)}$.
By Proposition \ref{propstructure} and Theorem \ref{thm:rank1}, there exists a finite $R^{(k)}$-module $M(n)_0$ such that 
$$
\rH^1_{\CF(n)}(K,T^{(k)}) \simeq R^{(k)}\oplus M(n), \ \ M(n) \simeq M(n)_0\oplus M(n)_0.
$$
Let $r(M(n))$ be the minimal number of $R$-generators of $M(n)$ and let
\[
d_1(n)= d_2(n)\geqslant d_3(n)= d_4(n)\geqslant\cdots \geqslant d_{r(M(n))-1}(n)= d_{r(M(n))}(n)
\]
be the lengths of the cyclic $R$-modules appearing in an expression 
for $M(n)$ as a direct sum of such modules. 
Let $s(n) = r(M(n))/2$. In particular, $s(1) = r(M)/2 = s$.  
In what follows we write, in an abuse of notation, $\kap_n$ to mean its image 
in $\rH^1(K,T^{(k)})$.

Suppose we have a sequence of integers $1=n_0,n_1,n_2,\dots,n_s\in \CN^{(k)}$ satisfying
\begin{itemize} 
\item[(a)] $s(n_j)\geqslant s(n_{j-1})-1$,
\item[(b)] $d_t(n_j) \geqslant d_{t+2}(n_{j-1})$, $t=1,\dots,s(n_{j-1})-1$,
\item[(c)] $\mathrm{length}_R(M(n_j)_0) \leqslant \mathrm{length}_R(M(n_{j-1})_0) - 
d_1(n_{j-1}) + 3e$,
\item[(d)] $\ord(\kap_{n_j}) \geqslant \ord(\kap_{n_{j-1}}) - e$, and
\item[(e)] $\ord(\kap_{n_{j-1}}) \leqslant \ord(\kap_{n_{j}}) - d_1(n_{j-1}) + 3e$,
\end{itemize}
for all $1\leqslant j\leqslant s$. 
Since $\rH^1_{\CF}(K,T)$ is torsion free, $\ind(\kap_1) = k - \ord(\kap_1)$, and so
repeated combination of (b) and (e) yields
\begin{align*}
\ind(\kap_1) = k -\ord(\kap_{n_0}) &\geqslant d_1(n_0) + d_3(n_0) + \cdots + d_{2s-1}(n_0) - 3se  
+ (k-\ord(\kap_{n_s}))\\
& \geqslant \mathrm{length}_R(M(n_0)_0) - 3se.
\end{align*}
Since $M(n_0)_0 = M(1)_0 \simeq M^{(k)}(1)_0\simeq M_0$ by the choice of $k$ and $3se = \frac{3}{2}r(M)e$, this means (\ref{eq:main}) holds.  
So to complete the proof of the theorem it suffices to find such a sequence of $n_j$'s.
In the following we will define such a sequence by making repeated use of Corollary \ref{cor:prime2}
to choose suitable primes in $\CL^{(k)}$. Note that if $s=0$ then there is nothing to prove, so we assume $s>0$.

Suppose $1=n_0,n_1,\dots,n_i \in \CN^{(k)}$, $i<s$, are such that (a)--(e) hold for all $1\leqslant j\leqslant i$ (note that if $i=0$, then this is
vacuously true).  We will explain how to choose a prime $\ell\in \CL^{(k)}$ such that $n_0,\dots,n_i,n_{i+1}=n_i\ell$
satisfy (a)--(e) for all $1\leqslant j\leqslant i+1$. Repeating this process yields the desired sequence $n_0,\dots,n_s$.

From (a), $s(n_i) \geqslant s - i>0$, so $d_1(n_i)>0$. 
Let $c_1, c_2 \in \rH^1_{\CF(n_i)}(K,T^{(k)})$ be such that $c_1$ generates an $R^{(k)}$-summand complementary
to $M(n_i)$  and
$R c_2 \simeq R/\fm^{d_1(n_i)} = \fm^{k-d_1(n_i)}R^{(k)}$ is a direct summand of 
$M(n_i) = M(n_i)_0\oplus M(n_i)_0$. Then $Rc_1+R c_2\subset \rH^1_{\CF(n_i)}(K,T^{(k)})$ contains a submodule
isomorphic to $R^{(k)}\oplus \fm^{k-d_1(n_i)}R^{(k)}$. 
Let $\ell \in \CL^{(k)}$ be a prime as in Corollary \ref{cor:prime2} that does not divide $n_1\cdots n_i$. 
In particular,
$$
\ord(\loc_\ell(c_1)) \geqslant k-e
$$
and
$$
\text{$R\loc_\ell(c_1) + R\loc_\ell(c_2)$ contains a submodule isomorphic to 
$\fm^{e}R^{(k)}\oplus \fm^{k-d_1(n_i)+2e}R^{(k)}$.}
$$
It follows that there is a short exact sequence
\begin{equation}\label{ses3}
0 \rightarrow \rH^1_{\CF(n_i)_\ell}(K,T^{(k)}) \rightarrow \rH^1_{\CF(n_i)}(K,T^{(k)}) \xrightarrow{\loc_\ell} R/\fm^{k-a} \oplus R/\fm^{b}
\rightarrow 0, \ \ e \geqslant a, \ b\geqslant d_1(n_i)-2e.
\end{equation}
Global duality then implies that there is another exact sequence
\begin{equation}\label{ses4}
0 \rightarrow \rH^1_{\CF(n_i)_\ell}(K,T^{(k)}) \rightarrow \rH^1_{\CF(n_i\ell)}(K,T^{(k)}) \xrightarrow{\loc_\ell} R/\fm^{a'} \oplus R/\fm^{k-b'}
\rightarrow 0, \ \ e\geqslant a\geqslant a', \ b'\geqslant b.
\end{equation}
Here we have used that the arithmetic dual of $T^{(k)} = T^{(k)}_\alpha$ is $T^{(k)}_{\alpha^{-1}}$ and that the complex conjugation $\tau$ induces an isomorphism $\rH^1_{\CF(n)}(K,T^{(k)}_{\alpha^{-1}})\simeq \rH^1_{\CF(n)}(K,T^{(k)}_{\alpha})$.

Combining (c) for $1\leqslant j \leqslant i$ yields 
$$
\mathrm{length}_R(M(n_{i})_0) \leqslant \mathrm{length}_R(M(n_0)_0) + 3i e.
$$
From this, together with $r(M)=2s$, $i< s$, and the assumption (\ref{eq:kbig}), we find
$$
k> 2\,\mathrm{length}_R(M(n_0)_0) + 2r(M)e \geqslant 2\,\mathrm{length}_R(M(n_i)_0)+2r(M)e-3ie  
> \mathrm{length}_R(M(n_i))+2e.
$$
It follows that \eqref{ses3} and \eqref{ses4} satisfy the hypotheses \eqref{ses-hyp} for \eqref{ses1} and \eqref{ses2}
with 
\[
X =  \rH^1_{\CF(n_i)_\ell}(K,T^{(k)}),
\quad 
M = M(n_i),\quad M' = M(n_i\ell).
\]

Let $n_{i+1}=n_i\ell$.  Then (a) for $j={i+1}$ follows from Lemma \ref{lem:key}(i)
while (b) for $j=i+1$ follows from Proposition \ref{prop:key}(iii).
To see that (c) holds we observe that by \eqref{ses3} and \eqref{ses4}
$$
\mathrm{length}_R(M(n_{i+1})) = \mathrm{length}_R(M(n_{i})) - (b+b') + (a+a') 
\leqslant \mathrm{length}_R(M(n_{i})) - 2d_1(n_{i}) + 6e.
$$

To verify (d) for $j=i+1$ we first observe that by the Kolyvagin system relations under the finite singular map 
$$
\ord(\kap_{n_{i+1}}) = \ord(\kap_{n_{i}\ell}) \geqslant \ord(\loc_{\ell}(\kap_{n_{i}\ell}))= \ord(\loc_\ell(\kap_{n_{i}})).
$$
So (d) holds for $j=i+1$ if we can show that $\ord(\loc_\ell(\kap_{n_{i}})) \geqslant \ord(\kap_{n_{i}})-e$.
To check that this last inequality holds, we first note that
$\ord(\kap_{n_{i}}) \geqslant \ord(\kap_{n_0}) - ie$ by (d) for $1\leqslant j \leqslant i$. 
But $\ord(\kap_{n_0}) =\ord(\kap_1)= k-\ind(\kap_1)$ by the choice of $k$ (and the fact that $\rH^1_{\CF}(K,T)$ is torsion-free), 
and so by (\ref{eq:kbig}) and repeated application of (c) for $1\leqslant j\leqslant i$ we have
\begin{equation*}\begin{split}
\ord(\kap_{n_i}) \geqslant k-\ind(\kap_1) - ie & >  4\cdot\mathrm{length}_R(M(n_0)_0)+(4s-i+2)e \\
& > 3\cdot\mathrm{length}_R(M(n_0)_0) + \mathrm{length}_R(M(n_i)_0) +(4s-4i+2) e  \\
& > \mathrm{length}_R(M(n_i)_0) + 2e.
\end{split}\end{equation*}
Write $\kap_{n_{i}} = x c_1 + m$ with $x\in R^{(k)}$ and $m\in M(n_i)$. 
Since $\ord(\kap_{n_i}) > \exp(M(n_i)_0)$, it follows that 
$x = \pi^t u$ for $t = k-\ord(\kap_{n_{i}})$ and some $u\in R^\times$. Let $n = \exp(M(n_{i}))$.
It follows that
$$
\pi^n\loc_\ell(\kap_{n_i}) = \pi^{n+t} u\loc_\ell(c_1).
$$
By the choice of $\ell$, $\ord(\loc_\ell(c_1)) \geqslant k - e$. Since $n+t = k - \ord(\kap_{n_i}) + \exp(M(n_i)_0)
<k - 2e$, it then follows
that 
$$
\ord(\loc_\ell(\kap_{n_i})) = \ord(\loc_\ell(c_1)) - t \geqslant k - e -t = \ord(\kap_{n_i})-e.
$$

It remains to verify (e) for $j=i+1$. Let $c\in \rH^1_{\CF(n_{i+1})}(K,T^{(k)})$ be a generator of an $R^{(k)}$-summand
complementary to $M(n_{i+1})$. 
Write $\kap_{n_i} = u\pi^g c_1 + m$ and $\kap_{n_{i+1}} = v\pi^h c + m'$, where $u,v\in R^\times$, $m\in M(n_{i})$ and $m'\in M(n_{i+1})$. Arguing as in the proof that (d) holds shows that 
$\ord(\kap_{n_j})> \exp(M(n_j))+2e$ for $1\leqslant j\leqslant i+1$, hence 
$g = k - \ord(\kap_{n_i})$ and $h = k - \ord(\kap_{n_{i+1}})$.
Arguing further as in the proof that (d) holds also yields
$$
\ord(\loc_\ell(\kap_{n_i}))  = \ord(\loc_\ell(c_1))-g \ \ \text{and} \ \  \ord(\loc_\ell(\kap_{n_{i+1}})) = \ord(\loc_\ell(c))-h.
$$
From the Kolyvagin system relations under the finite singular map it then follows that
$$
h-g = \ord(\loc_\ell(c))-\ord(\loc_\ell(c_1)).
$$
We refer again to the exact sequences \eqref{ses3} and \eqref{ses4}.
By the choice of $\ell$, 
$\ord(\loc_\ell(c_1))\geqslant k-e> \exp(M(n_i)_0)\geq b$, the last inequality by Lemma \ref{lem:key}(ii). Hence we must have 
$\ord(\loc_\ell(c_1))=k-a$. As shown in the proof of Proposition \ref{prop:key} (ii), we also must have $\ord(\loc_\ell(c))=k-b'$. Hence we find
$$
h-g =  (k-b')-(k-a) = a-b' \leqslant 3e -d_1(n_{j-1}).
$$
Since $h-g = \ord(\kap_{n_i})-\ord(\kap_{n_{i+1}})$, this proves (e) holds for $j=i+1$ and so concludes the proof of 
Theorem \ref{thm:Zp-twisted}.

\subsection{Iwasawa theory} 
\label{subsec:Iw-proof}

Let $E$, $p$, and $K$ be as in $\S\ref{sec:Zp-twisted}$. Let $\Lambda=\Z_p\llbracket\Gamma\rrbracket$ be the anticyclotomic Iwasawa algebra, and consider the $\Lambda$-modules
\[
M_E:=(T_pE)\otimes_{\Z_p}\Lambda^\vee,\quad
\mathbf{T}:=M_E^\vee(1)\simeq(T_pE)\otimes_{\Z_p}\Lambda,
\]
where the $G_K$-action on $\Lambda^\vee$ is given by the inverse $\Psi^{-1}$ of the tautological character $\Psi: G_K\twoheadrightarrow\Gamma\hookrightarrow\Lambda^\times$.

For $w$ a prime of $K$ above $p$, put
\[
{\rm Fil}_w^+(M_E):={\rm Fil}_w^+(T_pE)\otimes_{\Z_p}\Lambda^\vee,\quad
{\rm Fil}_w^+\mathbf{T}:={\rm Fil}_w^+(T_pE)\otimes_{\Z_p}\Lambda.
\]
Define the \emph{ordinary} Selmer structure $\Fcal_\Lambda$ on $M_E$ and $\mathbf{T}$ by 
\[
\rH^1_{\Fcal_\Lambda}(K_w,M_E):=
\begin{cases}
{\rm im}\bigl\{\rH^1(K_w,{\rm Fil}_w^+(M_E))\rightarrow\rH^1(K_w,M_E)\bigr\} & \textrm{if $w\vert p$,}\\
0& \textrm{else,}
\end{cases}
\]
and
\[
\rH^1_{\Fcal_\Lambda}(K_w,\mathbf{T}):=
\begin{cases}
{\rm im}\bigl\{\rH^1(K_w,{\rm Fil}_w^+(\mathbf{T}))\rightarrow\rH^1(K_w,\mathbf{T})\bigr\} & \textrm{if $w\vert p$,}\\
\rH^1(K_w,\mathbf{T})& \textrm{else.}
\end{cases}
\]

Denote by 
\[
\mathcal{X}=\rH^1_{\Fcal_\Lambda}(K,M_E)^\vee={\rm Hom}_{\rm cts}(\rH^1_{\Fcal_\Lambda}(K,M_E),\Q_p/\Z_p) 
\] 
the Pontryagin dual of the 
associated Selmer group $\rH^1_{\Fcal_\Lambda}(K,M_E)$, and let $\cL_E\subset\cL_0$ be as in $\S\ref{sec:Zp-twisted}$.

Recall that $\gamma\in \Gamma$ is a topological generator. Then $\mathfrak{P}_0:=(\gamma-1)\subset \Lambda$ is a height one prime independent of the choice of $\gamma$.

\begin{theorem}\label{thm:howard}
Suppose there is a Kolyvagin system $\kappa\in\mathbf{KS}(\mathbf{T},\Fcal_\Lambda,\cL_E)$ with $\kappa_1\neq 0$. Then $\rH^1_{\Fcal_\Lambda}(K,\mathbf{T})$ has $\Lambda$-rank one, and there is a finitely generated torsion $\Lambda$-module $M$ such that
\begin{itemize}
\item[(i)] $\mathcal{X}\sim\Lambda\oplus M\oplus M$,
\item[(ii)] ${\rm char}_\Lambda(M)$ divides ${\rm char}_\Lambda\bigl(\rH^1_{\Fcal_\Lambda}(K,\mathbf{T})/\Lambda\kappa_1\bigr)$ in $\Lambda[1/p,1/(\gamma-1)]$.
\end{itemize}
\end{theorem}

\begin{proof} 
This follows by applying Theorem~\ref{thm:Zp-twisted} for the specializations of $\mathbf{T}$ at height one primes of $\Lambda$, similarly as in the proof of \cite[Thm.~2.2.10]{howard}. We only explain how to deduce the divisibility in part (ii), since part (i) is shown exactly as in \cite[Thm.~2.2.10]{howard}.

For any height one prime $\mathfrak{P}\neq p\Lambda$ of $\Lambda$, let $S_{\mathfrak{P}}$ be the integral closure of $\Lambda/\mathfrak{P}$ and consider the $G_K$-module
\[
T_\mathfrak{P}:=\mathbf{T}\otimes_\Lambda S_{\mathfrak{P}},
\]
where $G_K$ acts on $S_\mathfrak{P}$ via $\alpha_\mathfrak{P}:\Gamma\hookrightarrow\Lambda^\times\rightarrow S_\mathfrak{P}^\times$. Note that $T_\mathfrak{P}$ is a $G_K$-module of the type considered in $\S\ref{sec:Zp-twisted}$.
In particular, $S_\mathfrak{P}$ is the ring of integers of a finite extension of $\Q_p$, and 
$T_\mathfrak{P} = T_pE\otimes_{\Z_p}S_\mathfrak{P}(\alpha_\mathfrak{P})$, where 
$\alpha_\mathfrak{P} = \Psi^{-1} \,\mathrm{mod}\, \mathfrak{P}$.

Fix $\mathfrak{P}$ as above, write $\mathfrak{P}=(g)$, and set $\mathfrak{Q}:=(g+p^m)$ for some integer $m$.  
For $m\gg 0$, $\mathfrak{Q}$ is also a height one prime of $\Lambda$. As explained in \cite[p.~1463]{howard}, there is a specialization map 
\[
\mathbf{KS}(\mathbf{T},\Fcal_\Lambda,\cL_E)\rightarrow
\mathbf{KS}(T_{\mathfrak{Q}},\Fcal_{\rm ord},\cL_E).
\] 
Writing $\kappa^{(\mathfrak{Q})}$ for the image of $\kappa$ under this map, the hypothesis $\kappa_1\neq 0$ implies that $\kappa_1^{(\mathfrak{Q})}$ generates an infinite $S_{\mathfrak{Q}}$-submodule of $\rH^1_{\Fcal_{\rm ord}}(K,T_{\mathfrak{Q}})$ for $m\gg 0$. By Theorem~\ref{thm:Zp-twisted}, it follows that $X$ and $\rH^1_{\Fcal_\Lambda}(K,\mathbf{T})$ have both $\Lambda$-rank one, and letting $f_\Lambda$ be a characteristic power series for $\rH^1_{\Fcal_\Lambda}(K,\mathbf{T})/\Lambda\kappa_1$ we see as in \cite[p.~1463]{howard} that the equalities
\[
{\rm length}_{\bZ_p}\bigl(\rH^1_{\cF_{\mathfrak{Q}}}(K,T_{\mathfrak{Q}})/S_{\mathfrak{Q}}\kappa_1^{(\mathfrak{Q})}\bigr)=md\;{\rm ord}_{\mathfrak{P}}(f_\Lambda)
\]
and
\[
2\;{\rm length}_{\bZ_p}(M_{\mathfrak{Q}})=md\;{\rm ord}_{\mathfrak{P}}\bigl({\rm char}_\Lambda(\mathcal{X}_{\rm tors})\bigr)
\]
hold up to $O(1)$ as $m$ varies, where $d={\rm rank}_{\Z_p}(\Lambda/\mathfrak{P})$ and  $X_{\rm tors}$ denotes the $\Lambda$-torsion submodule of $X$. 

On the other hand, 
Theorem~\ref{thm:Zp-twisted} yields the inequality
\[
{\rm length}_{\bZ_p}(M_{\alpha_\mathfrak{Q}})\leqslant{\rm length}_{\bZ_p}\bigl(\rH^1_{\cF_{\mathfrak{Q}}}(K,T_{\mathfrak{Q}})/S_{\mathfrak{Q}}\kappa_1^{(\mathfrak{Q})}\bigr)+E_{\alpha_\mathfrak{Q}}.
\]
If $\mathfrak{P} \neq \mathfrak{P}_0$, then 
the error term $E_{\alpha_\mathfrak{Q}}$ is bounded independently of $m$, since 
${\rm rank}_{\Z_p}(S_{\mathfrak{Q}})={\rm rank}_{\Z_p}(S_{\mathfrak{P}})$ and the term $C_{\alpha_\mathfrak{Q}}$ in (\ref{eq:erralpha}) satisfies $C_{\alpha_\mathfrak{Q}}=C_{\alpha_\mathfrak{P}}$ for $m\gg 0$. Letting $m\to\infty$ we thus deduce
\[
{\rm ord}_{\mathfrak{P}}\bigl({\rm char}_\Lambda(\mathcal{X}_{\rm tors}\bigr)\bigr)\leqslant 2\;{\rm ord}_{\mathfrak{P}}(f_\Lambda),
\]
for $\mathfrak{P}\neq (p), \mathfrak{P}_0$, yielding the divisibility in part (ii).
\end{proof}

\begin{cor}\label{cor:howard}
Let the hypotheses be as in Theorem \ref{thm:howard}. Assume also that $\rH^1_{\CF}(K,E[p^\infty])$ has $\Z_p$-corank one (equivalently, $\rH^1_{\CF}(K,T_pE)$ has $\Z_p$-rank one).  Then 
${\rm char}_\Lambda(M)$ divides ${\rm char}_\Lambda\bigl(\rH^1_{\Fcal_\Lambda}(K,\mathbf{T})/\Lambda\kappa_1\bigr)$ in $\Lambda[1/p]$.
\end{cor}

\begin{proof} The assumption that $\rH^1_{\CF}(K,E[p^\infty])$ has $\Z_p$-corank one implies that 
$X_{\rm tors}/\mathfrak{P}_0X_{\rm tors}$ is a torsion $\Z_p$-module and hence that ${\rm ord}_{\mathfrak{P}_0}({\rm char}_\Lambda(X_{\rm tors})) =0$.
\end{proof}

\section{Proof of Theorem~\ref{thm:C} and Corollary~\ref{cor:D}}
\label{sec:CD}

\subsection{Preliminaries}
\label{subsec:HPKS}

Let $E$, $p$, and $K$ be as in $\S\ref{sec:Zp-twisted}$, and assume in addition that hypotheses (\ref{eq:intro-Heeg}) and (\ref{eq:intro-disc}) hold. Fix an integral ideal $\mathfrak{N}\subset\mathcal{O}_K$ with $\mathcal{O}_K/\mathfrak{N}=\Z/N\Z$. For each positive integer $m$ prime to $N$, let $K[m]$ be the ring class field of $K$ of conductor $m$, and set
\[
G[m]={\rm Gal}(K[m]/K[1]),\quad\quad\mathcal{G}[m]={\rm Gal}(K[m]/K).
\]
Let also $\mathcal{O}_m=\Z+m\mathcal{O}_K$ be the order of $K$ of conductor $m$.

By the theory of complex multiplication, the cyclic $N$-isogeny between complex CM elliptic curves
\[
\mathbb{C}/\mathcal{O}_K\rightarrow\mathbb{C}/(\mathfrak{N}\cap\mathcal{O}_m)^{-1}
\]
defines a point $x_m\in X_0(N)(K[m])$, and fixing a modular parameterization $\pi:X_0(N)\rightarrow E$ we define the \emph{Heegner point} of conductor $m$ by   
\[
P[m]:=\pi(x_m)\in E(K[m]).
\]
Building on this construction, one can prove the following result.

\begin{theorem}
\label{thm:howard-HPKS}
Assume $E(K)[p]=0$. Then there exists a Kolyvagin system $\kappa^{\rm Hg}\in\mathbf{KS}(\mathbf{T},\mathcal{F}_\Lambda,\cL_E)$ such that $\kappa_1^{\rm Hg}\in\rH^1_{\Fcal_\Lambda}(K,\mathbf{T})$ is nonzero.
\end{theorem}

\begin{proof} 
Under the additional hypotheses that $p\nmid h_K$, the class number of $K$, and the representation $G_K\rightarrow{\rm Aut}_{\Z_p}(T)$ is surjective, this is \cite[Thm.~2.3.1]{howard}. In the following paragraphs, we explain how to adapt Howard's arguments to our situation. 

We begin by briefly recalling the construction of $\kappa^{\rm Hg}$ in \cite[\S{2.3}]{howard}. Let $K_k$ be the subfield of $K_\infty$ with $[K_k\colon K]=p^k$. For each $n\in\cN$ set
\[
P_k[n]:={\rm Norm}_{K[np^{d(k)}]/K_k[n]}(P[np^{d(k)}])\in E(K_k[n]),
\]
where $d(k)=\min\{d\in\Z_{\geqslant 0}\colon K_k\subset K[p^{d(k)}]\}$, and $K_k[n]$ denotes the compositum of $K_k$ and $K[n]$. Letting $H_k[n]\subset E(K_k[n])\otimes\Z_p$ be the $\Z_p[{\rm Gal}(K_k[n]/K)]$-submodule generated by $P[n]$ and $P_j[n]$ for $j\leqslant k$, it follows from the Heegner point norm relations \cite[\S{3.1}]{perrinriou} that one can form the $\mathcal{G}(n)$-module
\[
\mathbf{H}[n]:=\varprojlim_k H_k[n].
\]
By \cite[Lem.~2.3.3]{howard}, there is a family 
\[
\{Q[n]=\varprojlim_k Q_k[n]\in\mathbf{H}[n]\}_{n\in\cN}
\]
such that
\begin{equation}\label{eq:multiplier}
Q_0[n]=\Phi P[n],\quad\textrm{where}\;
\Phi=
\left\{
\begin{array}{ll}
(p-a_p\sigma_p+\sigma_p^2)(p-a_p\sigma_p^*+\sigma_p^{*2}) & \textrm{if $p$ splits in $K$,}\\[0.2em]
(p+1)^2-a_p^2 & \textrm{if $p$ is inert in $K$,}
\end{array}
\right.
\end{equation}
with $\sigma_p$ and $\sigma_p^*$ the Frobenius elements at the primes above $p$ in the split case, and 
\[
{\rm Norm}_{K_{\infty}[n\ell]/K_{\infty}[n]}Q[n\ell] = a_{\ell}Q[n] 
\]
for all $n\ell\in\cN$. Letting $D_n\in\Z_p[G(n)]$ be Kolyvagin's derivative operators, and choosing a set $S$ of representatives for $\mathcal{G}(n)/G(n)$, the class $\kappa_n\in \rH^1(K,\mathbf{T}/I_n\mathbf{T})$ is defined as the natural image of 
\begin{equation}\label{eq:Qn}
\tilde{\kappa}_n:=\sum_{s\in S}sD_nQ[n]\in\mathbf{H}[n]
\end{equation}
under the composite map
\[
\bigl(\mathbf{H}[n]/I_n\mathbf{H}[n]\bigr)^{\mathcal{G}(n)}\xrightarrow{\delta(n)}\rH^1(K[n],\mathbf{T}/I_n\mathbf{T})^{\mathcal{G}(n)}\overset{\simeq}\longleftarrow\rH^1(K,\mathbf{T}/I_n\mathbf{T}),
\]
where $\delta(n)$ is induced by the limit of Kummer maps $\delta_k(n):E(K_k[n])\otimes\bZ_p\rightarrow{\rm H}^1(K_k[n],T)$, and the second arrow is given by restriction. (In our case, that the latter is an isomorphism follows from the fact that the extensions $K[n]$ and $\Q(E[p])$ are linearly disjoint, and $E(K_\infty)[p]=0$.) 

The proof that the classes  $\kappa_n$ land in $\rH^1_{\Fcal_\Lambda}(K,\mathbf{T})$ and can be modified to a system $\kappa^{\rm Hg}=\{\kappa_n^{\rm Hg}\}_{n\in\cN}$ satisfying the Kolyvagin system relations is the same as in \cite[Lem.~2.3.4]{howard} \emph{et seq.}, noting that the arguments proving Lemma~2.3.4 (in the case $v\vert p$) apply almost \emph{verbatim} in the case when $p$ divides the class number of $K$. Finally, that $\kappa_1^{\rm Hg}$ is nonzero follows from the works of Cornut and Vatsal \cite{cornut, vatsal}.
\end{proof}

Applying Theorem~\ref{thm:howard} and Corollary~\ref{cor:howard} to the Kolyvagin system $\kappa^{\rm Hg}$ of Theorem~\ref{thm:howard-HPKS}, we thus obtain the following.

\begin{theorem}\label{thm:howard-HP}
Assume $E(K)[p]=0$. Then the module $\rH^1_{\Fcal_\Lambda}(K,\mathbf{T})$ has $\Lambda$-rank one, and there is a finitely generated torsion $\Lambda$-module $M$ such that
\begin{itemize}
\item[(i)] $\mathcal{X}\sim\Lambda\oplus M\oplus M$,
\item[(ii)] ${\rm char}_\Lambda(M)$ divides ${\rm char}_\Lambda\bigl(\rH^1_{\Fcal_\Lambda}(K,\mathbf{T})/\Lambda\kappa_1^{\rm Hg}\bigr)$ in $\Lambda[1/p,1/(\gamma-1)]$.
\end{itemize}
Moreover, if $\rH^1_{\CF}(K,E[p^\infty])$ has $\Z_p$-corank one,  then ${\rm char}_\Lambda(M)$ divides ${\rm char}_\Lambda\bigl(\rH^1_{\Fcal_\Lambda}(K,\mathbf{T})/\Lambda\kappa_1^{\rm Hg}\bigr)$ in $\Lambda[1/p]$.
\end{theorem}

\begin{remark}\label{rem:comparison}
For our later use, we compare the class $\kappa_1^{\rm Hg}\in\rH^1_{\Fcal_\Lambda}(K,\mathbf{T})$ from Theorem~\ref{thm:howard-HPKS} with the $\Lambda$-adic class constructed in \cite[\S{5.2}]{cas-hsieh1} (taking for $f$ the newform associated with $E$). 

Denote by $\alpha$ the $p$-adic unit root of $x^2-a_px+p$. With the notations introduced in the proof of Theorem~\ref{thm:howard-HPKS}, define the \emph{$\alpha$-stabilized Heegner point} $P[p^k]_\alpha\in E(K[p^k])\otimes\Z_p$ by 
\begin{equation}\label{eq:stabilized-1}
P[p^k]_\alpha:=\left\{
\begin{array}{ll}
P[p^k]-\alpha^{-1}P[p^{k-1}]&\textrm{if $k\geqslant 1$,}\\[0.2em]
u_K^{-1}\bigl(1-\alpha^{-1}\sigma_p\bigr)\bigl(1-\alpha^{-1}\sigma_p^*\bigr)P[1]&\textrm{if $k=0$ and $p$ splits in $K$,}\\
u_K^{-1}\bigl(1-\alpha^{-2}\bigr)P[1]&\textrm{if $k=0$ and $p$ is inert in $K$.}
\end{array}
\right.
\end{equation}

Using the Heegner point norm relations, a straightforward calculation shows that the points $\alpha^{-k}P[p^k]_\alpha$ are norm-compatible.  Letting $\delta:E(K_k)\otimes\Z_p\rightarrow\rH^1(K_k,T_pE)$ be the Kummer map, we may therefore set
\[
\kappa_\infty:=\varprojlim_k\delta(\kappa_k)\in\varprojlim_k\rH^1(K_k,T_pE)\simeq\rH^1(K,\mathbf{T}),
\]
where $\kappa_k=\alpha^{-d(k)}{\rm Norm}_{K[p^{d(k)}]/K_k}(P[p^{d(k)}]_\alpha)$. The inclusion $\kappa_\infty\in\rH^1_{\Fcal_\Lambda}(K,\mathbf{T})$ follows immediately from the construction. 
For the comparison with $\kappa_1^{\rm Hg}$, note that by (\ref{eq:Qn}) the projection ${\rm pr}_K({\kappa}_1^{\rm Hg})$ of $\kappa_1^{\rm Hg}$ to $\rH^1(K,T_pE)$ is given by the Kummer image of ${\rm Norm}_{K[1]/K}(Q_0[1])$, while $\kappa_0$ is the Kummer image of ${\rm Norm}_{K[1]/K}(P[1]_\alpha)$. Thus comparing $(\ref{eq:multiplier})$ and $(\ref{eq:stabilized-1})$ we see that 
\begin{equation}\label{eq:comparison}
{\rm pr}_K(\kappa_1^{\rm Hg})=
\begin{cases}
u_K\alpha^2(\beta-1)^2\cdot\kappa_0&\textrm{if $p$ splits in $K$},\\
u_K\alpha^2(\beta^2-1)\cdot\kappa_0&\textrm{if $p$ is inert in $K$,}
\end{cases}
\end{equation}
where $\beta=p\alpha^{-1}$. 
In particular, $\kappa_\infty$ and $\kappa_1^{\rm Hg}$ generate the same $\Lambda$-submodule of $\rH^1_{\Fcal_\Lambda}(K,\mathbf{T})$. 
\end{remark}

\subsection{Proof of the Iwasawa main conjectures}

Let $\kappa_\infty\in\rH^1_{\Fcal_\Lambda}(K,\mathbf{T})$ be the $\Lambda$-adic Heegner class introduced in Remark~\ref{rem:comparison}, and 
put $\Lambda_{\rm ac}=\Lambda\otimes_{\Z_p}\Q_p$.
{Let ${\rm H}^1_{\Fcal_{\rm Gr}}(K,\mathbf{T})$ be defined just as ${\rm H}^1_{\Fcal_{\rm Gr}}(K,M_E)$ but with 
$\mathbf{T}$ replacing $M_E$ and the conditions on $v$ and $\bar v$ switched.

\begin{prop}\label{prop:equiv-imc}
Assume that $p=v\bar{v}$ splits in $K$ and that $E(K)[p]=0$. 
Then the following statements are equivalent:
\begin{enumerate}
\item[(i)] Both ${\rm H}^1_{\Fcal_\Lambda}(K,\mathbf{T})$ and 
$\mathcal{X} = {\rm H}^1_{\Fcal_\Lambda}(K,M_E)^\vee$ have $\Lambda$-rank one, and the divisibility 
\[
\mathrm{char}_\Lambda(\mathcal{X}_{\rm tors}) \supset\mathrm{char}_\Lambda\bigl({\rm H}^1_{\Fcal_\Lambda}(K,\mathbf{T})/\Lambda\kappa_\infty\bigr)^2
\]
holds in $\Lambda_{\rm ac}$.
\item[(ii)]  Both ${\rm H}^1_{\Fcal_{\rm Gr}}(K,\mathbf{T})$ and $\mathfrak{X}_{E} = {\rm H}^1_{\Fcal_{\rm Gr}}(K,M_E)^\vee$ are $\Lambda$-torsion, and the divisibility 
\[
\mathrm{char}_\Lambda(\mathfrak{X}_{E})\Lambda^{\rm ur}\supset(\mathcal{L}_E)
\]
holds in $\Lambda^{\rm ur}\otimes_{\bZ_p}\bQ_p$.
\end{enumerate}
Moreover, the same result holds for the opposite divisibilities.
\end{prop}
}

\begin{proof}
See \cite[Thm.~5.2]{BCK}, whose proof still applies after inverting $p$. 
\end{proof}

We can now conclude the proof of Theorem~\ref{thm:C} in the introduction. 

\begin{theorem}\label{thm:thmA}
Suppose $K$ satisfies hypotheses {\rm (\ref{eq:intro-Heeg})}, {\rm (\ref{eq:intro-spl})}, {\rm (\ref{eq:intro-disc})}, and 
{\rm (\ref{eq:intro-sel})}, and that
$E[p]^{ss}=\mathbb{F}_p(\phi)\oplus\mathbb{F}_p(\psi)$ as $G_\bQ$-modules, with $\phi\vert_{G_p}\neq\mathds{1},\omega$.
Then $\mathfrak{X}_E$ is $\Lambda$-torsion, and
\[
{\rm char}_\Lambda(\mathfrak{X}_E)\Lambda^{\rm ur}=(\mathcal{L}_E)
\]
as ideals in $\Lambda^{\rm ur}$. 
\end{theorem}

\begin{proof}
By Theorem~\ref{thm:howard-HP}, 
the modules ${\rm H}^1_{\Fcal_\Lambda}(K,\mathbf{T})$ and 
$\rH^1_{\Fcal_\Lambda}(K,M_E)^\vee$ have both $\Lambda$-rank one, with
\[
{\rm char}_\Lambda\bigl(\rH^1_{\Fcal_\Lambda}(K,M_E)^\vee_{\rm tors}\bigr)\supset{\rm char}_\Lambda\bigl({\rm H}^1_{\Fcal_\Lambda}(K,\mathbf{T})/\Lambda\kappa_1^{\rm Hg}\bigr)^2
\]
as ideals in $\Lambda_{\rm ac}=\Lambda[1/p]$. Since by Remark~\ref{rem:comparison} 
the classes $\kappa_1^{\rm Hg}$ and $\kappa_\infty$ generate the same $\Lambda$-submodule of $\rH^1_{\Fcal_\Lambda}(K,\mathbf{T})$, by Proposition~\ref{prop:equiv-imc} it follows that $\mathfrak{X}_E$ is $\Lambda$-torsion, with
\[
{\rm char}_\Lambda(\mathfrak{X}_E)\Lambda^{\rm ur}\supset(\mathcal{L}_E)
\]
as ideals in $\Lambda_{\rm ac}\hat{\otimes}_{\bZ_p}\bZ_p^{\rm ur}$. This divisibility, together with the equalities $\mu(\mathfrak{X}_E)=\mu(\mathcal{L}_E)=0$ and $\lambda(\mathfrak{X}_E)=\lambda(\mathcal{L}_E)$ in Theorem~\ref{mulambda}, yields the result.
\end{proof}

As a consequence, we can also deduce the first cases of Perrin-Riou's Heegner point main conjecture \cite{perrinriou} in the residually reducible case. More precisely, together with Theorem~\ref{thm:howard-HP}, the following yields Corollary~\ref{cor:D} in the introduction. 

\begin{cor}\label{cor:PR} Suppose $K$ satisfies hypotheses {\rm (\ref{eq:intro-Heeg})}, {\rm (\ref{eq:intro-spl})},  {\rm (\ref{eq:intro-disc})}, and {\rm (\ref{eq:intro-sel})}, and that
$E[p]^{ss}=\mathbb{F}_p(\phi)\oplus\mathbb{F}_p(\psi)$ as $G_\bQ$-modules, with $\phi\vert_{G_p}\neq\mathds{1},\omega$. 
Then both ${\rm H}^1_{\Fcal_\Lambda}(K,\mathbf{T})$ and 
$\rH^1_{\Fcal_\Lambda}(K,M_E)^\vee$ have $\Lambda$-rank one, and
\[
\mathrm{char}_\Lambda(\rH^1_{\Fcal_\Lambda}(K,M_E)^\vee_{\rm tors})=\mathrm{char}_\Lambda\bigl({\rm H}^1_{\Fcal_\Lambda}(K,\mathbf{T})/\Lambda\kappa_\infty\bigr)^2
\]
as ideals in $\Lambda_{\rm ac}$.
\end{cor}

\begin{proof}
In light of Remark~\ref{rem:comparison}, this is the combination of Theorem~\ref{thm:thmA} and Proposition~\ref{prop:equiv-imc}.
\end{proof}

\begin{remark}\label{rmk:hyp-sel}
If the Heeger point $P_K = {\rm Norm}_{K[1]/K}(P[1])\in E(K)$ is non-torsion, then
{\rm (\ref{eq:intro-sel})} holds by the main results of \cite{kolyvagin-mw-sha}.
In particular, this is so if the image ${\rm pr}_K(\kappa_1^{\rm Hg})$ of $\kappa_1^{\rm Hg}$ 
(equivalently, the class $\kappa_0$) in $\rH^1_{\CF}(K,T_p(E))$ is non-zero 
(as $\rH^1_{\CF}(K,T_p(E))$ is non-torsion since $E(K)[p]= 0$). 
\end{remark}

\section{Proof of Theorem~\ref{thm:E} and Theorem~\ref{thm:F}}
\label{sec:EF}

\subsection{Preliminaries}

Here we collect the auxiliary results we shall use in the next  sections to deduce Theorems~\ref{thm:E} and~\ref{thm:F} in the introduction from our main result, Theorem~\ref{thm:thmA}.

\subsubsection{Anticyclotomic control theorem}

Assume that $p=v\bar{v}$ splits in $K$, and as in \cite[\S{2.2.3}]{jsw}, define the \emph{anticyclotomic Selmer group} of $W=E[p^\infty]$ by
\[
\rH^1_{\Fcal_{\rm ac}}(K,W)={\rm ker}\biggl\{
\rH^1(K^\Sigma/K,W)\rightarrow\prod_{w\in \Sigma}\rH^1(K_w,W)\times\frac{\rH^1(K_v,W)}{\rH^1(K_v,W)_{\rm div}}\times\rH^1(K_{\bar{v}},W)\biggr\},
\]
where $\rH^1(K_v,W)_{\rm div}\subset\rH^1(K_v,W)$ denotes the maximal divisible submodule and $\Sigma=\{w: \ w\vert N\}$. 

The following result is a special case of the ``anticyclotomic control theorem'' of \cite[\S{3}]{jsw}.

\begin{theorem}\label{controlthm}
Assume that
\begin{itemize}
\item $E(\bQ_p)[p]=0$,
\item{} ${\rm rank}_{\Z}E(K)=1$,
\item{} $\#\sha(E/K)[p^\infty]<\infty$.
\end{itemize}
Then $\mathfrak{X}_E$ is a torsion $\Lambda$-module, and letting $\Fcal_E\in\Lambda$ be a generator of ${\rm char}_\Lambda(\mathfrak{X}_E)$, we have
\[
\# \Z_p /\Fcal_E(0) = 
\# \sha(E/K)[p^\infty]\cdot\biggl(  \dfrac{ \# (\Z_p / ( \frac{1-a_p+p}{p}) \cdot \mathrm{log}_{\omega_E} P) }{ [E(K): \Z\cdot P]_p }\biggr)^2\cdot\prod_{w\mid N} c_w^{}(E/K)_p,
\]	
where
\begin{itemize}
\item $P\in E(K)$ is any point of infinite order, 
\item $\log_{\omega_E}:E(K_v)_{/{\rm tors}}\rightarrow\bZ_p$ is the formal group logarithm associated to a N\'{e}ron differential $\omega_E$, 
\item $[E(K):\Z\cdot P]_p$ denotes the $p$-part of the index $[E(K):\Z\cdot P]$,
\item $c_w^{}(E/K)_p$ is the $p$-part of the Tamagawa number of $E/K_w$.
\end{itemize}
\end{theorem}

\begin{proof}
This follows from the combination of Theorem~3.3.1 and equation (3.5.d) in \cite[(3.5.d)]{jsw}, noting that the arguments in the proof of those results apply without change with the $G_K$-irreducibility of $E[p]$ assumed in \emph{loc.\,cit.} replaced by the weaker hypothesis that $E(K)[p]=0$, {which 
is implied by the hypothesis $E(\bQ_p)[p]=0$ since $p$ splits in $K$}. 
\end{proof}

\subsubsection{Gross--Zagier formulae}

Let $E/\bQ$ be an elliptic curve of conductor $N$, and fix a  parametrization
\[
\pi:X_0(N)\rightarrow E.
\]

Let $K$ be an imaginary quadratic field 
satisfying the Heegner hypothesis relative to $N$, and fix an integral ideal $\mathfrak{N}\subset\mathcal{O}_K$ with $\mathcal{O}_K/\mathfrak{N}=\Z/N\Z$. Let $x_1=[\mathbb{C}/\mathcal{O}_K\rightarrow\mathbb{C}/\mathfrak{N}^{-1}]\in X_0(N)$ 
be the Heegner point of conductor $1$ on $X_0(N)$, which is defined over the Hilbert class field $H=K[1]$ of $K$, and set
\[
P_K=\sum_{\sigma\in{\rm Gal}(H/K)}\pi(x_1)^\sigma\in E(K).
\]

Let $f\in S_2(\Gamma_0(N))$ be the newform associated with $E$, so that $L(f,s)=L(E,s)$, and consider the differential $\omega_f:=2\pi i f(\tau)d\tau$ on $X_0(N)$. Let also $\omega_E$ be a N\'{e}ron differential on $E$, and let $c_E\in\mathbb{Z}$ be the associated \emph{Manin constant}, so that $\pi^*(\omega_E)=c_E\cdot\omega_f$.
  
\begin{theorem}
\label{thmGZ} 
Under the above hypotheses, $L(E/K,1)=0$ and 
\begin{displaymath}
L'(E/K,1)=u_K^{-2}c_E^{-2}\cdot\sqrt{\vert D_K\vert}^{-1}\cdot\Vert \omega_E \Vert^2\cdot \hat{h}(P_K),
\end{displaymath}
where $u_K=\#(\mathcal{O}_K^\times/{\pm{1}})$, 
$\hat{h}(P_K)$ is the canonical height of $P_K$, and $\Vert\omega_E\Vert^2=\iint_{E(\C)}|\omega_E \wedge \bar{\omega}_E|$. 
\end{theorem}

\begin{proof}
This is \cite[Thm.~V.2.1]{grosszagier}.
\end{proof}

\begin{theorem}
\label{thmpadicGZ} 
Under the above hypotheses, let $p>2$ be a prime of good reduction for $E$ such that $p=v\bar{v}$ splits in $K$. Then
\[
\mathcal{L}_E(0)=c_E^{-2}\cdot\bigl(1-a_pp^{-1}+p^{-1}\bigr)^2\cdot{\rm log}_{\omega_E}(P_K)^2.
\]
where ${\rm log}_{\omega_E}:E(K_v)\rightarrow K_v$ is the formal group logarithm associated to $\omega_E$.
\end{theorem}

\begin{proof}
Let $J_0(N)$ be the Picard variety of $X_0(N)$, and set $\Delta_1=(x_1)-(\infty)\in J_0(N)(H)$. 
By \cite[Thm.~5.13]{BDP} specialized to the case $k=2$, $r=j=0$, and $\chi=\mathbf{N}_K^{-1}$, we have
\[
\mathcal{L}_E(0)=\bigl(1-a_pp^{-1}+p^{-1}\bigr)^2\cdot\biggl(\sum_{\sigma\in{\rm Gal}(H/K)}\log_{\omega_f}(\Delta_1^\sigma)\biggr)^2,
\]
where $\log_{\omega_f}:J_0(N)(H_v)\rightarrow H_v$ is the formal group logarithm associated to $\omega_f$. Since $\log_{\omega_f}(\Delta_1)=
c_E^{-1}\cdot\log_{\omega_E}(\pi(\Delta_1))$, this yields the result.
\end{proof}

\subsubsection{A result of Greenberg--Vatsal}

\begin{theorem}
\label{thmGV}
Let $A/\Q$ be an elliptic curve, and let $p>2$ be a prime of good ordinary reduction for $A$. Assume that $A$ admits a cyclic $p$-isogeny with kernel $\Phi_A$, with the $G_\bQ$-action on $\Phi_A$ given by a character which is either ramified at $p$ and even, or unramified at $p$ and odd. If $L(A,1)\neq 0$ then
\begin{displaymath}
\ord_{p}\biggl(\frac{L(A, 1)}{\Omega_{A}}\biggr)=\ord_{p}\biggl(\frac{\# \sha(A / \Q)\cdot{\rm Tam}(A/\Q)}{\#(A(\bQ)_{\rm tors})^2}\biggr),
\end{displaymath}
where ${\rm Tam}(A/\bQ)=\prod_{\ell} c_\ell(A/\bQ)$  
is the product over the bad primes $\ell$ of $A$ of the Tamagawa numbers of $A/\bQ_\ell$.
\end{theorem}

\begin{proof}
By \cite{kolyvagin-mw-sha}, if $L(A,1)\neq 0$ then ${\rm rank}_\Z A(\Q)=0$ and $\#\sha(A/\Q)<\infty$; in particular, $\#{\rm Sel}_{p^\infty}(A/\bQ)=\#\sha(A/\bQ)[p^\infty]<\infty$. Letting $\Lambda_{\rm cyc}=\Z_p\llbracket{\rm Gal}(\Q_\infty/\Q)\rrbracket$ be the cyclotomic Iwasawa algebra, by \cite[Thm.~4.1]{greenberg-cetraro} we therefore have
\begin{equation}\label{eq:greenberg-4.1}
\#\Z_p/\mathcal{F}_A(0)=\frac{\#\bigl(\Z_p/(1-a_p(A)+p)^2\cdot\#\sha(A/\Q)\cdot{\rm Tam}(A/\bQ)\bigr)}{\#\bigl(\Z_p/(A(\Q)_{\rm tors})^2\bigr)},
\end{equation}
where $\mathcal{F}_A\in\Lambda_{\rm cyc}$ is a generator of the characteristic ideal of the dual Selmer group $Sel_{\Q_\infty}(T_pA,T_p^+A)^\vee$ in the notations of \cite[\S{3.6.1}]{skinner-urban}. Under the given assumptions, the cyclotomic main conjecture for $A$, i.e., the equality
\begin{equation}\label{eq:cyc-IMC}
(\mathcal{F}_A)=(\mathcal{L}_A)\subset\Lambda_{\rm cyc}
\end{equation}
where $\mathcal{L}_A$ is the $p$-adic $L$-function of Mazur--Swinnerton-Dyer, follows from the combination of \cite[Thm.~12.5]{kato-euler-systems} and\cite[Thm.~1.3]{greenvats}. By the interpolation property of $\mathcal{L}_A$, 
\begin{equation}\label{eq:interp-MSW}
\mathcal{L}_A(0)=\bigl(1-\alpha_p^{-1}\bigr)^2\cdot\frac{L(A,1)}{\Omega_A},
\end{equation}
where $\alpha_p\in\Z_p^\times$ is the unit root of $x^2-a_p(A)x+p$. Noting that ${\rm ord}_p(1-a_p(A)+p)={\rm ord}_p(1-\alpha_p^{-1})$, the result thus follows from the combination of (\ref{eq:greenberg-4.1}), (\ref{eq:cyc-IMC}), and (\ref{eq:interp-MSW}).
\end{proof}

\subsection{Proof of the \texorpdfstring{$p$}{p}-converse} 

The next result is Theorem~\ref{thm:E} in the introduction.
Note that a result for $r=0$ can be obtained from the cyclotomic main conjecture proved by combining \cite[Thm.~1.3]{greenvats} and Kato's divisibility in \cite{kato-euler-systems}. However, our assumptions are less restrictive, since in \cite{greenvats} the character $\phi$ is assumed to be ramified at $p$ and even or unramified at $p$ and odd.
\begin{theorem}\label{thm:thmB}
Assume that $E[p]^{ss}=\mathbb{F}_p(\phi)\oplus\mathbb{F}_p(\psi)$ with $\phi\vert_{G_p}\neq\mathds{1},\omega$. Let $r\in\{0,1\}$. Then 
\[
{\rm corank}_{\Z_p}{\rm Sel}_{p^\infty}(E/\Q) = r 
\quad\Longrightarrow\quad
{\rm ord}_{s=1}L(E,s)=r,
\]
and so ${\rm rank}_\Z E(\Q)=r$ and $\#\sha(E/\Q)<\infty$.
\end{theorem}

\begin{proof}
The proof of this result is a consequence of Corollary~\ref{cor:PR} for suitable choices of a quadratic imaginary field $K$ depending on $r\in\{0,1\}$.

First we suppose ${\rm corank}_{\Z_p}{\rm Sel}_{p^\infty}(E/\Q) = 1$. It follows from \cite[Theorem 1.5]{monsky} that the root number $w(E/\Q)=-1$.  Choose an imaginary quadratic field $K$ of discriminant $D_K$ such that
\begin{itemize}
\item[(a)] $D_K<-4$ is odd,
\item[(b)] every prime $\ell$ dividing $N$ splits in $K$,
\item[(c)] $p$ splits in $K$, say $p=v\bar{v}$,
\item[(d)] $L(E^K,1)\neq 0$.
\end{itemize}
The existence of such $K$ (in fact, of an infinitude of them) is ensured by \cite[Thm.~B.1]{twist}, since (a), (b), and (c) impose only a finite number of congruence conditions on $D_K$, and any $K$ satisfying (b) is such that $E/K$ has root number $w(E/K)=w(E/\Q)w(E^K/\bQ)=-1$, and therefore $w(E^K/\bQ)=+1$. By work of Kolyvagin \cite{kolyvagin-mw-sha} (or alternatively, Kato \cite{kato-euler-systems}), the non-vanishing of $L(E^K,1)$ implies that ${\rm Sel}_{p^\infty}(E^K/\Q)$ is finite and therefore
\[
{\rm corank}_{\Z_p}{\rm Sel}_{p^\infty}(E/K) = 1,
\]
and hence $E$ satisfies \eqref{eq:intro-sel}.  In particular, all the hypotheses of Corollary~\ref{cor:PR} hold. 
As in the proof of Corollary~\ref{cor:howard}, the condition that ${\rm corank}_{\Z_p}{\rm Sel}_{p^\infty}(E/K) = 1$ easily implies that 
$\ord_{\mathfrak{P}_0}(X_{\rm tors}) = 0$, and so it then follows from Corollary \ref{cor:PR} that the image of $\kappa_1$ in 
$\rH^1(K,T_pE)$ is non-zero. This implies that the Heegner point $P_K\in E(K)$ is non-torsion and hence, by the Gross--Zagier formula that $\ord_{s=1}L(E/K,s) = 1$. Since $L(E/K,s)= L(E,s)L(E^K,s)$ and $\ord_{s=1}L(E^K,s) = 0$ by the choice of $K$,
it follows that $\ord_{s=1}L(E,s) = 1$.

We now assume ${\rm corank}_{\Z_p}{\rm Sel}_{p^\infty}(E/\Q) = 0$. The result of Monsky used above implies in this case that $w(E/\Q)=+1$. We now choose an imaginary quadratic field $K$ satisfying the same conditions (a), (b), (c) as above, in addition to the condition
\begin{itemize}
\item[(d')] $\ord_{s=1}L(E^K,1)=1$.
\end{itemize} 
The existence of infinitely many such $K$ follows from \cite[Thm.~B.2]{twist}, since any $K$ satisfying (b) is such that $w(E^K/\Q)=-1$. The Gross--Zagier--Kolyvagin theorem 
implies that ${\rm corank}_{\Z_p}{\rm Sel}_{p^\infty}(E^K/\Q) = 1$ and therefore 
\[
{\rm corank}_{\Z_p}{\rm Sel}_{p^\infty}(E/K) = 1.
\]
Thus, as above, $E$ satisfies \eqref{eq:intro-sel}, and we can apply Corollary~\ref{cor:PR} and the Gross--Zagier formula to obtain $\ord_{s=1}L(E/K,s) = 1$, which implies by our choice of $K$ that $L(E,1)\neq 0$. 
\end{proof}

Since the hypotheses of Theorem \ref{thm:thmB} imply $E(\Q)[p]= 0$, we see that ${\rm Sel}_{p^\infty}(E/\Q)[p]={\rm Sel}_p(E/\Q)$, whence the following mod $p$ version of the theorem.

\begin{cor}\label{cor:thmB} Suppose $E$ is as in Theorem \ref{thm:thmB} and $r\in\{0,1\}$. Then 
\[
\dim_{\mathbb{F}_p}{\rm Sel}_p(E/\Q) = r 
\quad\Longrightarrow\quad
{\rm ord}_{s=1}L(E,s)=r,
\]
and so ${\rm rank}_\Z E(\Q)=r$ and $\#\sha(E/\Q)<\infty$.
\end{cor}

For $p=3$, Corollary~\ref{cor:thmB} together with the work of Bhargava--Klagsbrun--Lemke Oliver--Shnidman \cite{BKLS} on the average $3$-Selmer rank in quadratic twist families, leads to the following result in the direction of Goldfeld's conjecture \cite{goldfeld}. 

\begin{cor}\label{cor:goldfeld}
Let $E$ be an elliptic curve over $\bQ$ with a rational $3$-isogeny. Then 
a positive proportion of quadratic twists of $E$ have algebraic and analytic rank equal to $1$
and a positive proportion of quadratic twists of $E$ have algebraic and analytic rank equal to $0$.
\end{cor}

\begin{proof}
Denote by $\phi:G_\bQ\rightarrow\mathbb{F}_3^\times=\mu_2$ the character giving the Galois action on the kernel of a rational $3$-isogeny of $E$. As the condition $\phi\vert_{G_p}\neq\mathds{1},\omega$ can be arranged by a quadratic twist, combining \cite[Thm.~2.6]{BKLS} and Corollary~\ref{cor:thmB}, the result follows.
\end{proof}

\begin{remark}\label{rem:BKLS}
The qualitative result of Corollary~\ref{cor:goldfeld} was first obtained by Kriz--Li (see \cite[Thm.~1.5]{kriz-li}), but thanks to \cite{BKLS} (see esp. [\emph{op.\,cit.}, p.~2957]) our result {can lead} to better lower bounds on the proportion of rank~1 twists. In particular the proportions provided by \cite[Thm.~2.5]{BKLS} are the largest when the parity of the logarithmic Selmer ratios is equidistributed in quadratic families. The elliptic curve of smallest conductor over $\Q$ for which this happens is the elliptic curve having Cremona label $19a3$ given by the affine equation $y^2+y=x^3+x^2+x$. The explicit bounds of \cite{BKLS} and our result give that at least $41.6\%$ of its quadratic twists have analytic and algebraic rank equal to $1$ and at least $25\%$ have analytic and algebraic rank equal to $0$.
\end{remark}

\subsection{Proof of the \texorpdfstring{$p$}{p}-part of BSD formula} 

The following is Theorem~\ref{thm:F} in the introduction.

\begin{theorem}\label{thm:thmC}
Let $E/\bQ$ be an elliptic curve, and let $p>2$ be a prime of good ordinary reduction for $E$. Assume that $E$ admits a cyclic $p$-isogeny with kernel $C=\mathbb{F}_p(\phi)$, with $\phi:G_\bQ\rightarrow\mathbb{F}_p^\times$ such that
\begin{itemize}
\item $\phi\vert_{G_p}\neq\mathds{1},\omega$,
\item $\phi$ is either ramified at $p$ and odd, or unramified at $p$ and even.
\end{itemize}
If ${\rm ord}_{s=1}L(E,s)=1$, then 
\[
\ord_{p}\biggl(\frac{L'(E, 1)}{\operatorname{Reg}(E / \Q) \cdot \Omega_{E}}\biggr)=\ord_{p}\biggl(\# \sha(E / \Q) \prod_{\ell \nmid \infty} c_{\ell}(E / \Q)\biggr).
\]
In other words, the $p$-part of the Birch--Swinnerton-Dyer formula for $E$ holds.
\end{theorem}

\begin{proof}
Suppose ${\rm ord}_{s=1}L(E,s)=1$ and choose, as in the proof of Theorem~\ref{thm:thmB}, an imaginary quadratic field $K$ of discriminant $D_K$ such that
\begin{itemize}
\item[(a)] $D_K<-4$ is odd,
\item[(b)] every prime $\ell$ dividing $N$ splits in $K$,
\item[(c)] $p$ splits in $K$, say $p=v\bar{v}$,
\item[(d)] $L(E^K,1)\neq 0$.
\end{itemize}
Then ${\rm ord}_{s=1}L(E/K,s)=1$, which by Theorem~\ref{thmGZ} implies that the Heegner point $P_K\in E(K)$ has infinite order, and therefore ${\rm rank}_\Z E(K)=1$ and $\#\sha(E/K)<\infty$ by \cite{kolyvagin-mw-sha}. 
In particular, \eqref{eq:intro-sel} holds, and so all the hypotheses of Theorem \ref{thm:thmA} are satisfied. Thus there is a $p$-adic unit $u\in(\Z_p^{{\rm ur}})^\times$ for which
\begin{equation}\label{eq:thmA}
\mathcal{F}_E(0)=u\cdot\mathcal{L}_E(0),
\end{equation}
where $\mathcal{F}_E\in\Lambda$ is a generator of ${\rm char}_\Lambda(\mathfrak{X}_E)$. The hypotheses on $\phi$ imply that $E(K)[p]=0$, and so Theorem~\ref{controlthm} applies with $P=P_K$, which combined with Theorem~\ref{thmpadicGZ} and the relations $(\ref{eq:comparison})$ and $(\ref{eq:thmA})$ 
yields the equality
\begin{equation}\label{eq:thmA-bis}
{\rm ord}_p\bigl(\#\sha(E/K)\bigr)=2\;{\rm ord}_p\bigl(c_E^{-1}u_K^{-1}\cdot[E(K):\Z.P_K]\bigr)-\sum_{w\in S}{\rm ord}_p\bigl(c_w(E/K)\bigr),
\end{equation}
 
On the other hand, the Gross--Zagier formula of Theorem~\ref{thmGZ} can be rewritten (see \cite[p.~312]{grosszagier}) as
\[
L'(E/K,1)=2^tc_E^{-2}u_K^{-2}\cdot\hat{h}(P_K)\cdot
\Omega_E\cdot\Omega_{E^K},
\]
where the power of $2$ is given by the number of connected components $[E(\mathbb{R}):E(\mathbb{R})^0]$.
This, together with the relations $L(E/K,s)=L(E,s)\cdot L(E^K,s)$ and
\[
\hat{h}(P_K)=[E(K):\Z\cdot P_K]^2\cdot{\rm Reg}_{}(E/K)=[E(K):\Z\cdot P_K]^2\cdot{\rm Reg}_{}(E/\Q),
\]
using that ${\rm rank}_\Z E^K(\Q)=0$ for the last equality, amounts to the formula
\begin{equation}\label{eq:GZ}
\frac{L'(E,1)}{{\rm Reg}(E/\bQ)\cdot\Omega_E}\cdot\frac{L(E^K,1)}{\Omega_{E^K}}=2^tc_E^{-2}u_K^{-2}\cdot[E(K):\Z\cdot P_K]^2.
\end{equation}
Note that $u_K=1$, since $D_K<-4$. Since $\sha(E/K)[p^\infty] \simeq\sha(E/\bQ)[p^\infty]\oplus\sha(E^K/\bQ)[p^\infty]$ as $p$ is odd, and   
\[
\sum_{w\vert\ell}{\ord}_p(c_w(E/K))={\rm ord}_p(c_\ell(E/\bQ))+{\rm ord}_p(c_\ell(E^K/\Q))
\]
for any prime $\ell$ (see \cite[Cor.~9.2]{skinner-zhang}), combining $(\ref{eq:thmA-bis})$ and (\ref{eq:GZ}) we arrive at
\begin{equation}\label{eq:thmA-bisbis}
\begin{aligned}
{\rm ord}_p\biggl(\frac{L'(E,1)}{{\rm Reg}(E/\Q)\cdot\Omega_E\cdot\prod_\ell c_\ell(E/\bQ)}\biggr)&-{\rm ord}_p(\#\sha(E/\Q))\\
&=
{\rm ord}_p\biggl(\frac{L(E^K,1)}{\Omega_{E^K}\cdot\prod_\ell c_\ell(E^K/\Q)}\biggr)-{\rm ord}_p\bigl(\#\sha(E^K/\Q)\bigr).
\end{aligned}
\end{equation}

Finally, by our hypotheses on $\phi$ the curve $E^K$ satisfies the hypotheses of Theorem~\ref{thmGV}, and hence the right-hand side of $(\ref{eq:thmA-bisbis})$ vanishes, concluding the proof of Theorem~\ref{thm:thmC}.
\end{proof}

\bibliographystyle{alpha}
\bibliography{references}

\end{document}